\newif\iffinal
\finalfalse	

\documentclass[letterpaper,11pt,reqno]{amsart}
\RequirePackage[utf8]{inputenc}
\usepackage[portrait,margin=2.5cm]{geometry}
\usepackage{mathrsfs,soul}
\usepackage{hyperref}
\usepackage[foot]{amsaddr}
\usepackage{amssymb,amsthm,amsfonts,bm,latexsym,dsfont}
\usepackage[textsize=small]{todonotes}       
\usepackage{graphicx}
\usepackage[numeric,initials,nobysame]{amsrefs}
\usepackage{upref,setspace}
\usepackage{xcolor,colortbl}
\usepackage{enumerate}
\usepackage{graphicx}
\usepackage{subcaption}  
\usepackage{todonotes}
\usepackage{soul}

\usepackage{macros}
\newcommand{\chsb}[1]{#1}
\newcommand{\chmd}[1]{#1}

\begin{document}

\title[Fluctuations for JSQ($d_n$)]{Near Equilibrium Fluctuations for
Supermarket Models with Growing Choices}

\date{}
\subjclass[2010]{Primary: 90B15, 60F17, 90B22, 60C05. }
\keywords{power of choice, join-the-shortest-queue, fluid limits, heavy traffic, Halfin-Whitt, load balancing, diffusion approximations, Skorohod problem, reflected diffusions, functional limit theorems}

\author[Bhamidi]{Shankar Bhamidi}
\author[Budhiraja]{Amarjit Budhiraja}
\author[Dewaskar]{Miheer Dewaskar}
\address{Department of Statistics and Operations Research, 304 Hanes Hall, University of North Carolina, Chapel Hill, NC 27599}
\email{bhamidi@email.unc.edu, budhiraja@email.unc.edu, miheer@live.unc.edu}

\maketitle

\begin{abstract}
   We consider the supermarket model in the usual Markovian setting where jobs arrive at rate $n \lambda_n$ for some $\lambda_n > 0$, with $n$ parallel servers  each processing jobs in its queue at rate 1. An arriving job joins the shortest  among $d_n \le n$ randomly selected service queues.
   We show that when $d_n \to \infty$ and $\lam_n \to \lam \in (0, \infty)$, under natural conditions on the initial queues, the state occupancy process converges in probability, in a suitable path space, to the unique solution of an infinite system of constrained ordinary differential equations parametrized by $\lambda$. Our main interest is in the study of fluctuations of the state process about its near equilibrium state in the critical regime, namely when $\lam_n \to 1$. Previous papers e.g. \cite{mukherjee2018universality} have considered the regime $\frac{d_n}{\sqrt{n}\log n} \to \infty$ while the objective of the current work is to develop diffusion approximations for the state occupancy process that allow for all possible rates of growth of $d_n$. In particular we consider the three canonical regimes
(a) ${d_n}/{\sqrt{n}} \to 0$;
(b) ${d_n}/{\sqrt{n}} \to c\in (0,\infty)$ and,
(c) ${d_n}/{\sqrt{n}} \to \infty$.
In all three regimes we show, by establishing suitable functional limit theorems, that (under conditions on $\lam_n$) fluctuations of the state process about its near equilibrium are of order $n^{-1/2}$
 and are governed asymptotically by a one dimensional Brownian motion. The forms of the limit processes in the three regimes are quite different; in the first case we get a linear diffusion; in the second case we get a diffusion with an exponential drift; and in the third case we obtain a reflected diffusion in a half space. In the special case
 ${d_n}/({\sqrt{n}\log n}) \to \infty$ our work gives alternative proofs for the   universality results established in \cite{mukherjee2018universality}. 
\end{abstract}

\section{Introduction}
\label{sec:intro}
\setstcolor{red}
\setul{}{1.5pt}
In this work we study the asymptotic behavior of a family of randomized load balancing schemes for many server systems.  Consider a processing system with $n$ parallel queues in which each queue's jobs are processed by the associated server at rate $1$. Jobs arrive at rate 
$n\lambda_n$ and join the shortest queue amongst $d_n$ randomly selected queues (without replacement), with $d_n \in [n]:=\set{1,\ldots, n}$. The interarrival times and service times are mutually independent exponential random variables. This queuing system with the above described `join-the-shortest-queue amongst chosen queues' discipline is often denoted as
$JSQ(d_n)$ and  frequently referred to as the supermarket model (cf. \cite{graham2000chaoticity,luczak2005strong,luczak2006maximum,luczak2007asymptotic,martin1999fast,mukherjee2018universality} and references therein). Note that when $d_n=n$ the above description
corresponds to a policy where an incoming job joins the shortest of all queues in the system (see e.g. \cite{eschenfeldt2018join}). The case   $d_n=1$ is the other extreme corresponding to incoming jobs joining a randomly chosen queue in which case the system is  equivalent to one with $n$ independent $M/M/1$ queues with arrival rate $\lambda_n$ and service rate 
$1$. The case $d_n=d$ where $d>1$ is a fixed positive integer is sometimes also referred to as the power-of-$d$ scheme.
The analysis of $JSQ(d_n)$ schemes has been a focus of much recent research motivated by problems from large scale service centers, cloud computing platforms, and data storage and retrieval systems (see. e.g. \cite{cardellini2002state,van2018scalable-big-survey,maguluri2012stochastic,ongaro2011fast,mukhopadhyay2015analysis,gupta2007analysis,altman2011load,bramson2012asymptotic}). The influential works of Mitzenmacher \cite{mitzenmacher2001power,MR1966907} and Vvedenskaya et al.~\cite{vvedenskaya1996queueing} showed by considering a fluid scaling that increasing $d$ from $1$ to $2$ leads to significant improvement in performance in terms of steady-state queue length distributions in that the tails of the asymptotic steady-state distributions decay exponentially when $d=1$ and super-exponentially when $d=2$. 
Limit theorems under a diffusion scaling for the $JSQ(d)$ system, with a fixed $d$, can be found in \cite{eschenfeldt2016supermarket, budhiraja2019diffusion}.
Although $JSQ(d)$ for a fixed $d\ge 2$ leads to significant improvements over $JSQ(1)$, as observed in \cite{gamarnik2016delay,gamarnik2018delay}, no fixed value of $d$ provides the optimal waiting time properties of the join-the-shortest-queue system (i.e. $JSQ(n)$). See the survey \cite{van2018scalable-big-survey} for an overview of the progress in this general area. This motivates the study of asymptotic behavior of a $JSQ(d)$ system in which the number of choices $d$ increase with system size, namely $n$. Such an asymptotic study is the goal of this work.

The paper \cite{mukherjee2018universality} studied the law of large numbers (LLN) behavior of a $JSQ(d_n)$ system, under a standard 
scaling, when $d_n \to \infty$. The precise result of \cite{mukherjee2018universality} is as follows. 
For  $i \in \NN_0:=\set{0,1,2,\ldots}$ and $t \in [0,\infty)$, let $G_{n,i}(t)$ denote the fraction of queues with at least $i$ customers at time $t$ in the $n$-th  system.
 Note that $G_{n,0}(t) = 1$ for all $t\geq 0$. We will call $ \Gn(t) \doteq \set{G_{n,i}(t): i> 0}$  the state occupancy process. This process has
 sample paths in the space of summable nonnegative sequences. More precisely, for $p\ge 1$, let $\lSpace{p}$ be the space of real sequences $\bvec{x}:=(x_1, x_2, \ldots)$
 such that $\|\bvec{x}\|_p \doteq \left(\sum_{i=1}^{\infty}|x_i|^p\right)^{1/p}<\infty$. Let
\begin{equation}
	\label{eqn:l1-down-def}
\ld \doteq \set{\bvec{x}\in \lSpace{1}: x_i \ge x_{i+1} \mbox{ and } x_i \in [0,1] \mbox{ for all } i \in \NN}
\end{equation}
be the space of non-increasing sequences in $\lSpace{1}$ with values in $[0,1]$, equipped with the topology generated by $\norm{\cdot}_1$. 
Note that $\ld$ is a closed subset of $\lSpace{1}$ and hence is a Polish space. 
Then, whenever $\|\Gn(0)\|_1 <\infty$ a.s., it can be shown that $\{\Gn(t):t\geq 0\}$ is a stochastic process with sample paths in 
$\DD([0,\infty): \ld)$ (the space of right continuous functions with left limits from $[0,\infty)$ to $\ld$ equipped with the usual Skorohod topology); see Section \ref{sec:basic-def}.
The paper \cite{mukherjee2018universality}  shows the following two facts under the assumption that $\Gn(0)$ converges in probability to
some $\bvec{r} \in \ld$: 
\begin{enumeratea}
	\item When $d_n=n$ and $\lambda_n \to \lambda \in (0,\infty)$,
$\Gn$ is a tight sequence in $\DD([0,\infty): \ld)$ and every weak limit point satisfies a certain set of ``fluid limit equations" 
(see \cite[Theorem 5]{mukherjee2018universality}, and equations \eqref{eq:limit-ode-explicit}-\eqref{eq:limit-ode-explicit2} in the current work);
\item When $d_n$ is an arbitrary sequence growing to $\infty$ and $\lambda_n \to \lambda \in (0,1)$, then 
the  statements in (a) hold once more for $\Gn$.
\end{enumeratea} 
   The current work begins by revisiting the above LLN results from 
\cite{mukherjee2018universality}. In Theorem \ref{thm:fluid-limit} of this work, we show that, when $\Gn(0)\pconv \bvec{r}$,  for  arbitrary sequences $d_n \to \infty$
and $\lambda_n\to \lambda \in (0, \infty)$,
$\Gn$ converges in probability in $\DD([0,\infty): \ld)$ to a continuous trajectory $g$ in $\ld$ that is characterized as the
{\em unique} solution of an infinite system of constrained ordinary differential equations (ODE) (see \eqref{eq:limit-ode2} in Proposition \ref{lem:uniqsoln}).
Using standard properties of the Skorohod map we observe in Remark \ref{rem:sameasdeb} that a continuous trajectory in $\ld$ solves the fluid limit equations of \cite{mukherjee2018universality} if and only if it solves \eqref{eq:limit-ode2}. This together with Proposition \ref{lem:uniqsoln} proves that the fluid limit equations in \cite{mukherjee2018universality} in fact have a unique solution. In this manner we complete and strengthen the result from \cite{mukherjee2018universality}. Our proof of the LLN result is quite different from the arguments in \cite{mukherjee2018universality}. The latter are based on sophisticated ideas of separation of time scales and weak convergence of measure valued processes from \cite{hunt1994large} to handle the convergence for $d_n=n$, and  certain coupling techniques to treat the general case when $d_n<n$ and $d_n\to \infty$. In contrast, our approach is more direct and uses martingale estimates and well known characterization properties of solutions of Skorohod problems (see e.g. proof of Lemma \ref{lem:flexistence}).

Our main goal in this work is to study diffusion approximations for $\Gn$ in the heavy traffic regime, namely when $\lam_n\to 1$.
In the case when $d_n=n$ ($JSQ(n)$ system), this problem has been studied in \cite{eschenfeldt2016supermarket}. Their basic result is as follows. Suppose $d_n=n$ and $\sqrt{n}(\lam_n-1) \to \beta >0$. Consider the unit vector $\bvec{e}_1 = (1, 0, \ldots)$ in $\lSpace{2}$. Then under conditions on $\Gn(0)$, $\Yn(\cdot) \doteq \sqrt{n}(\Gn(\cdot)-\bvec{e}_1)$ converges
in distribution in $\DD([0,\infty): \lSpace{2})$ to a continuous stochastic process $\Y = (Y_1, Y_2, \ldots)$, described in terms of a one dimensional Brownian motion,  for which $Y_i=0$ for $i > r$ for some $r \in \NN$ (which depends on the conditions assumed on $\Gn(0)$). Specifically, when $r=2$, the pair $Y_1,Y_2$
is given as a two dimensional diffusion in the half space $(-\infty , 0]\times \RR$ with oblique reflection in the direction $(-1, 1)'$ at the boundary $\{0\}\times \RR$. (For the form of the limit in the general case see Corollary \ref{cor:reflection-at-1}).
In \cite{mukherjee2018universality} this result is extended to the case where $d_n<n$ and  $\frac{d_n}{\sqrt{n}\log n} \to \infty$.
Under the same assumptions on the initial condition as in \cite{eschenfeldt2016supermarket}, it is shown in \cite{mukherjee2018universality}  that $\Yn$ converges to the same limit process as for the case $d_n=n$. The proof, as for the LLN result, proceeded by constructing a suitable coupling between a $JSQ(d_n)$ and $JSQ(n)$ system. The paper \cite{mukherjee2018universality}  also argued that
when $\frac{d_n}{\sqrt{n}\log n} \to 0$, the process $\Yn$ cannot be tight and thus in this regime the above diffusion approximation cannot hold.

Our objective in this work is to develop diffusion approximations for $\Gn$ in the critical regime (i.e. when $\lam_n \to 1$ in a suitable manner) that allow for possibly a slower growth of $d_n$ than that permitted by the results in \cite{mukherjee2018universality}.
In fact the results we establish will allow for $d_n \to \infty$ in an arbitrary manner and will recover the results of \cite{mukherjee2018universality} in the special case $\frac{d_n}{\sqrt{n}\log n} \to \infty$ (with \chmd{a} different proof).
In order to motivate the type of limit theorems we seek, we begin by observing that the centering $\bvec{e}_1$ used in the definition of $\Yn$ is a stationary point of the fluid limit given in \eqref{eq:limit-ode2} with $\lam=1$ and thus the results of \cite{eschenfeldt2016supermarket} and \cite{mukherjee2018universality} give information on fluctuations of the state process $\Gn$ about this stationary point.
However $\bvec{e}_1$ is not the only stationary point of \eqref{eq:limit-ode2} (when $\lam=1$) and in fact this ODE has uncountably many fixed points with a typical such point given as $\bvec{f}_k^{\gamma}\doteq \sum_{j=1}^k \bvec{e}_j+ \gamma \bvec{e}_{k+1}$, where 
$\bvec{e}_j$ is the $j$-th unit vector in $\lSpace{2}$ (with $1$ at the $j$-th coordinate and zeroes elsewhere), $k \in \NN$ and $\gamma \in [0, 1)$. All of these  stationary points arise in a natural fashion. Indeed, it turns out that the evolution of the state process $\Gn$ can be described via  the equation (see Remark \ref{rem:vect-equation}) 
\begin{equation*}
    \Gn(t) = \Gn(0) + \int_0^t \brS*{\bvec{a_n}(\Gn(s)) - \bvec{b}(\Gn(s))} ds + \Mn(t),
\end{equation*}
where $\Mn$ is a (infinite dimensional) martingale converging to $0$ in probability (see Lemma \ref{lem:martisnull}) and
 $\bvec{a_n},\bvec{b}$ are certain maps from  $\ld$ to $\lSpace{1}$ (see Remark \ref{rem:vect-equation} for details).
 Thus for large $n$, trajectories of $\Gn$ will be close to solutions of the infinite dimensional ODE
 $$\dot \gnn = \bvec{a_n}(\gnn) - \bvec{b}(\gnn).$$
 This equation has a  unique  stationary point $\Mu_n$ which is introduced in Definition \ref{def:fixed-point}. The fixed point $\Mu_n$ corresponds to the point in the state space $\ld$ at which the inflow rate equals the outflow rate in the $n$-th system and thus it is of interest to explore system behavior in the neighborhood of this point.
 Since $\Gn$ is approximated by $\gnn$ (over any compact time interval), one can loosely interpret $\Mu_n$ as a {\em near fixed point} of the state process $\Gn$. Furthermore, it can be shown (see Remark \ref{rem:onthm1}(\ref{item:cond-thm1})) that, if $d_n\to \infty$ and $\lam_n \to 1$ in a suitable manner,  $\Mu_n$ can converge to any specified fixed point $\bvec{f}_k^{\gamma}$ of \eqref{eq:limit-ode2} and thus every fixed point of \eqref{eq:limit-ode2} arises from $\Mu_n$ in a suitable asymptotic regime. In order to explore fluctuations of $\Gn$ close to different fixed points of \eqref{eq:limit-ode2} it is then natural to study the asymptotic behavior of
 \begin{equation}
 	\label{eq:zn-proc-def}
 	\Zn(t) \doteq \sqrt{n} \brR*{\Gn(t) - \Mun}, \qquad t\geq 0.
 \end{equation}
 We note that in the regime considered in \cite{mukherjee2018universality} where $\frac{d_n}{\sqrt{n}\log n} \to \infty$ and $\sqrt{n}(1-\lam_n) \to \alpha>0$, $\sqrt{n} \brR*{\bvec{e_1} - \Mun } \to \alpha$ and so in this case the asymptotic behavior of $\Zn$
 can be read off from that of $\Yn$ (see Corollary \ref{cor:reflection-at-1} and Remark \ref{rem:dn-large}(v)). However in general $\sqrt{n} \brR*{\bvec{e_1} - \Mun }$ (and more generally,
 $\sqrt{n} \brR*{\bvec{f}_k^{\gamma} - \Mun }$)  may not be bounded and so the asymptotic behavior of $\Zn$ and $\Yn$ may be very different.
 
 In this work we obtain limit theorems for $\Zn$ as $d_n \to \infty$ in an arbitrary fashion and $\lam_n \to 1$ in a suitable manner. Specifically in Theorems \ref{thm:diffusion-limit}, \ref{thm:exponential} and \ref{thm:reflection} we consider the three cases: \\
 \begin{inparaenuma}
 	\item ${d_n}/{\sqrt{n}} \to 0$;
	\item ${d_n}/{\sqrt{n}} \to c\in (0,\infty)$ and,
	\item ${d_n}/{\sqrt{n}} \to \infty$, respectively.
 \end{inparaenuma}
 In all three regimes we consider initial conditions $\Gn(0)$ such that for some $r \in \NN$, $G_{n,m}(0) = \mu_{n,m} + o_p(n^{-1/2})$ for all $m>r$ and in each case (under conditions on $\lam_n$) we obtain a limit process
 driven by a one dimensional Brownian motion with continuous sample paths in $\lSpace{2}$ which has all but finitely many coordinates $0$.
 In particular, when $r=2$ in the second and the third case and $r= k+2$ for some $k \in \NN$ in the first case (and $d_n$, $\lam_n$ depend on $k$ in a suitable fashion), one can describe the limit through a two dimensional diffusion driven by a one dimensional Brownian motion. The form of this two dimensional process in the three regimes is quite different; in the first case we get a linear diffusion (i.e. the drift is of the form $b(y)=Ay$ for, $y \in \RR^2$ and some $2\times 2$ matrix $A$); in the second case we get a diffusion with an exponential drift; and in the third case we obtain a reflected diffusion in the half space $(-\infty, \alpha]\times \RR$ for some $\alpha \ge 0$. 
 
 Although the limit processes in Theorems \ref{thm:diffusion-limit} and \ref{thm:exponential} are quite different from those obtained in \cite{eschenfeldt2016supermarket} and \cite{mukherjee2018universality}, the limit  in Theorem \ref{thm:reflection} has a similar form  (in that it is a reflected diffusion in a half space) as in the above papers. However here as well there are some differences.
 In particular, depending on how $\lam_n$ approaches $1$, the reflection occurs at a different barrier $\alpha \in (0, \infty)$; in fact $\alpha = \infty$ is possible as well in which case there is no reflection. Furthermore, recall that $\Zn$  is defined by centering about $\Mu_n$. In general $\sqrt{n}(\Mu_n - \bvec{e}_1)$ will diverge and thus the process $\Yn$ considered in the above cited papers may not converge in this regime. However, as noted previously, when $d_n$ grows sufficiently fast, namely $\frac{d_n}{\sqrt{n}\log n} \to \infty$  the process $\Yn$ will indeed converge and in that case we recover the result in \cite{mukherjee2018universality} (in fact a slight strengthening in that the drift parameter in Corollary \ref{cor:reflection-at-1} is allowed to be $0$).
In addition Theorem \ref{thm:reflection} also covers the case $\frac{d_n}{\sqrt{n}\log n} \to c \in (0,\infty)$ and situations where
$\lam_n = 1 + O(n^{-1/2})$ (see Remark \ref{rem:dn-large} (\ref{item:dn-boundry})). In such settings, once more both $\Zn$ and $\Yn$   converge and the limit of the latter has the same form as
in \cite{eschenfeldt2016supermarket, mukherjee2018universality}.

As is observed in Remarks \ref{rem:onthm2} and \ref{rem:dn-large}, under conditions of Theorem \ref{thm:exponential} or
Theorem \ref{thm:reflection}, $\Mu_n$ must converge to the fixed point $\bvec{e}_1 = \bvec{f}_1^{0}$. In contrast, Theorem \ref{thm:diffusion-limit} allows for a range of asymptotic behavior for $\Mu_n$. In particular, under the conditions
of the theorem, with suitable $\lam_n, d_n$, $\Mu_n$ can converge to the fixed point $\bvec{f}_k^0$ for an arbitrary $k \in \NN$ (see \cite{brightwell2018supermarket} for a similar observation). In such a setting the first $k-1$ coordinates of the limit process are essentially $0$ (see Theorem \ref{thm:diffusion-limit} for a precise statement) and the $k$-th coordinate is the first one to exhibit stochastic variability. Thus a rather novel asymptotic behavior for the $JSQ(d_n)$ system emerges when $d_n$ approaches $\infty$ at significantly slower rates than those considered in \cite{mukherjee2018universality} and $\lambda_n$ 
approach $1$ in a suitable manner (in relation to $d_n$).

We now  make some comments on the proofs of Theorems \ref{thm:diffusion-limit} - \ref{thm:reflection}. The starting point is a convenient semimartingale representation for the centered state process $\Zn$ in \eqref{eq:scaled-marteq}. In the study of the behavior of the drift term in this decomposition, an important ingredient is an analysis of the asymptotic properties  of the near fixed point $\Mu_n$, and the asymptotic behavior of the function $\beta_n$ (see Definition \ref{def:betan}) in $O(n^{-\frac{1}{2}})$ sized neighborhoods around the coordinates of $\Mu_n$. This behavior, which is different in the three regimes considered above, determines the asymptotics of the drift $\bvec{A_n}(\Zn(s)) - \bvec{b}(\Zn(s))$. Properties of $\Mu_n$ are also key in arguing that,  in all three cases, under our conditions, $(Z_{n, r+1}, \cdots )$  converges to $0$ in probability in $\DD([0,\infty): \lSpace{2})$ (see Lemma \ref{lem:finitedimensional}). The rest of the work is in characterizing the asymptotics of the finite dimensional process $(Z_{n,1}, \ldots , Z_{n,r})$. For this study, the three regimes require different approaches. In particular, Theorem \ref{thm:diffusion-limit} hinges on a detailed understanding of the asymptotic behavior of a tridiagonal matrix function $A_n(s)$ (see e.g. Lemmas \ref{lem:large-neg-eigenval} and \ref{lem:headconv0}); Theorem \ref{thm:exponential} requires an analysis of a stochastic differential equation with an exponential drift term (in particular the drift does not satisfy the usual growth conditions); and Theorem
\ref{thm:reflection} is based on a careful study of excursions of the prelimit processes above the limiting reflecting barrier and properties of Skorohod maps in order to characterize the reflection properties of the limit process.

\subsection{Organization of the paper}

 Section \ref{sec:mainresults_ab} contains all our main results. 
 The remaining Sections starting with Section \ref{sec:basic-def} contain proofs of the main results.   
\subsection{Notation and setup}
\label{subsec:notat}
For $m\geq 1$, let $[m] \doteq \set{1,2,\ldots, m}$. We will denote finite-dimensional vectors in $\R^m$ as $\vec{x}, \vec{y}$, etc. and $\inp{\vec{x}}{\vec{y}}$ will denote the standard inner-product. The standard basis vectors in $\R^m$ will be denoted by $\vec{e_i}$ for $i = 1, 2 \ldots m$. Also, $\norm{\vec{x}} \doteq \inp{\vec{x}}{\vec{x}}$ will denote the usual Euclidean norm.

We will often use bold symbols such as $\bvec{x}:=(x_1, x_2, \ldots)$ to denote a infinite dimensional vector or function. For $p\in \set{1,2,\ldots \infty}$, let $\norm{\cdot}_p$ denote the usual $p$-th norm on the space of infinite sequences and $\lSpace{p} \doteq \set{\bvec{x} \in \R^{\infty} \mid \norm{\bvec{x}}_p < \infty}$. Let \chsb{$\ld$ be as in \eqref{eqn:l1-down-def}}, which is a Polish space under $\norm{\cdot}_1$. For $k \in \nat$, let $\fk = (1, 1, \ldots, 1, 0, 0\ldots) \in \ld$ denote the vector with first $k$ indices equal to $1$, and $\ek = (0, \ldots, 0, 1, 0 \ldots) \in \lSpace{1}$ denote the vector with $1$ in the $k$th coordinate. Finally, for $\bvec{z}=(z_1, z_2, \ldots) \in \R^{\infty}$ and $r \in \nat$, let $\bvec{z_{r+}} \doteq (z_{r+1}, z_{r+2}, \ldots) \in \R^{\infty}$ denote the vector shifted by $r$ steps. Similar notation will be used for functions and processes with values in $\R^{\infty}$.

For a Polish space $\mathbb{S}$ and $T>0$, denote by $\mathbb{C}([0,T]:\mathbb{S})$
(resp.\ $\mathbb{D}([0,T]:\mathbb{S})$) the space of continuous functions
(resp.\ right continuous functions with left limits) from $[0,T]$ to $%
\mathbb{S}$, endowed with the uniform topology (resp.\; Skorokhod topology). 
\cg{Spaces $\mathbb{C}([0,\infty):\mathbb{S}),
\mathbb{D}([0,\infty):\mathbb{S})$ are defined similarly.}
For $f \in \DSpace{T}[\R]$ and $t\leq T$, let $\abs{f}_{*,t} \doteq \sup_{s \in [0,t]} \abs{f(s)}$. Similarly for $\bvec{g} \in \DSpace{T}[\lSpace{p}]$, let $\norm{\bvec{g}}_{p, t} \doteq \sup_{s \in [0,t]} \norm{\bvec{g}(s)}_{p}$.

We will use $\I{cond}$ to denote the indicator function that takes the value 1 if \textit{cond} is true, otherwise it takes the value 0. 
\cg{We will denote by $\id$ the identity map, $\id(t)=t$, on $[0,T]$  or $[0,\infty)$.}

\chsb{We use} $\prob$ and $\E$ to denote the probability and expectation operators, respectively. For $x,y \in \R$, $x \wedge y$ denotes the minimum and $x \vee y$ the maximum of $x$ and $y$ respectively. For any $x \in \R$, $x^+ = x \vee 0$ and $x^- = (-x) \vee 0$. We use $\pconv$ and $\Rightarrow$ to denote convergence in probability and convergence \cg{in distribution} respectively on an appropriate Polish space which will depend on the context. For a sequence of \chsb{real valued} random variables $\big(X_n,\ n\geq 1\big)$, we write $X_n=\op(b_n)$ when $|X_n|/b_n\probc 0$ as $n\rightarrow\infty$.
For non-negative functions $f(\cdot), g(\cdot)$, we write $f(n)=O(g(n))$ when $f(n)/g(n)$ is uniformly bounded, and $f(n)=o(g(n))$ (or $f(n)\ll g(n)$) when $\lim_{n\rightarrow \infty} f(n)/g(n)=0$. We write $f(n) \sim g(n)$ if $f(n)/g(n)\to 1$ as $n\to\infty$.
\cg{We will use the notation $\lam_n \nearrow 1$ to mean that $\lam_n <1$ for every $n$ and $\lam_n \to 1$ as $n\to \infty.$}

\section{Main Results}
\label{sec:mainresults_ab}

Recall the process $\Gn$ from Section \ref{sec:intro}. Our first result gives a law of large numbers (LLN) for the process $\Gn$ as $n\to \infty$.
In order to state this result we begin by recalling the one dimensional Skorohod map (cf. \cite[Section 3.6.C]{karshr}) with a reflecting barrier at $\alpha \in \RR$.
For $\alpha \in \RR$ and $f \in \DSpace*{\infty}[\R]$ with $f(0) \leq \alpha$, define $\Gamma_{\alpha}(f) , \hat \Gamma_{\alpha}(f) \in \DSpace*{\infty}[\R]$ as 
\begin{align}
	\label{eq:SMdef}
    \Gamma_\alpha(f)(t) =  f(t) - \sup_{s \in [0,t]} (f(s)-\alpha)^+, \;\; \hat \Gamma_\alpha(f)(t) = \sup_{s \in [0,t]} (f(s)-\alpha)^+.
\end{align}	
The map $\Gamma_{\alpha}$  (and sometimes the pair $(\Gamma_{\alpha}, \hat \Gamma_{\alpha})$) is referred to as the one-dimensional Skorohod map (with reflection at $\alpha$).
The following wellposedness result,  which is proved in Section \ref{sec:proof-fluid}, will be used to characterize the LLN limit of $\Gn$.
\begin{prop}
	\label{lem:uniqsoln}
	Fix $\bvec{r} \in \ld$. Then there is a unique $(\g,\vbo)\in \CC([0,\infty): \ld \times \lSpace{\infty})$ that solves the
	following system of equations
	\begin{equation}
	  \label{eq:limit-ode2}
	  \begin{aligned}
	     g_{i}(t) &=  \Gamma_1\brR*{r_i - \int_0^\cdot \brR{g_i(s) - g_{i+1}(s)} ds  + v_{i-1}(\cdot)}(t) \quad \forall \, i \geq 1, t \ge 0\\
		 v_i(t) &= \hat \Gamma_1\brR*{r_i - \int_0^\cdot \brR{g_i(s) - g_{i+1}(s)} ds  + v_{i-1}(\cdot)}(t) \quad \forall \, i \geq 1, \;\; v_0(t) = \lambda t, \;\;t \ge 0.
	  \end{aligned} 
	\end{equation}
\end{prop}

\begin{remark}
\label{rem:reflection}
Using the well known characterization of a one-dimensional Skorohod map, one can \cg{alternatively} characterize $(\g,\vbo)$ as the unique pair in $\CC([0,\infty): \ld \times \lSpace{\infty})$ such that $v_i$ is nondecreasing,
\begin{equation}
  \label{eq:limit-ode1}
  \left. \begin{aligned}
  g_{i}(t) &= r_i - \int_0^t \brR{g_i(s) - g_{i+1}(s)} ds  + v_{i-1}(t) - v_i(t) \\
   v_{i}(t)& \geq 0 \, , g_i(t) \leq 1, \,\int_0^t (1-g_i(s)) dv_{i}(s) = 0
  \end{aligned} \right\} \forall \, i \geq 1
\end{equation}
and $v_0(t) = \lambda t$, for all $t > 0$.
\end{remark}
We can now present the LLN result. The proof is given in Section \ref{sec:proof-fluid}.

\begin{theorem}
\label{thm:fluid-limit}
Let $\bvec{r} \in \ld$. Suppose that $\Gn(0) \pconv \bvec{r}$ in $\ld$, $\lambda_n \to \lambda$ and $d_n \to \infty$, as $n \to \infty$. 
Then $\Gn \to \g$ in probability in $\DD([0,\infty): \ld)$ as $n\to \infty$, where $(\g,\vbo) \in \CC([0,\infty): \ld \times \lSpace{\infty})$
is the unique solution of \eqref{eq:limit-ode2}.
\end{theorem}

\begin{remark}
	\label{rem:sameasdeb}
Note that Theorem \ref{thm:fluid-limit} allows $d_n\to \infty$ in an arbitrary manner.  In \cite{mukherjee2018universality}*{Theorem 1} it is shown that, under the assumptions of Theorem \ref{thm:fluid-limit}, $\Gn$ is a tight sequence
  of $\DD([0,\infty): \ld)$ valued random variables and   that every subsequential weak limit $\hat{\g}$ satisfies a system of equations given as
  \begin{align}
    \label{eq:limit-ode-explicit}
    \hat g_i(t) &= r_i - \int_0^t \brR{\hat g_i(s) - \hat g_{i+1}(s)} ds + \int_0^t p_{i-1}(\hat{\bvec{g}}(s))ds \chmd{\quad \text{ for } i \geq 1}
\end{align}
        where
		\begin{align}\label{eq:limit-ode-explicit2}
    p_j(\hat{\g}(s)) & = \begin{cases}
        \lambda - (\lambda -1 + \hat g_{j + 2}(s))^+& \text{ if } j = m(\hat{\g}(s)) - 1\\
        (\lambda -1 + \hat g_{j + 1}(s))^+& \text{ if } j = m(\hat{\g}(s))  > 0 \\
           \lambda & \text{ if } j =  m(\hat{\g}(s)) = 0 \\
           0 & \text { otherwise, }
    \end{cases}
  \end{align}
  and for $\bvec{x} \in \ld$,  $m(\bvec{x}) \doteq \inf \brC{i \mid x_{i+1} < 1 }$. (Note that $m(\Gn(t))$ is the length of the smallest queue at time $t$.) 
  The uniqueness of the above system of equations was not shown in \cite{mukherjee2018universality}.
  
  From  \eqref{eq:limit-ode2}  and the definition in \eqref{eq:SMdef} it follows that each $v_i$ is absolutely continuous and, for a.e. $t$,
  \begin{equation*}
   \frac{d v_i(t)}{dt} = \brR*{\frac{d v_{i-1}(t)}{dt} - g_i(t) + g_{i+1}(t)}^{+} \I{g_i(t) = 1} \label{eq:derivative1}
  \end{equation*}
  for any $i \geq 1$. From this we see that, for a.e. $t$,
  \begin{equation}\label{eq:defnvv}
  \frac{dv_i(t)}{dt} = \begin{cases}
      \lambda & \text{ if } i = 0 \\
      \frac{dv_{i-1}(t)}{dt} & \text{ if } i < m(\g(t))  \text{ and } i \geq 1,\\
  \brR*{\frac{dv_{i-1}(t)}{dt} - 1 + g_{i+1}(t)}^{+} & \text{ if } i = m(\g(t))  \text{ and } i \geq 1,\\
  0          & \text{ if } i > m(\g(t)).
  \end{cases}
  \end{equation}
  and consequently $ p_j({\g}(s)) = \frac{dv_{j}(s)}{ds} - \frac{dv_{j+1}(s)}{ds}$  for a.e. $s$.
  Substituting this back in \eqref{eq:limit-ode1} shows that $\g$ solves the system of equations in \eqref{eq:limit-ode-explicit}.
  Conversely, for any solution $\hat{\g}$ of \eqref{eq:limit-ode-explicit}, defining $\hat{\vbo}$ by the right side of \eqref{eq:defnvv} by replacing
  $\g$ with $\hat{\g}$, we see that $(\hat{\g},\hat{\vbo})$ solves \eqref{eq:limit-ode1}.  From the uniqueness result in Lemma \ref{lem:uniqsoln} it then follows that
  in fact there is only one solution to the system of equations in \eqref{eq:limit-ode-explicit} and this solution equals $\g$ given in \eqref{eq:limit-ode2}.
\end{remark}

Consider now the time asymptotic behavior of $\g$ given in \eqref{eq:limit-ode2}.
        When $\lambda < 1$, $(\lambda, 0, 0 \ldots) \in \lSpace{1}$ is the unique fixed point of \eqref{eq:limit-ode2}, as can be seen by setting the derivative of the right side of \eqref{eq:limit-ode-explicit} to  $0$. In the critical case, i.e. when  $\lambda = 1$, the situation is very different and in fact there are uncountably many 
		fixed points given by $\{ \bvec{f} \in \ld \mid m(\bvec{f}) > 0, f_{m(\bvec{f}) + 2} = 0 \} \subset \ld$, which once more is seen  by checking that the derivative on the right side of \eqref{eq:limit-ode-explicit} is 0 at exactly these points when $\lambda=1$. 
	In this work we are  interested in the fluctuations of $\Gn$ in the critical case when the system starts suitably close to one of the fixed points of 
	\eqref{eq:limit-ode1}. Thus for the remaining section we will assume that $\lambda_n< 1$ for every $n$ and $\lambda_n \to 1$ as $n\to \infty$.
	In order to formulate precisely what we mean by `suitably close to the fixed point' we need some definitions and notation.
	The functions $\beta_n$ in the next definition will play a central role.
\begin{definition}
	\label{def:betan}
	Define the function $\beta_n : [0,1] \to [0,1]$ by
	\begin{equation}
	\label{eq:betan}
        \beta_{n}(x) \doteq \prod_{i=0}^{d_n - 1} \brR*{\frac{x - \frac{i}{n}}{1 - \frac{i}{n}}}^+ 
	\end{equation}
\end{definition}	
The function $\beta_n(\cdot)$ arises when sampling $d_n$ random servers without replacement. Specifically, when $nx \in \nat$, 
$\beta_n(x) = \prob\brR*{ \mathbb{A}_{n, d_n} \subseteq \natpre{n x}}= \binom{n x}{d_n}/\binom{n}{d_n}$, where $\mathbb{A}_{n, d_n}$ is a randomly chosen subset (without replacement) from $\natpre{n}$ of size $d_n$. An alternative is to perform sampling with replacement,  which corresponds to the simpler function $\gamma_n(x) \doteq x^{d_n}$ in place of $\beta_n$.	

We now introduce the notion of a `near fixed point' of $\Gn$.

\begin{definition}
	\label{def:fixed-point}
	For $n \in \nat$, the {\bf near fixed point} $\Mun$ of $\Gn$ is the  vector in $\ld$ given as
	$\Mun = (\mu_{n,1}, \mu_{n,2} \ldots )$ where $\mu_{n,i}$ are defined  recursively as $\mu_{n,1} = \lambda_n $ and $\mu_{n,i+1} = \lambda_n \beta_n\brR*{\mu_{n,i}}$ for $i \geq 1$.
\end{definition}
Using $\beta_n(x) \leq x^{d_n}\le x$ and $\lambda_n < 1$, it is easy to check that $\Mun \in \ld$. 
The reason $\Mun$ is referred to as a near fixed point of $\Gn$ is discussed in Remark \ref{rem:vect-equation}. To study the fluctuations of the process around the near fixed point $\Mun$ we define the centered and scaled process, $\Zn$
as  in \eqref{eq:zn-proc-def}. We now present our three main results on fluctuations which correspond to the three cases ${d_n}/{\sqrt{n}} \to 0$, ${d_n}/{\sqrt{n}} \to c\in (0,\infty)$,
and ${d_n}/{\sqrt{n}} \to \infty$ respectively.

\begin{theorem} 
\label{thm:diffusion-limit}
Suppose that, as $n \to \infty$,    $1 \ll d_n \ll \sqrt{n}$, \cg{$\lam_n\nearrow 1$}, and there is a $k \in \nat$ so that $\mu_{n,k} \to 1$ and $\beta_n'(\mu_{n,k}) \to \alpha \in [0,  \infty)$ as $n \to \infty$.  Further suppose  that $\brC*{\norm{\Zn(0)}_1}_{n \in \nat}$ is tight and that $\Zn(0) \pconv \bvec{z}$ in $\lSpace{2}$, where $\bvec{z_{r+}} = \bvec{0}$ for some $r > k$. 
Then for any  $T \in (0,\infty)$, 
\begin{equation}
    \lim_{M \to \infty}\sup_{n} \prob\brR*{\norm{\Zn}_{2,T} > M} = 0.
    \label{eq:tightness}
\end{equation}
Furthermore, if $k > 1$, then $\sup_{t \in [\epsilon, T]} \abs{Z_{n,i}(t)} \pconv 0$ as $n \to \infty$ for any $T<\infty$,   $0 < \epsilon \leq T$ and $i \in [k-1]$.

Consider the shifted process $\Yn(t) \doteq (\sum_{i=1}^k Z_{i,n}(t), Z_{k+1,n}(t), Z_{k+2,n}(t), \ldots)$ and $\bvec{y} \doteq (\sum_{i=1}^k z_i, z_{k+1}, z_{k+2}, \ldots)$. Then  $\Yn \dconv \Y$ in $\DSpace*{\infty}[\lSpace{2}]$, where $\Y \in \CSpace*{\infty}[\lSpace{2}]$ is the unique pathwise solution to
\begin{equation}
        \label{eq:diffusion-limit1}
        \begin{aligned}
            Y_1(t) &= y_1 -(\alpha + \I{k=1}) \int_0^t Y_1(s) ds + \int_0^t Y_2(s) ds + \sqrt{2}B(t) \\
            Y_2(t) &= y_2 + \alpha \int_0^t Y_1(s) ds - \int_0^t Y_2(s) ds + \int_0^t Y_3(s) ds  \\
            Y_i(t) &= y_i - \int_0^t Y_i(s)ds + \int_0^t Y_{i+1}(s) ds \qquad \text{for } i \in \{3, \ldots, r - k + 1\} \\
            Y_i(t) &= 0 \qquad \text{for } i > r - k + 1,
        \end{aligned}
\end{equation}
and \chsb{$B(\cdot)$} is a one dimensional standard Brownian motion. 
\end{theorem}
\begin{remark}\label{rem:onthm1}$\,$
	\begin{enumerate}[(i)]
		\item Note that the convergence $\sup_{t \in [\epsilon, T]} \abs{Z_{n,i}(t)} \pconv 0$ as $n \to \infty$ for any  $0 < \epsilon \leq T$ is equivalent to the statement that $Z_{n,i} \to 0$ in probability in $\DD((0,T]:\RR)$ where the latter space is equipped with the topology of uniform convergence on compacts. Note also that, since, \cg{for $i \in [k-1]$, $Z_{n,i}(0)$ may converge in general}  to a non-zero limit, the above convergence to $0$ cannot be strengthened
		to a convergence in probability in $\DD([0,T]:\RR)$.
            \item \label{item:mun-conv-fk} By Corollary \ref{cor:betaprimemun} in Section \ref{sec:technical-estimates}, when $\mu_{n,k}$ is away from $0$,
 \begin{equation*}
 \label{eq:muk1}
     \beta_n'(\mu_{n,k}) = (1+o(1)) \frac{d_n \mu_{n,k+1}}{\lambda_n \mu_{n,k}}
 \end{equation*}
 as $n \to \infty$.
 Hence the assumptions $d_n \to \infty$, $\lambda_n \to 1$, $\mu_{n,k} \to 1$ and $\beta_n'(\mu_{n, k}) \to \alpha < \infty$ in Theorem \ref{thm:diffusion-limit} say that $\mu_{n,k+1} \to 0$. Since $\mu_{n,k} \to 1$, this in fact shows that $\Mun \to \fk$ in $\ld$, where recall that $\fk$ is one of the fixed points of the fluid-limit \eqref{eq:limit-ode2} when $\lambda = 1$. The fact that the  convergence happens in $\ld$ can be seen on observing that if $\mu_{n,k+1} \leq \epsilon$ then, by \eqref{eq:gammageqbeta}, $\mu_{n, k + 1 + i} \leq \epsilon^{d_n^i}$. This convergence, along with \eqref{eq:tightness} shows that most queues will be of length $k$ on any fixed interval $[0,T]$. We also note that in general $\sqrt{n}(\Mu_n- \fk)$ will diverge, and thus $\sqrt{n}(\Gn -\fk)$ will typically not be tight, in this regime.
 
 \item In the special case when the system starts sufficiently close to  the near fixed point $\Mu_n$ so that \cg{$z_i=0$ for $i>k+1$}, the limit process $\Y$ simplifies to
 an essentially two dimensional process given as, $Y_i(t) = 0$ for $i > 2$, and
 \begin{equation*}
         \label{eq:diffusion-limit1zero}
         \begin{aligned}
             Y_1(t) &=  \cg{y_1}-(\alpha + \I{k=1}) \int_0^t Y_1(s) ds + \int_0^t Y_2(s) ds + \sqrt{2}B(t) \\
             Y_2(t) &=  \cg{y_2 +} \alpha \int_0^t Y_1(s) ds - \int_0^t Y_2(s) ds   
           \end{aligned}
 \end{equation*}
 \item \label{item:cond-thm1} The convergence behavior of $\Zn$ is governed by the sequence of parameters $(d_n, \lambda_n)$. In Corollary \ref{cor:diffusionconditions} from Section \ref{sec:technical-estimates}, we show that if $1 \ll d_n^{k+1} \ll n$ and $1 - \lambda_n = \frac{\xi_n + \log d_n}{d^{k}_n}$ with $\xi_n \to -\log(\alpha) \in (-\infty, \infty]$ and $\frac{\xi_n^2}{d_n} \to 0$, then the conditions $\mu_{n,k} \to 1$ and $\beta_n'(\mu_{n,k}) \to \alpha \in [0,\infty)$ of Theorem \ref{thm:diffusion-limit} are satisfied. Using this fact we make the following observations. For simplicity, consider $\bvec{z}=0$.
	\begin{enumerate}[(a)]
		\item Suppose that $d_n=\log n$, $1- \lambda_n = \frac{\log \log n}{(\log n)^k}$. In this case the assumptions of Theorem \ref{thm:diffusion-limit}
		are satisfied and one essentially sees non-zero fluctuations only in the $k$-th and $k+1$-th coordinates. 
		Note that as $k$ becomes large, the traffic intensity increases and one sees more and more coordinates of the near fixed point approach $1$.
		\item With the same $d_n$ as in (a) but a somewhat lower traffic intensity given as $1- \lambda_n = \frac{(\log n)^{1/2-\epsilon}}{(\log n)^k}$ for some
		$\epsilon \in (0, 1/2)$, one sees that condition of the theorem are satisfied with $\alpha=0$ (i.e. $\beta_n'(\mu_{n,k}) \to 0$). Thus
		the limit process $\Y$, in the case $k>1$, simplifies to $Y_i=0$ for $i>1$ and $Y_1(t) = \sqrt{2} B(t)$. When $k=1$, $Z_1 = Y_1$ is instead given as the following Ornstein-Uhlenbeck(OU) process 
                \begin{equation}
                    \label{eq:OUprocess}
                    Z_1(t) = - \int_0^t Z_1(s) ds + \sqrt{2}B(t).
                \end{equation}
		\item With higher values of $d_n$,  using Theorem \ref{thm:diffusion-limit}, one can analyze fluctuations for systems with higher traffic intensity. For example, suppose that $d_n = \frac{\sqrt{n}}{\log n}$. Then the conditions of the theorem are satisfied with $k=1$ and 
                    $1-\lambda_n \sim (\log n)^2/\sqrt{n}$. In fact in this case $\alpha=0$  and the limit process is described by the one dimensional OU process \eqref{eq:OUprocess}. With a slightly higher traffic intensity given as $1-\lambda_n = ((\log n)^2 -2 \log n \log\log n))/2\sqrt{n}$ one obtains a two dimensional limit diffusion.
                \item The theorem allows for traffic intensity in the Halfin-Whitt scaling regime (i.e. {$\sqrt{n}(1 - \lambda_n) \to \beta > 0$})  as well.
		Specifically, for $k \geq 2$, if $d_n = (\sqrt{n} \log n)^{\frac{1}{k}}$ and  $(1 - \lambda_n) = \frac{\beta + o(1)}{\sqrt{n}}$ for some $\beta>\beta_{0}=1/2k$, the conditions of the theorem are satisfied with $\alpha=0$.
		  With slightly higher traffic intensity
		 (e.g. $\beta + o(1)$ replaced by $\beta_0+ (\frac{1}{k} \log \log n - \log\alpha)/\log n$) conditions of the theorem are met with a non-zero $\alpha$.
             \item \label{item:all-fixed-points} Recall that a  fixed point of \eqref{eq:limit-ode2} when $\lambda=1$ takes the form $\bvec{f}_k^{\gamma}\doteq \bvec{f}_k+ \gamma \bvec{e}_k$, where $k \in \NN$ and $\gamma \in [0, 1)$. Although Theorem \ref{thm:diffusion-limit} only considers settings where the near fixed point $\Mu_n$ converges to 
                 $\bvec{f}_k^{0} = \bvec{f}_k$ for some $k$, it is possible to give conditions under which $\Mu_n$ converges to a different fixed point. Specifically, suppose that $1 \ll d_n^{k+1} \ll n$ and $1 - \lambda_n = \frac{a}{d^{k}_n}$  for some $a>0$. Then it can be checked using Lemma \ref{lem:fixedpointapprox} that $\Mu_n \to \bvec{f}_k^{\gamma}$ with $\gamma=e^{-a}$.
				 	\end{enumerate}
             \item  Suppose for some $a \in (0, \frac{1}{2})$, $d_n = n^{a + o(1)}$  and $\lambda_n$ is taken as in Remark \ref{rem:onthm1} (\ref{item:cond-thm1}) with $k \in \nat$ such that $a(k+1) < 1$. By Theorem \ref{thm:diffusion-limit},
			 all but $O(\sqrt{n})$ queues will have length $k$ over bounded times. This result is analogous to \cite{brightwell2018supermarket}*{Theorem 1.1} which considers, for such choice of $d_n, \lam_n$, the behavior of queues in equilibrium in a setting where $d_n$ queues are sampled with replacement (instead of without replacement as  in the current work). 
			 In fact, for this scenario, \cite{brightwell2018supermarket}*{Theorem 1.1} is able to show a stronger result which says that with high probability, as $n \to \infty$, most of the queues in equilibrium will have length $k$ and that there will be no larger queues.

\end{enumerate}	
\end{remark}

The next theorem describes the fluctuations of $\Zn$ when $d_n$ is of order $\sqrt{n}$.
\begin{theorem}
	\label{thm:exponential}
    Suppose that $\frac{d_n}{\sqrt{n}} \to c \in (0, \infty)$ and $\lambda_n =  1 - \brR*{\frac{\log d_n}{d_n} + \frac{\alpha_n}{\sqrt{n}}}$ with $\alpha_n \to \alpha \in (-\infty, \infty]$ and $\alpha_n = o(n^{1/4})$. Then,  $\Mun \to \bvec{f_1}$ in $\ld$.
    \cg{Suppose further that} $\brC*{\norm{\Zn(0)}_1}_{n \in \nat}$ is tight and $\Zn(0) \pconv \bvec{z}$ in $\lSpace{2}$ with $\bvec{z_{r+}} = 0$ for some  $r \geq 2$.  Then, as $n \to \infty$, $\Zn \dconv \bvec{Z}$ in $\DSpace*{\infty}[\lSpace{2}]$, where $\bvec{Z}$ is the unique pathwise solution to:
	\begin{align*}
            Z_1(t) &= z_1 - \int_0^t \brR{Z_1(s) - Z_2(s)} ds  - (ce^{c\alpha})^{-1}\int_0^t \brR*{e^{cZ_1(s)} - 1} ds + \sqrt{2} B(t),\\
            Z_2(t) &= z_2 - \int_0^t \brR{Z_2(s) - Z_3(s)} ds + (ce^{c\alpha})^{-1}\int_0^t \brR*{e^{cZ_1(s)} - 1} ds,\\
	Z_i (t) &= z_i - \int_0^t \brR{Z_i(s) - Z_{i+1}(s)} ds \qquad \text{ for each } i \in \set{3 \ldots r},\\
	Z_i(t)  &= 0 \qquad \text{for each } i > r,  
	\end{align*}
        \cg{and} $B$ is standard Brownian motion.
\end{theorem}
\begin{remark}\label{rem:onthm2}$\,$
	\begin{enumerate}[(i)]
            \item Note that the coefficients in the above system of equations are only locally Lipschitz and have an exponential growth. However since $c$ is positive, the system of equations has a unique pathwise solution \cg{as is shown in Lemma \ref{lem:exp-uniqueness}}. 
		\item Once more, when \cg{$z_i=0$ for all 
		$i>2$}, the system of equations simplifies to a two dimensional system given as $Z_i=0$ for all $i>2$, and
		\begin{align*}
                    Z_1(t) &= \cg{z_1}- \int_0^t \brR{Z_1(s) - Z_2(s)} ds  - (ce^{c\alpha})^{-1}\int_0^t \brR*{e^{cZ_1(s)} - 1} ds + \sqrt{2} B(t),\\
                    Z_2(t) &=  \cg{z_2}- \int_0^t Z_2(s) ds + (ce^{c\alpha})^{-1}\int_0^t \brR*{e^{cZ_1(s)} - 1} ds.
		\end{align*}
		\item In the regime considered in Theorem \ref{thm:exponential}, the near fixed point $\Mu_n$ can converge to only one particular fixed point of \eqref{eq:limit-ode2}, namely $\bvec{f_1}$. As before, the term $\sqrt{n}(\Mu_n-\bvec{f_1})$ may diverge and thus $\sqrt{n}(\Gn(\cdot)-\bvec{f_1})$ will in general not be tight.
		\item Suppose that $d_n = c\sqrt{n}$ for some $c>0$, $\z=0$ and $1- \lambda_n = (\beta + o(1))\log n/\sqrt{n}$ for some $\beta > \beta_0 = 1/2c$.
		Then the assumptions of the above theorem are satisfied with $\alpha= \infty$ and the limit system simplifies to a one dimensional OU process given as
                $Z_i=0$ for all $i>1$, and $Z_1$ satisifes \eqref{eq:OUprocess}.
		If $(\beta + o(1))\log n$ is replaced by $\beta_0 \log n + \gamma$ for some $\gamma \in \RR$, we instead obtain a two dimensional limit system given as
		$Z_i=0$ for all $i>2$, and
				\begin{align*}
                                    Z_1(t) &= - \int_0^t \brR{Z_1(s) - Z_2(s)} ds  - \cg{e^{-c\gamma}}\int_0^t \brR*{e^{cZ_1(s)} - 1} ds + \sqrt{2} B(t),\\
                                    Z_2(t) &=  - \int_0^t Z_2(s) ds + \cg{e^{-c\gamma}}\int_0^t \brR*{e^{cZ_1(s)} - 1} ds.
				\end{align*}
		\end{enumerate}
\end{remark}
Finally we consider the fluctuation behavior when  $d_n \gg \sqrt{n}$. This time the limit system will involve  reflected diffusion processes.	Recall from \eqref{eq:SMdef} the definition of the Skorohod maps  $\Gamma_{\alpha}$ and $\hat \Gamma_{\alpha}$ associated with a reflection barrier at $\alpha \in \RR$. We will extend the definition of these maps to $\alpha = \infty$ by setting
\begin{equation}\label{eq:skorinf}
	\Gamma_{\infty}(f) = f, \;\; \hat \Gamma_{\infty}(f) = 0 \mbox{ for } f \in  \DSpace*{\infty}[\R]. \end{equation}
\begin{theorem}
	\label{thm:reflection}
        Suppose that $\sqrt{n} \ll d_n$ and 
		\begin{equation}\label{eq:cond402}
			\lambda_n =  1 - \brR*{\frac{\log d_n}{d_n} + \frac{\alpha_n}{\sqrt{n}}}, \mbox{ where } \alpha_n \to \alpha \in [0, \infty],  \mbox{ with } \alpha_n^{-} = O\brR*{\sqrt{n}/d_n},  \mbox{ and } \alpha_n=O(n^{1/6}).
			\end{equation}
			Then $\Mun \to \bvec{f_1}$ in $\ld$.
Suppose further that $\brC*{\norm{\Zn(0)}_1}_{n \in \nat}$ is tight and $\Zn(0) \pconv \bvec{z}$ in $\lSpace{2}$ where  $z_1 \leq \alpha$ and $\bvec{z_{r+}} = 0$ for some $r \geq 2$.   Then, as $n \to \infty$, $\Zn \dconv \bvec{Z} \in \DSpace*{\infty}[\lSpace{2}]$, where $(\bvec{Z}, \eta)$
 is a $\lSpace{2}\times\R_+$ valued continuous process given as the unique solution to:
	\begin{equation}\label{eq:z1zi}
	\begin{aligned}
	Z_1(t) &= \Gamma_{\alpha}\left(z_1 - \int_0^{\cdot} \brR{Z_1(s) - Z_2(s)} ds + \sqrt{2} B(\cdot)\right)(t),\\
	Z_2(t) &= z_2 - \int_0^t \brR{Z_2(s) - Z_3(s)} ds + \eta(t),\\
        \eta(t) &= \hat \Gamma_{\alpha}\left(z_1 - \int_0^{\cdot} \brR{Z_1(s) - Z_2(s)} ds + \sqrt{2} B(\cdot)\right)(t),\\
	Z_i (t) &= z_i - \int_0^t \brR{Z_i(s) - Z_{i+1}(s)} ds \qquad \text{ for each } i \in \set{3 \ldots r},\\
	Z_i(t)  &= 0 \qquad \text{for each } i > r, 
	\end{aligned}
	\end{equation}
        and $B$ is a standard Brownian motion.
\end{theorem}
As a corollary to this Theorem, we obtain the specific regime considered in \cite{mukherjee2018universality} (in fact we provide a slight strengthening in that, unlike \cite{mukherjee2018universality}, we allow $\alpha=0$). See Remark \ref{rem:dn-large} \eqref{it:debankur-disc} for further discussion. 
\begin{cor}
    \label{cor:reflection-at-1}
    As $n \to \infty$, suppose that $d_n \gg \sqrt{n}\log n$ and $\sqrt{n}(1-\lambda_n) \to \alpha \in [0, \infty)$, along with $\sqrt{n}(1-\lambda_n)\ge (\sqrt{n}\log n)/d_n$ for large $n$ if $\alpha = 0$. Let $\Yn(\cdot) \doteq \sqrt{n} \brR*{\Gn(\cdot) - \bvec{f_1}}$ and assume that the sequence of random variables $\brC*{\norm{\Yn(0)}_1}$ is tight, and as $n \to \infty$, $\Yn(0) \pconv \bvec{y} \in \lSpace{2}$ with $\bvec{y_{r+}} = 0$ for some $r \geq 2$. Then $\Yn \dconv \Y$ in $\DSpace*{\infty}[\lSpace{2}]$, where $(\Y, \tilde\eta)$ is the $\lSpace{2} \times [0, \infty)$ valued continuous process given by the unique solution to
	\begin{align*}
            Y_1(t) &= \Gamma_0\brR*{y_1  - \alpha \; \mbox{\em{id}}(\cdot) - \int_0^{\chmd{\cdot}} \brR{Y_1(s) - Y_2(s)} ds + \sqrt{2} B(\cdot)}(t) \\
	Y_2(t) &= y_2 - \int_0^t \brR{Y_2(s) - Y_3(s)} ds + \tilde \eta(t),\\
        \tilde \eta(t) &= \hat \Gamma_{0}\left(y_1 - \alpha\; \mbox{\em{id}}(\cdot) - \int_0^{\cdot} \brR{Y_1(s) - Y_2(s)} ds + \sqrt{2} B(\cdot)\right)(t),\\
	Y_i (t) &= y_i - \int_0^t \brR{Y_i(s) - Y_{i+1}(s)} ds \qquad \text{ for each } i \in \set{3 \ldots r},\\
	Y_i(t)  &= 0 \qquad \text{for each } i > r, 
	\end{align*}
        and \chsb{$B$} is a standard Brownian motion.  
\end{cor}
\begin{remark}$\,$
	\label{rem:dn-large}
	\begin{enumerate}[(i)]
		\item The existence and uniqueness of solutions to the stochastic integral equations in \eqref{eq:z1zi} follows by standard fixed point arguments on using the Lipschitz property of the map $\Gamma_{\alpha}$ on $\DSpace*{\infty}[\R]$.
		This system of equations can equivalently be written as 
	\begin{equation}\label{eq:z1zialt}	
		\begin{aligned}
		Z_1(t) &= z_1 - \int_0^{t} \brR{Z_1(s) - Z_2(s)} ds + \sqrt{2} B(t) - \eta(t),\\
		Z_2(t) &= z_2 - \int_0^t \brR{Z_2(s) - Z_3(s)} ds + \eta(t),\\
		Z_i (t) &= z_i - \int_0^t \brR{Z_i(s) - Z_{i+1}(s)} ds \qquad \text{ for each } i \in \set{3 \ldots r},\\
		Z_i(t)  &= 0 \qquad \text{for each } i > r, 
		\end{aligned}
	\end{equation}
	where $\eta =0$  when $\alpha = \infty$, and when $\alpha \in \RR$, it satisfies 
	\begin{equation} \label{eq:reflcond}
		\left .\begin{aligned}
		 \eta(0) =0 \mbox{ and } \eta &\mbox{ is a monotonically increasing function.}\\
		Z_1(t) &\leq \alpha \\
		 \int_0^{\infty} (\alpha - Z_1(s)) d\eta(s) &= 0
		\end{aligned}\right\}
	\end{equation}
		\item  The convergence $\Mun \to \bvec{f_1}$ along with tightness of $\brC*{\Zn}_{n \in \nat}$ shows that, under the conditions of Theorems \ref{thm:exponential} or \ref{thm:reflection}, most queues will be of length $1$ on any fixed interval $[0,T]$.
		\item The limit system in Theorem \ref{thm:reflection} simplifies when $z_i=0$ for $i >2$ and is  given as $Z_i=0$ for all $i>2$, and
		\begin{align*}
		Z_1(t) &= z_1 - \int_0^t \brR{Z_1(s) - Z_2(s)} ds + \sqrt{2} B(t) - \eta(t),\\
		Z_2(t) &=  z_2- \int_0^t Z_2(s)  ds + \eta(t),
		\end{align*}
		where $\eta$ is as in the statement of the theorem. 
            \item \label{item:dn-boundry} Suppose that $d_n = \sqrt{n} \log n/2a$ for some $a > 0$ and $1-\lambda_n = \frac{a}{\sqrt{n}}+ 
	\frac{2a(\log \log n + O(1))}{\sqrt{n}\log n}$. Then the assumptions in Theorem \ref{thm:reflection} are satisfied with $\alpha = 0$. 
	In this case the reflection barrier is at $0$, namely $Z_1(t)\le 0$ for all $t$.   
		Also note that since $\sqrt{n}(1-\lambda_n)\to a$, we have that $\mu_{n,1} = \lambda_n \to 1$. Since $d_n/\sqrt{n} \to \infty$, this shows that for $k \ge 2$
                \begin{equation*}
                    \sqrt{n}\mu_{n,2}= \sqrt{n}\lambda_n \beta_{n}(\lambda_{n}) \le \sqrt{n}\lambda_n \lambda_{n}^{d_n} = \sqrt{n}(1 - (1-\lambda_n))^{d_n+1} \to 0.
            \end{equation*}
            Using $\mu_{n,i+1} \leq \mu_{n,i}^{d_n}$, see that $\sqrt{n} (\Mun - \bvec{f_1}) \to -a \bvec{e_1} \in \lSpace{1}$ and \chmd{hence} the fluctuations of $\Gn$ about the fixed point $\bvec{f}_1$ can be characterized as well. Specifically, letting $\Yn(\cdot) = \sqrt{n}(\Gn(\cdot)- \bvec{f}_1) = \Zn(\cdot) + \sqrt{n}(\Mun - \bvec{f_1})$, we see that, under the condition of the above theorem, $\Yn \dconv \Y$ in $\DSpace*{\infty}[\lSpace{2}]$, where $\Y = \Z- a \bvec{e_1}$ and hence, assuming $z_i =0$ for $i>2$, $(\Y, \tilde\eta) \in \CC([0,\infty): \lSpace{2}\times\R_+)$ is the unique solution to \eqref{eq:reflcond} with $(Z_1, \eta, \alpha)$
		replaced with $(Y_1, \tilde \eta, -a)$, and the equations
	\begin{align*}
	Y_1(t) &=  y_1  - at - \int_0^t \brR{Y_1(s) - Y_2(s)} ds + \sqrt{2} B(t) - \tilde \eta(t),\\
	Y_2(t) &=  y_2 - \int_0^t Y_2(s)  ds + \tilde \eta(t),
	\end{align*}
        where   $\bvec{y} = \z - a \bvec{e}_1$ and $B$ is a standard Brownian motion. In particular, the limit $\Y$ takes the same form as in
		\cite{eschenfeldt2016supermarket, mukherjee2018universality}.
    \item \label{it:debankur-disc} 
	Suppose that $d_n \gg \sqrt{n} \log n$. Then it is easy to see that \eqref{eq:cond402} holds with some $\alpha>0$ if and only if 
	$\sqrt{n}(1-\lam_n) \to \alpha>0$.
	This regime was  studied in \cite{mukherjee2018universality}. Using the arguments as in (iv) above, it is easy to check that $\sqrt{n} \brR*{\Mun - \bvec{f_1}} \to -\alpha \bvec{e_1}$ in $\lSpace{1}$ (and hence $\lSpace{2}$). Corollary \ref{cor:reflection-at-1} is  immediate from this and Theorem \ref{thm:reflection}. In particular we recover \cite[Theorem 3]{mukherjee2018universality}. However the proof techniques in the current paper are  different from the stochastic coupling techniques employed in \cite{mukherjee2018universality}.  
    \item Suppose $\sqrt{n} \ll d_n \ll \sqrt{n}\log n$ and that \eqref{eq:cond402} holds with $\alpha<\infty$.
	Then, as observed in \cite{mukherjee2018universality}, in this regime $\Yn$ is not tight. Indeed, it is easy to see that
	$\sqrt{n}(1 - \lam_n) = (\sqrt{n}\log d_n)/d_n + \alpha_n \to \infty$. Nevertheless the process $\Zn$ converges in distribution and the limit process has   a reflecting barrier at $\alpha$, i.e. $Z_{1} \leq \alpha$. In particular, unlike the case $d_n \gg \sqrt{n} \log n$, the barrier  in this case does not come from the constraint $G_{n,1} \leq 1$. 
    \item Theorem \ref{thm:reflection} allows for a slower approach to criticality than $n^{-1/2}$, e.g. $\lam_n$ such that $n^{1/3}(\lam_n-1) \to \gamma>0$.
	In this case $\alpha = \infty$ and there is no reflection. When $z_i = 0$ for all $i \geq 1$, this system reduces to the one dimensional OU process given by \eqref{eq:OUprocess} with $Z_i = 0$ for $i > 1$.
	\end{enumerate}
\end{remark}

\section{Poisson Representation of State Processes}
\label{sec:basic-def}
We now embark on the proofs of the main results. We start with a brief overview of the organization of the proofs. In this Section we describe a specific construction of the state process. Proof of the law of large numbers (Theorem \ref{thm:fluid-limit}) is given in Section \ref{sec:proof-fluid}. Section \ref{sec:technical-estimates} describes fine-scaled (deterministic) properties of the function $\beta_n$ and the near fixed points $\Mu_n$ which play a key technical role in the proofs of our diffusion approximations. Section \ref{sec:general-lemmas} derives preliminary estimates required to prove all the main results for the fluctuations of the state process.  Sections \ref{sec:proof-diffusion}, \ref{sec:proof-thm-exponential} and \ref{sec:proof-reflection} complete the proofs of Theorem \ref{thm:diffusion-limit}, \ref{thm:exponential} and \ref{thm:reflection} respectively.   

We start with a specific construction of the state process through time changed Poisson processes (cf. \cites{time-change-reference,ethkur}).
A similar representation has been used in previous work on $JSQ(d)$ systems (cf. \cites{mukherjee2018universality, eschenfeldt2016supermarket}).
Let $\set{N_{i,+}, N_{i,-}:i\geq 1}$ be a collection of mutually independent rate one Poisson processes given on some probability space $(\Omega, \clf, \prob)$.
Then $\Gn$ has the following (equivalent in distribution) representation.
   For  $i \geq 1$ and $t\ge 0$
\begin{equation}
  \begin{aligned}
	G_{n,i}(t) = G_{n,i}(0) -&\frac{1}{n} N_{i,-}\brR*{n \int_0^t [G_{n, i}(s) - G_{n, i+1}(s)] ds} \\
	&+\frac{1}{n} N_{i,+}\brR*{\lambda_n n \int_0^t [\beta_{n}(G_{n, i-1}(s)) - \beta_{n}(G_{n, i} (s))] ds},
  \end{aligned}   \label{eq:time-change}
\end{equation}
	 where $G_{n,0}(t) = 1$ for all $t \geq 0$.
Denoting
$$
\Scale[0.95]
{A_{n,i}(t) \doteq N_{i,+}\brR*{\lambda_n n \int_0^t [\beta_{n}(G_{n, i-1}(s)) - \beta_{n}(G_{n, i} (s))] ds},\, D_{n,i}(t) \doteq N_{i,-}\brR*{n \int_0^t [G_{n, i}(s) - G_{n, i+1}(s)] ds}},$$
the above evolution equation can be rewritten as
\begin{equation}\label{eq:GAD}
G_{n,i}(t) = G_{n,i}(0) - \frac{1}{n} D_{n,i}(t) + \frac{1}{n} A_{n,i}(t), \; i \in \NN, t\ge 0.\end{equation}
Here $D_{n,i}$ describe events causing a decrease in $G_{n,i}$ owing to completion of service events for jobs in queues of length exactly $i$ whilst $A_{n,i}$ describe events causing an increase in $G_{n,i}$ which only occur if the chosen queue of a new job has exactly $i-1$ individuals; this occurs if amongst the $d_n$ random choices made by this job, all of the chosen queues have load at least $i-1$ but not all have load at least $i$.

Let
$$\tilde \clf^n_t = \sigma \left\{ A^n_i(s), D^n_i(s), s\le t, \; i \ge 1\right\},$$
and let $\clf^n_t$ be the augmentation of $\tilde \clf^n_t$ with $\prob$-null sets. It then follows that, for each $i\ge 1$
\begin{align}
 	M_{n, i, +}(t) &\doteq \frac{1}{n} N_{i, +}\brR*{\lambda_n n \int_0^t \beta_{n}(G_{n, i-1}(s)) - \beta_{n}(G_{n, i} (s)) ds} \nonumber\\
              &\qquad - \lambda_n \int_0^t \beta_{n}(G_{n, i-1}(s)) - \beta_{n}(G_{n, i} (s)) ds \label{eq:martplus}
 	\shortintertext{and}
    M_{n, i, -}(t) &\doteq \frac{1}{n} N_{i, -}\brR*{n \int_0^t G_{n, i}(s) - G_{n, i+1}(s) ds} - \int_0^t \brR*{G_{n, i}(s) - G_{n, i+1}(s)} ds \label{eq:martminus}
\end{align}
 are $\{\clf_t^n\}$-martingales with predictable (cross) quadratic variation processes given, for $t \ge 0$, as
\begin{align*}
	\QV{M_{n,i, +}}_t &= \frac{\lambda_n}{n}  \int_0^t \left(\beta_{n}(G_{n, i-1}(s)) - \beta_{n}(G_{n, i} (s))\right) ds, \;  i\ge 1,\\
	\QV{M_{n,i, -}}_t &= \frac{1}{n}\int_0^t \brR*{G_{n, i}(s) - G_{n, i+1}(s)} ds, \;  i\ge 1,\\
	\CQV{M_{n,i, -}}{M_{n,j, -}}_t &=0, \;\; \CQV{M_{n,i, +}}{M_{n,j, +}}_t, \mbox{ for all } i, j \ge 1, i\neq j \mbox{ and }\\
	 \CQV{M_{n,i, +}}{M_{n,k, -}}_t
	&= 0 \mbox{ for all } i, k \ge 1.
\end{align*}
Using these martingales, the evolution of $\Gn$ can be rewritten as
\begin{equation}
\begin{aligned}	
	\label{eq:martdecom}
 	G_{n, i}(t) &= G_{n, i}(0) - \int_0^t \brR*{ G_{n, i}(s) - G_{n, i+1}(s)} ds \\
 	&\quad + \lambda_n \int_0^t \beta_{n}(G_{n, i-1}(s)) - \beta_{n}(G_{n, i} (s)) ds \; + \; M_{n,i}(t), \; i \ge 1
 	\end{aligned}
\end{equation}
where $M_{n, i}(t) \doteq M_{n,i,+}(t) - M_{n,i,-}(t)$ and
	\begin{equation}
		\QV{M_{n,i}}_t = \frac{1}{n} \left( 	\int_0^t \brR*{G_{n, i}(s) - G_{n, i+1}(s)} ds + \lambda_n \int_0^t \left(\beta_{n}(G_{n, i-1}(s)) - \beta_{n}(G_{n, i} (s))\right) ds \right).
		\label{eq:martQV}
	\end{equation}
We will assume throughout that $\Gn(0) \in \ld$ a.s. Then it follows that, for every $t\ge 0$, $\norm{\Gn(t)}_1 < \infty$ almost surely. 
Indeed, over any time interval $[0,t]$ finitely many jobs enter the system a.s. and denoting by $k(n)$ the number of jobs that arrive over $[0,t]$, we see that $\norm{\Gn(t)}_1 \le \norm{\Gn(0)}_1 + k(n)/n<\infty$ a.s.  Thus $\Gn$ is a stochastic 
process with sample paths in $\DD([0,\infty): \ld)$.   Note that, for any $t>0$, $\norm{\Gn(t)-\Gn(t-)}_1 \leq \frac{1}{n}$. 
\begin{remark}
	\label{rem:vect-equation}
Let $\bvec{a_n},\bvec{b} : \ld \to \lSpace{1}$ be given by 
$$ \bvec{a_n}(\bvec{x})_i \doteq \lambda_n (\beta_n(x_{i-1}) - \beta_n(x_i)),
\;\; \bvec{b}(\bvec{x})_i \doteq x_i - x_{i+1},\;\; \bvec{x} \in \ld, \; i \geq 1,$$
where, by convention, for $\bvec{x} \in \ld$, $x_0=1$. Then \eqref{eq:martdecom} can be rewritten as an evolution equation in $\lSpace{1}$ as,
\begin{equation}
    \Gn(t) = \Gn(0) + \int_0^t \brS*{\bvec{a_n}(\Gn(s)) - \bvec{b}(\Gn(s))} ds + \Mn(t), \label{eq:clean-time-change}
\end{equation}
where $\Mn(t) \doteq \brR{M_{n,i}(t)}_{i \geq 1}$ is a stochastic process with sample paths in $\DD([0,\infty): \lSpace{1})$ and the integral is a Bochner-integral \cite{yosida}. 
Note that  the near fixed point $\Mun$ from Definition \ref{def:fixed-point} satisfies $\bvec{a_n}(\Mun) = \bvec{b}(\Mun)$. It is in fact the unique solution to, 
\begin{equation}
\label{eq:fixedpoint}
\bvec{a_n}(\bvec{x}) = \bvec{b}(\bvec{x}) \qquad \text{for } \bvec{x} \in \ld,
\end{equation}
as is seen by adding up all the coordinates of \eqref{eq:fixedpoint} and using $\bvec{x} \in \lSpace{1}$.
In Lemma \ref{lem:martisnull}  we will see that for any $T>0$, as $n \to \infty$, $\sup_{t \leq T} \norm{\Mn(t)}_2 \pconv 0$. Hence if $\Gn(0) = \Mun$, then by \eqref{eq:clean-time-change}, we expect the process $\Gn(t)$ to stay close to $\Mun$ (over any compact time interval) as $n \to \infty$. 
In this sense $\Mu_n$ can be viewed as a `near fixed point' of $\Gn(\cdot)$ and the terminology in Definition \ref{def:fixed-point} is justified.
Another reason for this terminology comes from the results in Theorems \ref{thm:diffusion-limit}--\ref{thm:reflection} which show that, under conditions, $\Mu_n$ converges to one of the fixed points of
the fluid limit \eqref{eq:limit-ode2} when $\lambda =1$.  
\end{remark}

\section{The Law of Large Numbers}
\label{sec:proof-fluid}
In this section we prove Proposition \ref{lem:uniqsoln} and Theorem \ref{thm:fluid-limit}.

\subsection{Uniqueness of Fluid Limit Equations.}
\label{subsec:uniqflu}
In this subsection we show that there is at most one solution of \eqref{eq:limit-ode2} in $\CC([0,\infty): \ld \times \lSpace{\infty})$.
\cg{Results of Section \ref{sec:tighanlpch} will provide existence of solutions to this equation.}
Suppose $(\bvec{g}, \bvec{v})$ and  $(\bvec{g'}, \bvec{v'})$ are two solutions to \eqref{eq:limit-ode2} in $\CC([0,\infty): \ld \times \lSpace{\infty})$. We will now argue that the two solutions are equal.

      We claim that that $v'_i$ and $v_i$ are non-zero for only finitely many $i$'s. Indeed, since $\bvec{g}, \bvec{g'} \in 
        \CSpace{T}[\ld]$, there is a constant $C \in (0,\infty)$ so that $\sup_{s \leq T} \norm{\bvec{g}(s)}_1 \vee \sup_{s \leq T} \norm{\bvec{g'}(s)}_1  \leq C$.  Since 
		\begin{equation}
			\label{eq:esticoor}
			x_i \leq {\norm{\bvec{x}}_1}/{i} \mbox{ for any } \bvec{x} \in \ld,
			\end{equation}
			 taking $M \doteq \ceil{C+1} \in
        \nat$ shows that $\sup_{s \leq T} g_{i}(s) \vee g'_{i}(s) < 1$ for any $i \geq M$. But then by the equivalent representation of \eqref{eq:limit-ode2} given in
		\eqref{eq:limit-ode1} (in particular the second line), we must have $v_i = v'_{i} = 0$ for any $i \geq M$. This proves the claim.

    Since $v_i = v_i'=0$ for $i \geq M$, the first line of the equivalent formulation in \eqref{eq:limit-ode1} shows that 
	both $\bvec{x} = \bvec{g}$ and $\bvec{x}= \bvec{g'}$ satisfy
	the integral equations 
	\begin{align*} 
            x_{i}(t) = r_i - \int_0^t \brR{x_i(s) - x_{i+1}(s)} ds \quad \text{ for } i \geq M+1 \text{ and } t \in [0,T]. \label{eq:unreflected} 
	\end{align*}
        By standard arguments using Gronwall's lemma \cite[Appendix 5]{ethkur}, we then must have $g_{i} = g'_{i}$ for each $i \geq M + 1$. Indeed, letting
		$z_i(\cdot)\doteq  g_i(\cdot)-g'_i(\cdot)$ for $i \ge M+1$ and $v(t) \doteq \sum_{i=M+1}^{\infty} |z_i(t)|$ for $t\in [0,T]$, we have that
		$$|z_i(t)| \le \int_0^t (|z_i(s)| + |z_{i+1}(s)| ) ds \; \mbox{ for all } i \ge M+1, \mbox{ and } \; t \in [0,T]$$
		and so
		$$v(t) \le 2 \int_0^t v(s) ds , \; t \in [0,T],$$
		which implies that $v(t)=0$ for $t \in [0,T]$.

		 We now show that $g_{i} = g'_{i}$ for $i \leq M$.
  From the definition of the Skorohod map in \eqref{eq:SMdef} we see that for $f_1, f_2 \in \DSpace*{\infty}[\R]$ with $f_i(0) \leq 1$, $i=1,2$, and $t \ge 0$
  \begin{equation*}\label{eq:lippropSM}
  	\norm{\Gamma_1(f_1)-\Gamma_2(f_2)}_{*, t} \le 2 \norm{f_1-f_2}_{*,t}, \; \norm{\hat\Gamma_1(f_1)-\hat\Gamma_2(f_2)}_{*, t} \le  \norm{f_1-f_2}_{*,t}.
  \end{equation*}
  Thus, since $(\bvec{g}, \bvec{v})$ and  $(\bvec{g'}, \bvec{v'})$ solve \eqref{eq:limit-ode2},
        \begin{align}
            \norm{g_i - g'_i}_{*, t} &\leq 2 \brR*{\int_0^t \norm{g_i - g'_i}_{*,s} ds  + \int_0^t \norm{g_{i+1} - g'_{i+1}}_{*,s} ds + \norm{v_{i-1} - v'_{i-1}}_{*,t}} \label{eq:SPlip2}, \text{ and }  \\
                \norm{v_i - v'_i}_{*,t} &\leq \int_0^t \norm{g_i - g'_i}_{*,s} ds  + \int_0^t \norm{g_{i+1} - g'_{i+1}}_{*,s} ds + \norm{v_{i-1} - v'_{i-1}}_{*,t} \label{eq:SPlip1}
        \end{align}
        for any $i \geq 1$. Let $H_t \doteq \max_{i \in \{1, \ldots M\}} \norm{g_i - g_i'}_{*,t}$. Note $g_{M+1} = g'_{M+1}$ and hence $H_t = \max_{i \in \{1, \ldots M+1\}} \norm{g_i - g_i'}_{*,t}$. Then from \eqref{eq:SPlip1}, we have
        \begin{equation}
            \label{eq:vrecbound}
            \norm{v_i - v'_i}_{*,t} \leq 2 \int_0^t H_s ds + \norm{v_{i-1} - v'_{i-1}}_{*,t} \quad \text{for any } i \leq M.
        \end{equation}
        Repeatedly using  \eqref{eq:vrecbound} along with $v_0 = v_0'$ shows that $\norm{v_i - v_i'}_{*,t} \leq 2 i \int_0^t H_s ds$ for any $i \leq M$. Using this bound in \eqref{eq:SPlip2} shows for $1 \leq i \leq M$:
        \begin{equation*}
            \norm{g_i - g'_i}_{*,t} \leq 2 \brR*{2\int_0^t H_s ds + 2(i-1) \int_0^t H_s ds} = 4i \int_0^t H_s ds.
        \end{equation*}
        Hence considering the maximum of $\norm{g_i - g'_i}_{*,t}$ over $1 \leq i \leq M$ we get
        \begin{equation*}
            0 \leq H_t \leq 4 M \int_0^t H_s ds \quad \text{ for each } t \in [0,T].
        \end{equation*}
        Gronwall's Lemma now shows that $H_T = 0$, and hence $g_i = g_i'$ for $i = 1 \ldots M$. Finally, since $v_0 = v_0'$, we see recursively from
		the second equation in \eqref{eq:limit-ode2} that $v_i= v'_i$ for all $i\ge 0$.
\hfill \qed

\subsection{Tightness and Limit Point Characterization}
\label{sec:tighanlpch}
Some of the arguments in this section are similar to  \cite{mukherjee2018universality} however in order to keep the presentation self-contained  we provide details in a concise manner.
The next result establishes the convergence of the martingale term $\Mn$ in the semimartingale decomposition in \eqref{eq:clean-time-change}. 
Throughout this subsection and the next we assume that the conditions of Theorem \ref{thm:fluid-limit} are satisfied, namely, $\Gn(0) \pconv \bvec{r}$ in $\ld$, $\lambda_n \to \lambda$ and $d_n \to \infty$, as $n \to \infty$. 
\begin{lemma}
	\label{lem:martisnull}
	For any $T > 0$, $\sup_{s \leq T}\norm{\Mn(s)}_2 \pconv 0$.
\end{lemma}
\begin{proof}
    It suffices to show that for any $T > 0$, $\lim_n \E \sup_{s \leq T} \norm{\Mn(s)}_2^2 = 0$. 
	Applying Doob's maximal inequality we have that
	\begin{equation}
        \E \sup_{s \leq T} \norm{\Mn(s)}_2^2 \leq 4\E \norm{\Mn(T)}_2^2  = 4 \E \sum_{i \geq 1} M_{n,i}(T)^2 \label{eq:mn2sum}. 
		\end{equation}
       Since $\E M_{n,i}^2(T) = \E \QV{M_{n,i}}_T$, using the monotone convergence theorem in \eqref{eq:mn2sum} shows,
	  \begin{align}
        \E \sup_{s \leq T} \norm{\Mn(s)}_2^2   &\le 4 \E \sum_{i \geq 1} \QV{M_{n,i}}_T \leq  4 \frac{T(1+\sup_n \lambda_n)}{n},
        \label{eq:mn2conv0}
	\end{align}
	where the last inequality is from
        \eqref{eq:martQV} on observing that
		 $$\sum_{i=1}^\infty \QV{M_{n,i}}_T \leq \frac{1}{n} \int_0^T G_{n,1}(s) + \frac{\lambda_n}{n} \int_0^T \beta_n(G_{n,0}(t)) \leq \frac{T(1+\lambda_n)}{n}.$$
		 Sending $n \to \infty$ in \eqref{eq:mn2conv0} completes the proof of the lemma.
\end{proof}

The next lemma characterizes compact sets in $\ld$. The proof is standard and can be found for example in \cite{mukherjee2018universality}. 

\begin{prop}
	\label{prop:compactness-l1}
A subset $C \subseteq \ld$ is precompact if and only if the following two conditions hold: 
\begin{enumerate}
    \item (norm-bounded) $\sup_{\bvec{x} \in C} \norm{\bvec{x}}_1 < \infty$, and
        \item (uniformly decaying tails) $\limsup_{M \to \infty} \sup_{\bvec{x} \in C} \sum_{i > M} \abs{x_i} = 0$.
\end{enumerate} 
\end{prop}

\begin{lemma}
    \label{prop:normtight}
	For each $n\in \NN$ there is a  square integrable $\{\clf^n_t\}$-martingale $\{L_n(t)\}$ such that, for any $t\ge 0$,
	\begin{equation*}
	\label{eq:normbound}
        \sup_{s \in [0, t]} \norm{\bvec{G_n}(s)}_1 \leq  \norm{\bvec{G_n}(0)}_1 +  \lambda_n t + L_n(t).
	\end{equation*}
     Furthermore,  $\QV{L_n}_t \leq \frac{\lambda_n t}{n}$, for all $t \geq 0$.
\end{lemma}
\begin{proof}
	For $i= 1, \ldots, n$,  let $X_i(t)$ denote the number of jobs in  the $i$-th server's queue  at time $t$. Then 
	\begin{align*}
	\norm{\Gn(t)}_1 &= \sum_{j=1}^{\infty} G_{n,j}(t) = \sum_{j=1}^\infty \sum_{i = 1}^n \frac{\I{X_i(t) \geq j}}{n}
	= \frac{1}{n} \sum_{i=1}^n \sum_{j = 1}^\infty \I{X_i(t) \geq j}	= \frac{1}{n} \sum_{i=1}^n X_i(t).
	\end{align*}
        Hence $\norm{\Gn(t)}_1$ is the total number of jobs in the system at time $t$, divided by $n$. 

	Since the total number of jobs  in the system at time  $t$ is bounded above by the sum of number of job arrivals by time $t$ and the initial number of jobs,
 $\sup_{s \in [0,t]} \norm{\Gn(s)} \leq \norm{\Gn(0)}_1 + \frac{A_n(t)}{n}$,
        where $A_n(t)$ is the total number of arrivals to the system by time $t$. Since, $A_n$ is a Poisson process with arrival rate $\lambda_n n$, the result follows on setting $L_n(t) = \frac{A_n(t)}{n} - \lambda_n t$, $t\ge 0$. 
\end{proof}

The estimate in the next lemma will be useful when applying Aldous-Kurtz tightness criteria \cite{ethkur} for proving tightness of $\{\Gn\}$.
\begin{lemma}
    \label{prop:deltaboundfl}
	Fix $n \in \NN$ and $\delta \in (0,\infty)$. Let $\tau$ be a bounded $\{\clf^n_t\}$-stopping time.  Then 
		\begin{align*}
	\E \norm{\bvec{G_n}(\tau + \delta) - \bvec{G_n}(\tau)}_1 \leq (\lambda_n + 1) \delta
	\end{align*}
\end{lemma}
\begin{proof}
	From \eqref{eq:GAD},   for any $i \in \nat$, 
	\begin{equation}
	\label{eq:diffbound}
		\abs{G_{n,i}(\tau + \delta) - G_{n,i}(\tau)} \leq \frac{1}{n} \brR*{A_{n,i}(\tau + \delta) - A_{n,i}(\tau) + D_{n,i}(\tau + \delta) - D_{n,i}(\tau)}.
	\end{equation}
From \eqref{eq:martplus} and \eqref{eq:martminus} we see that
        \begin{align*}
                \E \frac{1}{n} \brR*{A_{n,i}(\tau + \delta) - A_{n,i}(\tau)} &=  \lambda_n \E \int_{\tau}^{\tau + \delta} \left(\beta_n(G_{n,i-1}(s)) - \beta_n(G_{n,i}(s))\right) ds\\
  \E \frac{1}{n} \brR*{D_{n,i}(\tau + \delta) - D_{n,i}(\tau)} &= \E \int_{\tau}^{\tau + \delta} \left(G_{n,i}(s) - G_{n,i+1}(s)\right) ds.
                 \end{align*}
        Using the above identities in \eqref{eq:diffbound}
        \begin{equation}
	\begin{aligned}
		&\E \abs{G_{n,i}(\tau + \delta) - G_{n,i}(\tau)} \\
		&\leq \lambda_n \E \int_{\tau}^{\tau + \delta} \left(\beta_n(G_{n,i-1}(s)) - \beta_n(G_{n,i}(s))\right) ds +  \E \int_\tau^{\tau + \delta}  \left(G_{n,i}(s) - G_{n,i+1}(s)\right) ds
	\end{aligned}
	\label{eq:diffbound2}
	\end{equation}
	Adding \eqref{eq:diffbound2} over various values of $i \in \nat$, we have
	\begin{equation*}
		\begin{aligned}
                    \E \norm{\Gn(\tau+\delta) - \Gn(\tau)}_1 &\leq \lambda_n \sum_{i=1}^\infty  \E \int_{\tau}^{\tau + \delta} \left(\beta_n(G_{n,i-1}(s)) - \beta_n(G_{n,i}(s))\right) ds\\
                                                             &\quad + \sum_{i=1}^\infty \E \int_\tau^{\tau + \delta}  \left(G_{n,i}(s) - G_{n,i+1}(s)\right) ds\\
			&\leq \E \int_\tau^{\tau + \delta} \left(\lambda_n\beta_n(G_{n,0}(s)) + G_{n,1}(s)\right) ds\\
			&\leq (\lambda_n + 1) \delta. 
 		\end{aligned}
	\end{equation*}	
\end{proof}
The following lemma will be useful in verifying the tightness of $\{\G_n(t)\}$ in $\ld$ for each fixed $t\ge 0$.
\begin{lemma}
    \label{prop:tails}
    \chmd{For every $n, m \in \NN$} there is a square integrable $\{\clf^n_t\}$ martingale $L_{n,m}(\cdot)$ so that, for all $t\ge 0$,
	\begin{equation*}
            \sup_{s \leq t} \sum_{i > m} G_{n,i}(s) \leq \sum_{i > m} G_{n,i}(0) + \frac{\lambda_nt}{m}\norm{\Gn}_{1,t} + L_{n,m}(t)
   \end{equation*}
   and $\QV{L_{n, m}}_t \leq \frac{\lambda_nt}{n m}\norm{\Gn}_{1,t}$.
\end{lemma}
\begin{proof}
	From \eqref{eq:time-change}, for any $i \in \nat$ and $t \geq 0$:
	\begin{equation}
            \label{eq:gntails}
		G_{n,i}(t)  \leq G_{n,i}(0) + \frac{1}{n} N_{+,i}\brR*{n \lambda_n \int_0^t \beta_n(G_{n,i-1}(s)) - \beta_n(G_{n,i}(s)) ds}
	\end{equation}
        Consider the point-process given by 
		$$B_{n,m}(t) \doteq \sum_{i > m} N_{+,i} \brR*{n \lambda_n \int_0^t \beta_n(G_{n,i-1}(s)) - \beta_n(G_{n,i}(s)) ds}.$$
		 Adding over $i > m$ in \eqref{eq:gntails} we get
\begin{equation}\label{eq:bnmt}
	\sup_{s \leq t} \sum_{i > m} G_{n,i}(s) \leq \sum_{i > m} G_{n,i}(0) + \frac{1}{n} B_{n,m}(t)
\end{equation}
It is easy to see that, with 
$$b_{n,m}(t) \doteq n \lambda_n \sum_{i > m} \int_0^t \beta_n(G_{n,i-1}(s)) - \beta_n(G_{n,i}(s)) ds, \; t \ge 0,$$
$\tilde L_{n,m}(t) \doteq B_{n,m}(t) - b_{n,m}(t)$ is a $\clf^n_t$-martingale and
\begin{equation*}
\begin{aligned}
\QV{\tilde L_{n,m}}_t &= b_{n,m}(t) 
	                  = n \lambda_n \int_0^t \beta_n(G_{n,m}(s)) ds \\
					  &\leq n \lambda_n \int_0^t G_{n,m}(s) ds  
                          \leq n \lambda_n t \brR*{\sup_{s \leq t} G_{n,m}(s)} \leq \frac{n \lambda_nt}{m} \norm{\Gn}_{1,t},
\end{aligned}
\end{equation*}
where, for the last inequality we have used \eqref{eq:esticoor}. The lemma now follows on setting $L_{n,m}(t) = \tilde L_{n,m}(t)/n$
and 
and using \eqref{eq:bnmt}.
\end{proof}

Recall that under our assumptions, $\lambda_n \to \lambda$ and $d_n \to \infty$ as $n \to \infty$.
\begin{lemma}
	\label{lem:fltightness}
	Suppose that $\{\Gn(0)\}_{n \geq 1}$ is a tight sequence of  $\ld$ valued random variables. Then for any $T > 0$, $\brC{\Gn}_{n \geq 1}$ is a tight sequence of $\DSpace{T}[\ld]$ valued random variables.
\end{lemma}
\begin{proof}
    To show that $\brC{\Gn}_{n \geq 1}$ is tight it suffices to show that (cf. \cite[Theorem 8.6]{ethkur})
    \begin{enumerate}
        \item For any $t \in [0, T]$ and $\epsilon > 0$, there is a compact set $\Gamma \subseteq \ld$ so that $\inf_{n \in \nat} \prob \brR*{\Gn(t) \in \Gamma} \geq 1 - \epsilon$.
        \item $\lim_{\delta \to 0} \limsup_{n \to \infty} \sup_{\tau \leq T} \E \norm{\Gn(\tau + \delta) - \Gn(\tau)}_1 = 0$, where the innermost supremum is taken over all $\clf^n_t$-stopping times $\tau$ that are bounded by $T-\delta$.
    \end{enumerate}
    The second condition is immediate from Proposition \ref{prop:deltaboundfl}. Now consider (1).
	Fix $\ep>0$. Let $\bar \lam = \sup_{n\ge 1} \lam_n$.
	Since $\Gn(0)$ is tight, there is a compact $K_1 \subset \ld$ such that
	\begin{equation*}
		\label{eq:inittight}
		\prob(\Gn(0) \in K_1) \ge 1-\frac{\ep}{8} \mbox{ for all } n \in \NN.
	\end{equation*}
	From Proposition \ref{prop:compactness-l1} there is a $\kappa_1 \in (0,\infty)$ such that $\sup_{\bvec{x} \in K_1} \norm{\bvec{x}}_1 \le \kappa_1$. From Lemma \ref{prop:normtight} we can find $\kappa_2 \in (0,\infty)$ so that
	\begin{equation*}
		\prob(\bar \lam T +\norm{L_n}_{1,T}> \kappa_2) \le \frac{\ep}{8}.
	\end{equation*}
	Then,  using the above estimates and Lemma \ref{prop:normtight} again, with $\kappa = \kappa_1+\kappa_2$,
	\begin{equation*}
		\prob(\norm{\Gn}_{1,T} \ge \kappa) \le \frac{\ep}{4}.
	\end{equation*}
	Let $m_k \uparrow \infty$ be a sequence such that 
	$4 \frac{\bar \lam T \kappa}{m_k^{1/2}} \le \frac{\ep}{2^{k+2}} \mbox{ for all } k \in \NN.$
	Define
	$$K_2 = \left\{\bvec{y} \in \ld: \norm{\bvec{y}}_1 \le \kappa \mbox{ and for some } \bvec{x} \in K_1,
	\sum_{i>m_k} y_i \le \sum_{i>m_k} x_i + \frac{\bar \lam T \kappa}{m_k} + \frac{1}{m_k^{1/4}}, \forall k \in \NN\right\}.$$
	Since $K_1$ is compact, it is immediate from Proposition \ref{prop:compactness-l1} that $K_2$ is precompact in $\ld$.
	Also, using Lemma \ref{prop:tails}, for any $t \in [0,T]$,
	\begin{align*}
		\prob(\Gn(t) \in K_2^c) &\le \prob(\norm{\Gn}_{1,T} \ge \kappa) + \prob(\Gn(0)\in K_1^c) +
		\prob(\norm{L_{n,m_k}}_{*,T} > \frac{1}{m_k^{1/4}} \mbox{ for some } k \in \NN)\\
		&\le \frac{\ep}{4} + \frac{\ep}{8} + 4 \cg{\kappa}\bar \lam T\sum_{k=1}^\infty m_k^{1/2} \frac{1}{m_k} \leq \ep,
		\end{align*}
		where the second inequality follows from Doob's maximal inequality and from the expression of $\QV{L_{n,m_k}}$ in Lemma \ref{prop:tails} and the third inequality follows from the choice of $\set{m_k}$. This proves (1) and completes the proof of the lemma.
\end{proof}

The following lemma gives a characterization of the limit points of $\Gn$.
\begin{lemma}
\label{lem:flexistence}
\cg{Fix $T \in (0,\infty)$}. Suppose that, along some subsequence $\{n_k\}_{k \geq 1}$, $\Gnk \dconv \G$ in $\DSpace{T}[\ld]$ as $k \to \infty$. Then $\G \in \CSpace{T}[\ld]$ a.s., and  \eqref{eq:limit-ode2} is satisfied with $(g_i, v_i)$ replaced with $(G_i, V_i)$, where $V_i$ are defined recursively using the second equation in \eqref{eq:limit-ode2} with $V_0(t) = \lam t$ for $t\ge 0$.
\end{lemma}
\begin{proof}
    From Lemma \ref{lem:martisnull} we see that $\bvec{M_{n_k}} \pconv 0$, in $\DSpace{T}[\lSpace{2}]$.
	By Skorohod embedding theorem, let us
	assume that $\bvec{G_{n_k}}, \bvec{M_{n_k}}, \bvec{G}$ are all
	defined on the same probability space and 
        \[(\bvec{G_{n_k}},\bvec{M_{n_k}}) \to (\bvec{G},0), \mbox{ a.s. }\] in $\DSpace{T}[\ld\times \lSpace{2}]$. Since the jumps of $\bvec{G_n}$ have size at most $1/n$,  $\bvec{G}$ is continuous and $\norm{\bvec{G}(s) - \bvec{G_{n_k}}(s)}_{1,T} \to
        0$ a.s. Similarly, $\norm{\bvec{M_{n_k}}(s)}_{2,T} \to 0$ almost surely. To simplify notation from now on we will \cg{take $n_k = n$}.

	Let $V_{n, i}(t) \doteq \lambda_n \int_0^t \beta_{n}(G_{n, i}(s)) ds$ for $i \geq
	1$ and $V_{n, 0}(t) \doteq \lambda_n t$. From \eqref{eq:time-change}, for any $i \geq
	1$
	\begin{equation}
	G_{n, i}(t) = G_{n, i}(0) - \int_0^t \brR{G_{n, i}(s) - G_{n, i+1}(s)} ds + V_{n, i-1}(t) - V_{n, i}(t) + M_{n,i}(t).
	 \label{eq:timechange-with-vn}
	\end{equation} 
	For  $i\in \NN$, $\sup_{s \leq T}\abs{G_{n, i}(s) - G_{i}(s)} \leq
	\sup_{s \leq T} \norm{\bvec{G_n}(s) - \bvec{G}(s)}_1 \to 0$ and $\sup_{s
		\leq T} \abs{M_{n, i}(s)} \leq \sup_{s \leq T} \norm{\bvec{M_{n}}(s)}_2 \to 0$, almost surely as $n \to \infty$.
		We now show that, for each $i \in \NN_0$, $V_{n, i}$ converges uniformly (a.s.) to some limit process $V_i$. Clearly this is true for $i=0$
		and in fact $V_0(t) = \lam t$, $t\ge 0$. Proceeding recursively, suppose now that $V_{n, i-1} \to V_{i-1}$ for some $i \ge 1$.
		Then, since all the terms in \eqref{eq:timechange-with-vn}, except $V_{n,i}$, converge uniformly, $V_{n,i}$ must converges uniformly as well to some limit process $V_i$.  Sending $n\to \infty$ in
	\eqref{eq:timechange-with-vn} we get, for every $t \leq T$ and $i
	\geq 1$:
	\begin{equation*}
	\label{eq:limit-with-v}
	G_{i}(t) = G_i(0) - \int_0^t \brR{G_i(s) - G_{i+1}(s)} ds  + V_{i-1}(t) - V_i(t), \mbox{ a.s. }
	\end{equation*}
        This shows the first line in \eqref{eq:limit-ode1} is satisfied with $(g_i, v_i)$ replaced with $(G_i, V_i)$.
		
		We now show that the second line in  \eqref{eq:limit-ode1} is satisfied as well. Since $V_i$ is the limit of $\set{V_{n,i}}$, the following properties hold:
	\begin{enumeratei}
		\item $V_0(t) = \lambda t$ for all $t \in [0,T]$.
		\item $V_i$ is continuous, non-decreasing and $V_i(0) = 0$.
		\item For any $t \in [0,T]$, $\int_0^t (1-G_i(s)) dV_i(s) = 0$. This is a consequence of the following identities: 
		\begin{align*}
		\int_0^t (1-G_i(s)) dV_i(s) &= \lim_n \int_0^t (1-G_i(s)) dV_{n, i}(s) \\
		&= \lim_n \int_0^t \lambda_n (1-G_i(s)) \beta_{n}(G_{n, i}(s)) ds  \\
		&= \lambda \int_0^t \lim_{n \to \infty} (1-G_i(s))\beta_{n}(G_{n, i}(s)) ds  \\
		&= 0 
		\end{align*}
		where the first equality holds since $G_i$ is a continuous and bounded function and $V_{n,i}\to V_i$ uniformly on $[0,T]$; the second equality uses the definition of $V_{n,i}$, the third is from the dominated convergence theorem, and the fourth follows since
		 $\beta_{n}(x) \leq x^{d_n}$, for $x\in [0,1]$ and $d_n \to \infty$, $\beta_{n}(x) \to 0$ for every $x \in [0,1)$.
	\end{enumeratei}
	Thus we have verified that the second line in \eqref{eq:limit-ode1} is satisfied with $(G_i, V_i)$ as well.  The result is now immediate from Remark \ref{rem:reflection}.
\end{proof}

\subsection{Completing the Proof of LLN}
We can now complete the proofs of Proposition \ref{lem:uniqsoln} and Theorem \ref{thm:fluid-limit}.

\begin{proof}[Proof of Proposition \ref{lem:uniqsoln}]
	Fix $\bvec{r} \in \ld$, $\lam>0$ and choose a sequence $\bvec{r}_n \in \ld$ such that $\bvec{r}_n \to \bvec{r}$ in $\ld$ and for each $i$,
	$n r_{n,i} \in \NN_0$. Consider parameters $\lam_n=\lam$, $d_n=n$ and a JSQ($d_n$) system initialized at $\Gn(0)= \bvec{r}_n$.
	From Lemma \ref{lem:flexistence} we have that  there is at least one solution of \eqref{eq:limit-ode2}
	which is given as a limit point of an arbitrary  weakly convergent subsequence of $\Gn$ (such a sequence exists in view of the tightness shown in Lemma \ref{lem:fltightness}). The fact that this equation can have at most one solution was shown in Section \ref{subsec:uniqflu}. The result follows.
\end{proof}

\begin{proof}[Proof of Theorem \ref{thm:fluid-limit}]
	Since $\Gn(0) \pconv \bvec{r}$ in $\ld$, the hypothesis of Lemma \ref{lem:fltightness} is satisfied, and thus the sequence  $\{\Gn\}_{n \geq 1}$ is tight in $\DSpace{T}[\ld]$ for any fixed $T > 0$. The result is now immediate from Lemma \ref{lem:flexistence}
	and unique solvability of \eqref{eq:limit-ode2} shown in Proposition \ref{lem:uniqsoln}.
\end{proof} 
\begin{remark}\label{rem:cgcevnk}
	We note that the proofs of Lemma \ref{lem:flexistence} and Theorem \ref{thm:fluid-limit} also show that, under the conditions of Theorem \ref{thm:fluid-limit}, for each $i \geq 1$, 
\begin{equation*}
\label{eq:local-time-conv}
    \sup_{t \leq T}\abs{\lambda_{n} \int_0^t \beta_{n}(G_{n,i}(s)) ds - v_i(t)} \pconv 0,
\end{equation*}
where $(g_i, v_i)$ is the unique solution of \eqref{eq:limit-ode2}.
\end{remark}

\section{Properties of the Near Fixed Point}
\label{sec:technical-estimates}
In this section we give some important properties of the near fixed point $\Mu_n$ that will be needed in the proofs of fluctuation theorems. Since $\Mu_n$ is defined in terms of the function $\beta_n$, we begin by giving some results on the asymptotic behavior of
$\beta_n$ and its derivatives. Proofs follow via elementary algebra and Taylor's approximation and can be found in \chsb{Appendix \ref{app:proof-tech-est}}.  Roughly speaking, these results control the 
 error between sampling with and without replacement of $d_n$ servers from a collection of $n$ servers.  
 We first note that the function $\beta_n$ is differentiable on $(0, 1)\setminus\{\frac{d_n-1}{n}\}$
 and the derivative is given as 
	\begin{equation}
	\label{eq:betanprime}
        \beta_n'(x) = \sum_{j=0}^{d_n - 1} (1 - j/n)^{-1} \prod_{\substack{i=0 \\ i \neq j}}^{d_n-1} \frac{x - i/n}{1 - i/n}
		\mbox{ for } x \in (\frac{d_n-1}{n}, 1] \mbox{ and } \beta_n'(x) = 0 \mbox{ for } x \in (0,\frac{d_n-1}{n}).
	\end{equation}
As a convention, we set $\beta_n'(x) = 0$ for $x = \frac{d_n-1}{n}$.

Note that $f(t) = \frac{a+t}{b+t}$ is an increasing function of $t$ on $(-b, \infty)$ when $b > a$. Using this fact in \eqref{eq:betan}  shows that, when  $d_n \le n$,
	\begin{equation}
	\label{eq:gammageqbeta}
		0 \leq \beta_n(x) \leq x^{d_n} \doteq \gamma_n(x), \; x \in [0, 1].
	\end{equation}
	 Using the same fact in \eqref{eq:betanprime} shows that, for $d_n <n$,
	\begin{equation}
	\label{eq:betanprimeub}
        0 \leq \beta_n'(x) \leq \frac{d_n x^{d_n-1}}{1- \frac{d_n}{n}},  x \in (0,1).
	\end{equation}

The following lemma estimates the ratio between $\beta_n$ and $\gamma_n$ and its derivatives. 
\begin{lemma}
	\label{lem:betan-approx}
	Assume $\frac{d_n}{n} \to 0$. Then for any $\epsilon \in (0,1)$, as $n \to \infty$,
	\begin{equation}
		\sup_{x \in [\epsilon, 1]} \abs{\frac{\beta'_n(x)/\beta_n(x)}{\gamma'_n(x)/\gamma_n(x)} - 1} \to 0. \label{eq:betan-ratio-approx}
	\end{equation}
	Furthermore, if $\frac{d_n}{\sqrt{n}} \to 0$, then
	\begin{align}
		\sup_{x \in [\epsilon,1]} \abs{\frac{\beta_n(x)}{\gamma_n(x)} - 1} \to 0 \; \mbox{ and } \;
		\sup_{x \in [\epsilon,1]} \abs{\frac{\beta'_n(x)}{\gamma'_n(x)} - 1} \to 0. \label{eq:betan-prime-approx}
	\end{align}
\end{lemma}	

\begin{corollary}
    \label{cor:betanbounds}
    Assume $d_n \ll n$. Then for any $\epsilon \in (0,1)$
    \begin{equation*}
        \sup_{x\in [\epsilon, 1]} \abs{\log \beta_n(x) - \log \gamma_n(x)} = O\brR*{\frac{d_n^2}{ n}}.
    \end{equation*}
\end{corollary}

Recall the near fixed points $\Mun = \brR*{\mu_{n,i}}_{i \geq 1}$ introduced in Definition \ref{def:fixed-point}.

\begin{corollary}
    \label{cor:betaprimemun}
    Suppose that $d_n \ll n$. Let $i \in \NN$ be such that  $\liminf_n \mu_{n,i} > 0$. Then
    \begin{equation*}
        \label{eq:betaprimemun}
    \lim_{n\to \infty} \frac{\lambda_n\mu_{n,i}\beta_n'(\mu_{n,i})}{d_n \mu_{n,i+1}} = 1.
    \end{equation*}
\end{corollary}

\begin{lemma}
    \label{lem:fixedpointapprox}
    Assume $d_n \ll n$ and fix $\ep \in (0,1)$. Then there is a $C\in (0,\infty)$ and $n_0 \in \NN$ such that, if for some
	$k\in \NN$ and $n_1 \in \NN$, $\mu_{n,k} \ge \ep$ for all $n \ge n_1$, then for all $n \ge n_1\vee n_0$
    \begin{equation*}
        \left|\log \mu_{n,k+1} - \brR*{\log \lambda_n } \brR*{\sum_{i=0}^k d_n^i}\right| \le \frac{C}{ n}\sum_{i=1}^{k} d_n^{i+1}.
        \label{eq:muapprox}
    \end{equation*}
\end{lemma}

\begin{corollary}
	\label{cor:diffusionconditions}
Suppose that $d_n \to \infty$ and that for some  $k \in \NN$  $d_n^{k+1} \ll n$. Suppose also that $1 - \lambda_n = \frac{\xi_n + \log d_n}{d^{k}_n}$ where $\xi_n \to -\log(\alpha) \in (-\infty, \infty]$ and $\frac{\xi_n^2}{d_n} \to 0$. Then  $\mu_{n,k} \to 1$ and $\beta_n'(\mu_{n,k}) \to \alpha$.
\end{corollary}

\begin{lemma}
    \label{prop:ratiosto1}
   Suppose that $\lambda_n \nearrow 1$, $d_n \to \infty$ and  $d_n \ll n$.  Suppose also that, for some $k \geq 2$,  $\mu_{n,k} \to 1$ and $\beta_n'(\mu_{n,k}) \to \alpha \in [0, \infty)$ as $n \to \infty$. Then 
$\beta_n'(\mu_{n,1}) \to \infty
$
and for any $i \in [k-1]$
\begin{equation*}
    \frac{\beta_n'(\mu_{n,i})}{\beta_n'(\mu_{n,1})} \to 1.
    \label{eq:betaratiosone}
\end{equation*}
\end{lemma}

The following result is along the lines of Lemma \ref{lem:betan-approx}. It allows for  weaker assumptions on $d_n$ but gives an approximation only in a neighborhood of $1$.

\begin{lemma}
	\label{lem:betanapprox-1}
	Suppose that $\frac{d_n}{n^{2/3}} \to 0$, as $n \to \infty$. Let $\{\epsilon_n\}$ be a sequence in $[0,1]$ such that $d_n \epsilon_n^2 \to 0$. Then as $n \to \infty$:
	\begin{align}
		\sup_{x \in [1-\epsilon_n,1]} \abs{\frac{\beta_n(x)}{\gamma_n(x)} - 1} \to 0 ,\label{eq:betan-approx1}\\
		\shortintertext{and}
		\sup_{x \in [1-\epsilon_n,1]} \abs{\frac{\beta'_n(x)}{\gamma'_n(x)} - 1} \to 0. \label{eq:betan-prime-approx1}
	\end{align}
\end{lemma}

The next result shows that if $d_n \to \infty$, then the behavior of $\beta_n(x)$  is interesting only when $x$ is sufficiently close to $1$ .
\begin{lemma}
	\label{prop:betan-conv-zero}
	Suppose that $d_n \to \infty$, and let $\epsilon_n \doteq \frac{2\log d_n}{d_n} $. Then as $n \to \infty$
	$\sup_{x \in [0, 1 - \epsilon_n]} \abs{\beta_n(x)}  \to 0$.
	Furthermore, if $\limsup_n \frac{d_n}{n} < 1$ then we also have
	$\sup_{x \in [0, 1 - \epsilon_n]} \abs{\beta_n'(x)}  \to 0$.
\end{lemma}

\section{Preliminary estimates under diffusion scaling}
\label{sec:general-lemmas}

Recall the near fixed point $\Mu_n$ from Definition \ref{def:fixed-point} and the process $\Zn$ introduced in \eqref{eq:zn-proc-def}.
Also, recall the maps $\bvec{a_n}$ and $\bvec{b}$ from Remark \ref{rem:vect-equation}.
We will extend the definition of $\beta_n$ and $\beta'_n$ to $\RR$ by setting $\beta_n(x)= \beta'_n(x)=0$ for $x<0$.
Further, in what follows, for $z<0$ and real valued integrable function $h(\cdot)$, the integral $\int_{[0,z]} h(u) du =-\int_{[z,0]} h(u) du$. We start by giving a semimartingale decomposition for $\Zn$. 

\begin{lemma}
	\label{lem:zn-compact}
	For $t\ge 0$, $\Zn(t)$ satisfies
	\begin{equation}
	\Zn(t) = \Zn(0) + \int_0^t \bvec{A_n}(\Zn(s)) ds - \int_0^t \bvec{b}(\Zn(s)) ds + \sqrt{n} \Mn(t) \label{eq:scaled-marteq}
	\end{equation}
	where $\bvec{A_n}: \lSpace{\infty} \to \lSpace{\infty}$ is given as
	\begin{equation}\label{eq:anzi}\bvec{A_n}(\bvec{z})_i = t_{n,i-1}(z_{i-1}) - t_{n,i}(z_i), \; i \in \NN\end{equation}
	and for $i \in \NN$
	\begin{equation}\label{eq:tniz}
	\begin{aligned}
	t_{n,i}(z) &\doteq  
	\lambda_n\int_{[0,z]} \beta_n'\brR*{\mu_{n,i} + y/\sqrt{n}} dy, &  z \in \RR \\
	t_{n, 0}(z)  &\doteq 0, &  z \in \RR.
	\end{aligned}
	\end{equation}
\end{lemma}
\begin{proof}
	From \eqref{eq:clean-time-change} and since $\bvec{a_n}(\Mun) = \bvec{b}(\Mun)$,
	\begin{align*}
	\sqrt{n}\brR*{\Gn(t) - \Mun} &= \sqrt{n}\brR*{\Gn(0) - \Mun} + \int_0^t \sqrt{n}\brC*{\bvec{a_n}(\Gn(s)) - \bvec{a_n}(\Mun)} ds \nonumber\\
	&{} - \int_0^t \sqrt{n}\brC*{\bvec{b}(\Gn(s)) - \bvec{b}(\Mun)}ds + \sqrt{n}\Mn(t) \label{eq:centered-time-change}
	\end{align*}
        Let $\bvec{A_n}(\bvec{z}) \doteq \sqrt{n}\brC*{\bvec{a_n}(\Mun + n^{-1/2} \bvec{z}) - \bvec{a_n}(\Mun)}$. 
		 By the definition of $\bvec{a_n}$ we see that \eqref{eq:anzi} holds
		 where
        \begin{equation}
            t_{n,i}(z) \doteq \begin{cases}
            \lambda_n \sqrt{n}\brC*{\beta_n\brR*{\mu_{n,i} + z/\sqrt{n}} - \beta_n\brR*{\mu_{n,i}}} & \text{ for } i \geq 1\\
            0                               & \text{ if } i = 0
            \end{cases} \label{eq:tn2}
        \end{equation}
	Clearly, the $t_{n,i}$ defined in \eqref{eq:tn2} is same as  that given in \eqref{eq:tniz}. The result follows.
\end{proof}

\begin{lemma}
	\label{lem:martconv}
	Suppose that $d_n \to \infty$, $\lambda_n \to 1$, and for some $k \geq 1$,
	 $\Gn(0) \pconv \fk$ in  $\lSpace{1}$. 
	Then there is a standard Brownian motion $B$ so that
	$\sqrt{n} \Mn \dconv \sqrt{2}B\ek$	
	in $\DSpace*{\infty}[\lSpace{2}]$.
\end{lemma}
\begin{proof}
	Fix $T > 0$. Since $\Gn(0) \to \fk$ and $\fk$ is a fixed point of  \eqref{eq:limit-ode2},  by Theorem \ref{thm:fluid-limit}, 
	 $\Gn \pconv \fk$ in $\DSpace{T}[\lSpace{1}]$, where $\fk$ here is viewed as the function on $[0,T]$ that takes the constant value $\fk \in \lSpace{1}$. Moreover, by Remark \ref{rem:cgcevnk}, for every $i \geq 1$  $V_{n,i}(t) \doteq \lambda_n \int_0^t \beta_n(G_{n,i}(s)) ds$ converge uniformly on $[0,T]$ in probability to $v_i(t)$, where $v_i$ solves
	\begin{equation}
             v_i  = \hat\Gamma_1\brR*{f_{k,i} -  \brR{f_{k,i} - f_{k,i+1}} \id  + v_{i-1}(\cdot)}, \; i \ge 1, \\
	\end{equation}
        and $v_0(t) \doteq t$, where recall that $\id:[0,T]\to [0,T]$ is the identity map. Recalling the definition of $\fk$ we see by a recursive argument that 
	\begin{equation}
	v_{i}(t) \doteq \begin{cases}
	t & \text{ if } i < k\\
	0 & \text{ if } i \geq k.
	\end{cases}
	\end{equation}
	Combining this with \eqref{eq:martQV}, we have for each $i \geq 1$
	\begin{align*}
	\QV{\sqrt{n} M_{n,i}}_\cdot & = \int_0^\cdot \brR*{G_{n, i}(s) - G_{n, i+1}(s)} ds  + \lambda_n \int_0^\cdot \left(\beta_{n}(G_{n, i-1}(s)) - \beta_{n}(G_{n, i} (s))\right) ds \\
        &\to   (f_{k,i} - f_{k,i+1}) \id + v_{i-1}(\cdot) - v_{i}(\cdot) = H(\cdot), 
	\end{align*}
	in probability in $\CC([0,T]:\RR)$ where
       $$
        H(t) \doteq \begin{cases}
	2t & \text{ if } i = k \\
	0 & \text{ if } i \neq k
	\end{cases},\;\; t \in [0,T].
	$$
	Adding \eqref{eq:martQV} over $i$, we  have for $t\in [0,T]$,
	\begin{equation}\label{eq:qvtail}
            \sum_{i > k}\QV{\sqrt{n} M_{n, i}}_t \leq \int_0^t G_{n,k+1}(s) ds + \lambda_n \int_0^t \beta_n(G_{n,k})(s) ds.
		\end{equation}
	The process on the right side converges in probability in $\CC([0,T]:\RR)$ to		
      $f_{k,k+1} \id + v_{k}(\cdot)  = 0$ and thus $ \sum_{i > k}\QV{\sqrt{n} M_{n, i}}_T $ converges to $0$ in probability.
	By Doob's maximal inequality,
	\begin{equation*}
	\label{eq:marttail}
	n\E \sup_{t \leq T} \sum_{i > k} M_{n,i}^2 (t) \le 4 \E \sum_{i > k} \QV{\sqrt{n}M_{n,i}}_T \to 0, \; \mbox{ as } n\to \infty,
	\end{equation*}
	where the last convergence follows by the dominated convergence theorem on noting that the right side of \eqref{eq:qvtail}
	is bounded above by  $\sup_n (1 + \lambda_n) < \infty$.
	The result now follows on using the martingale central limit theorem (cf. \cite[Theorem 7.1.4]{ethkur}) for the $k$-dimensional martingale sequence
	$(\sqrt{n} M_{n,1}, \cdots , \sqrt{n} M_{n,k})$.
\end{proof}

Recall the functions $t_{n,i}$ from Lemma \ref{lem:zn-compact}.
\begin{lemma}
	Assume that for some $r \in \nat$, $\limsup_{n\to \infty} \mu_{n,r} <1$. Then for any $L > 0$
	\begin{equation*}
	\limsup_{n\to \infty} \sup_{ i \geq r} \sup_{0 < \abs{z} \leq L} \abs{\frac{t_{n,i}(z)}{z}} = 0 \label{eq:taileffectszero}
	\end{equation*}
	 \label{prop:taileffectszero}
\end{lemma}
\begin{proof}
	By \eqref{eq:tniz}:
	\begin{equation*}
	\begin{aligned}
	\sup_{i \geq r} \sup_{0 < \abs{z} \leq L} \abs{\frac{t_{n,i}(z)}{z}} &\leq \lambda_n \sup_{i \geq r} \sup_{0 < \abs{z} \leq L} \sup_{|y|\le z} \abs{\beta_n'\brR*{\mu_{n,i} + \frac{y}{\sqrt{n}}}}\\
	&= \lam_n\sup_{i \geq r} \sup_{\abs{z} \leq L}\abs{\beta_n'\brR*{\mu_{n,i} + \frac{z}{\sqrt{n}}}}
	\leq \lam_n\sup_{0 \leq x \leq  \mu_{n,r} + \frac{L}{\sqrt{n}}} \beta_n'(x)
	\end{aligned}
	\end{equation*}
        which converges to $0$ by Lemma \ref{prop:betan-conv-zero}, since $\limsup_{n\to \infty} \brR*{\mu_{n,r} + \frac{L}{\sqrt{n}}} <1$.
\end{proof}

    For  $L \in (0,\infty)$ define the stopping time 
    \begin{equation}\label{eq:defntaunm}
    \tau_{n, L} \doteq \inf\left\{ t | \norm{\Zn(t)}_{2} \geq L - \frac{1}{\sqrt{n}}\right\}.
    \end{equation}
Since the jumps of $\Zn$ are of size $\frac{1}{\sqrt{n}}$, we see that, for any $T>0$
\begin{equation}
    \label{eq:znboundeduptostopping}
    \norm{\Zn}_{2,T \wedge \tau_{n,L}} \leq L.
\end{equation}
Recall from Section \ref{subsec:notat} the vector $\bvec{z}_{r+} \in \RR^{\infty}$ associated with a vector $\bvec{z} \in \RR^{\infty}$.

\begin{lemma}
	\label{lem:finitedimensional}
        Suppose that as $n \to \infty$,  $\Gn(0) \pconv \fk$ in $\ld$ and $\bvec{Z_{n,r+}}(0) \pconv \bvec{0}$ in $\lSpace{2}$ for some $r > k$. Then for any $T, L > 0$,  $\norm{\bvec{Z_{n,r+}}}_{2,T \wedge \tau_{n, L}} \pconv 0$.
\end{lemma}
\begin{proof}
    For  $i > k$ and $z\in \RR$, let $\Delta_{n,i}(z) \doteq \frac{t_{n,i}(z)}{z} \I{z \neq 0}$.  Then, since $\lim_{n\to \infty} \mu_{n,k+1} = 0$, by Lemma \ref{prop:taileffectszero}
	\begin{equation}
	\label{eq:deltaconvergestozero}
	\delta_{n,L} \doteq \sup_{i \geq k+1} \sup_{\abs{z} \leq L} \abs{\Delta_{n,i}(z)} \to 0, \mbox{ as } n \to \infty.
	\end{equation}
        Next, from \eqref{eq:scaled-marteq},  for $i \geq r+1 > k+1$ 
		\begin{align*}
		Z_{n,i}(t \wedge \tau_{n}) &= Z_{n,i}(0)  + \int_0^{t \wedge \tau_n} \Delta_{n,i-1}(Z_{n,i-1}(s)) Z_{n,i-1}(s) ds - \int_0^{t \wedge \tau_n} \Delta_{n,i}(Z_{n,i}(s))Z_{n,i}(s) ds \nonumber\\
		&\quad- \int_0^{t \wedge \tau_n} (Z_{n,i}(s) - Z_{n, i+1}(s)) ds  + \sqrt{n} M_{n,i}(t \wedge \tau_{n}) 
		\end{align*}
		where we  use $\tau_n$ instead of $\tau_{n ,L}$ for notational simplicity.  Then, observing from  \eqref{eq:deltaconvergestozero} that
		$
		\sup_{i \geq k+1} \sup_{t \in [0,  \tau_{n}]} \abs{\Delta_{n,i}(Z_{n,i}(t))}\leq \delta_{n, L} 
		$, we have
		\begin{align}
				\abs{Z_{n,i}(t \wedge \tau_{n})} &\leq \abs{Z_{n,i}(0)}  + \delta_{n, L} \int_0^{t \wedge \tau_n} ( \abs{Z_{n,i-1}(s)} + \abs{Z_{n,i}(s)}) ds \nonumber\\
						&\quad+ \int_0^{t \wedge \tau_n} (\abs{Z_{n,i}(s)} + \abs{Z_{n, i+1}(s)}) ds  + \abs{ \sqrt{n}M_{n,i}(t \wedge \tau_{n})}. \label{eq:ripeeq}
		\end{align}
		Define maps $\bvec{A_1}, \bvec{A_2} : \R^{\infty} \to \R^{\infty}$ by
		\begin{align*}
			\brR*{\bvec{A_1}\bvec{x}}_i  &= \begin{cases}
                                    x_1  & i = 1   \\
				x_{i-1} + x_i & i \geq 2
 			\end{cases}\\
 			\brR*{\bvec{A_2}\bvec{x}}_i &= x_i + x_{i+1}, \; i \in \NN.
		\end{align*} 
		Then by collecting \eqref{eq:ripeeq} over all $i \geq r+1$ we get
		\begin{equation}
		\begin{aligned}
			\abs{\bvec{Z_{n,r+}}(t \wedge \tau_n)} &\leq	\abs{\bvec{Z_{n,r+}}(0)} + \delta_{n,M} \int_0^{t\wedge \tau_n} \bvec{A_1}\abs{\bvec{Z_{n,r+}}(s)} ds +  \delta_{n,M} \int_0^{t \wedge \tau_n} \abs{Z_{n,r}(s)}  \bvec{e_1}  ds\\
			 & + \int_0^{t \wedge \tau_n} \bvec{A_2}\abs{\bvec{Z_{n,r+}}(s)} ds + \abs{\sqrt{n} \bvec{M_{n,r+}}(t \wedge \tau_{n})}
		\end{aligned}
		\label{eq:vect-recursion}
		\end{equation}
		where the absolute values and the integrals are interpreted as being coordinate-wise for infinite dimensional vectors. Now noting that the maps $\bvec{A_i}$, when considered from $\lSpace{2} \to \lSpace{2}$, are bounded linear operators with norm bounded by 2, we have for $i=1,2$,
		\begin{equation*}
		\norm{\int_0^{t\wedge \tau_n} \bvec{A_i}\abs{\bvec{Z_{n,r+}}(s)} ds}_2  \leq \int_0^{t\wedge \tau_n}2 \norm{\bvec{Z_{n,r+}}(s)}_2 ds.
		\end{equation*}
		Using the triangle inequality in \eqref{eq:vect-recursion} shows for any $t \leq T$
		\begin{equation*}
		\begin{aligned}
		\norm{\bvec{Z_{n,r+}}(t \wedge \tau_n)}_{2} &\leq 
		\norm{\bvec{Z_{n,r+}}	(0)}_{2} + 	\norm{\sqrt{n}\bvec{M_{n,r+}}}_{2, T} + \delta_{n,M} MT \\
		&+ \ 2(1 + \delta_{n, M}) \int_0^{t \wedge \tau_n} 	\norm{\bvec{Z_{n,r+}}(s)}_{2} ds
		\end{aligned}
		\end{equation*}
                where we have used that $\int_0^{t \wedge \tau_n} \abs{Z_{n,r}(s)} ds \leq Lt$. Hence, using Gronwall's inequality
		\begin{equation*}
		\norm{\bvec{Z_{n,r+}}}_{2, T \wedge {\tau_n}} \leq \brR*{\norm{\bvec{Z_{n, r+}}(0)}_{2} + \delta_{n, L} L T +
                    \norm{\sqrt{n}\bvec{M_{n,r+}}}_{2, T}
		} e^{2(1+\delta_{n, L})T}
		\end{equation*}
		Now, as $n \to \infty$, $\norm{\bvec{Z_{n, r+}}(0)}_{2} \pconv 0$ by assumption, $\delta_{n, L} \to 0$ by  \eqref{eq:deltaconvergestozero}, and $\norm{\sqrt{n}\bvec{M_{n,r+}}}_{2, T} \pconv 0$ by Lemma \ref{lem:martconv}. The result follows.
\end{proof}

The following elementary lemma will allow us to replace $\tau_{n,L}\wedge T$ with $T$ in various convergence results. The proof is omitted.
\begin{lemma}
    \label{lem:stopping-time-trick} Fix $T \in [0,\infty)$.
    Suppose for each $n \in \nat$ and $L > 0$ that $\tau_{n,L}$ is a $[0,T]$ valued random variable such that
  $\lim_{L \to \infty} \sup_n \prob \brR*{\tau_{n,L} < T} \to 0$
    for some $T > 0$. Suppose that there is a sequence of stochastic processes $\brC*{F_n}_{n \in \nat}$ with sample paths in
	$\DSpace{T}[\R]$ such that for each $L > 0$
        $\abs{F_n}_{*, T \wedge \tau_{n,L}} \pconv 0$
    as $n \to \infty$. Then in fact $\abs{F_n}_{*,T} \pconv 0$ as $n \to \infty$.
\end{lemma}
The next lemma gives conditions under which the near fixed point $\Mun$ converges to $\bvec{f_1}$.
\begin{lemma}
	\label{prop:f1}
	 Let $0 \leq \epsilon_n \doteq 1 - \lambda_n$ be such that $\epsilon_n \to 0$ and $\epsilon_n d_n \to \infty$. Then $\Mun \to \bvec{f_1}$ in $\lSpace{1}$ as $n \to \infty$.
\end{lemma}
\begin{proof}
	Since $d_n \to \infty$ under our assumptions, in order to show $\Mun \to \bvec{f_1}$  in $\lSpace{1}$ it suffices to show that
(1) $\mu_{n,1} \to 1$ , and (2) $\mu_{n,2} \to 0$. 
The convergence in (1) is immediate on observing that
 $\mu_{n,1} = \lambda_n = 1 - \epsilon_n \to 1$, and (2) follows by noting from Definition \ref{def:fixed-point} and \eqref{eq:gammageqbeta}
 that
   $\mu_{n,2} \leq \mu_{n,1}^{d_n} = (1 - \epsilon_n)^{d_n} \leq e^{-\epsilon_n d_n} \to 0$.
\end{proof}

The following lemma gives a convenient approximation of the  term $t_{n,1}$ introduced in \eqref{eq:tniz} in terms of certain exponentials.
\begin{lemma}
	\label{lem:t-as-exponential}
	Suppose $d_n \to \infty$ and $d_n \ll n^{2/3}$. Let $\lambda_n = 1 - \brR*{\frac{\log d_n}{d_n} + \frac{\alpha_n}{\sqrt{n}}}$ for some real sequence $\{\alpha_n\}$ satisfying $\frac{d_n \alpha_n^2}{n} \to 0$. Then, for any $L > 0$, 
	\begin{equation}
	\label{eq:t1-rearranged}
	\limsup_{n\to \infty}\sup_{0 < \abs{z} \leq L} \abs{\frac{\exp\brR*{\frac{d_n}{\sqrt{n}} (z - \alpha_n)} - \exp \brR*{-\frac{d_n}{\sqrt{n}} \alpha_n}}{t_{n,1}(z) d_n/\sqrt{n}}- 1} = 0.
	\end{equation}
	
\end{lemma}
\begin{proof}
	We only consider the case $0<z \le L$. The case $-L\le z <0$ is treated similarly.
	Recall that $\mu_{n,1} = \lambda_n$. 
	Noting that $d_n(1-\lambda_n + \frac{L}{\sqrt{n}})^2 \leq 4d_n (\frac{\log^2 d_n}{d_n^2} + \alpha^2_n/n+ L^2/n) \to 0$ we have on applying 
	 Lemma \ref{lem:betanapprox-1} with $\ep_n=(1-\lambda_n + \frac{L}{\sqrt{n}})$ that, for any $\abs{z} \leq L$,
	 \begin{align*}
	t_{n,1}(z) &=	\brR*{1+o(1)} \int_0^z \gamma_n' \brR*{\lambda_n + \frac{y}{\sqrt{n}}} dy  \nonumber\\
	&= \brR*{1+o(1)} \int_0^z \exp \brR*{(d_n-1) \log\left\{\lambda_n + \frac{y}{\sqrt{n}}\right\} + \log d_n} dy  \nonumber\\
	&= \brR*{1+o(1)} \int_0^z \exp \brR*{d_n \log\left\{\lambda_n + \frac{y}{\sqrt{n}}\right\} + \log d_n} dy  
	\end{align*}
Using   expansion for $\log (1+h)$ around $h=0$ and once more the fact that $d_n\brR*{1-\lambda_n + \frac{L}{\sqrt{n}}}^2 \to 0$,
\begin{align*}
	t_{n,1}(z) &=\brR*{1+o(1)} \int_0^z \exp \brR*{d_n \brC*{\lambda_n - 1 + \frac{y}{\sqrt{n}}} + \log d_n} dy  \nonumber\\
	&=  \brR*{1+o(1)} \int_0^z \exp \brR*{\frac{d_n}{\sqrt{n}} (y - \alpha_n)} dy\\ 
	&=  \brR*{1+o(1)} \frac{\exp\brR*{\frac{d_n}{\sqrt{n}} (z - \alpha_n)} - \exp \brR*{-\frac{d_n}{\sqrt{n}} \alpha_n}}{d_n/\sqrt{n}}
	\end{align*}
	which proves \eqref{eq:t1-rearranged}.
	
\end{proof}

Proof of the following lemma proceeds by standard arguments but we provide details in Appendix \ref{appendsec:proof-lemma-neg}.
\begin{lemma}
	\label{lem:negativedrift}
        Fix  $T > 0$. Let $f, g, M$ be three bounded measurable functions from $[0,T] \to \R$ and assume further that  $M$ is a right continuous bounded variation
		function. Suppose that $m \doteq \inf_{s \in [0,  T \wedge \tau]} f(s) > 0$ for some $\tau \geq 0$.
Let $z : [0,T] \to \R$ be a bounded measurable function that satisfies for every $t \in [0,T]$
\begin{equation}
	\label{eq:ode-blowingup}
        z(t) = z(0) - \int_0^t f(s) z(s)  ds + \int_0^t g(s) ds + M(t).
\end{equation}
Then for any $t \in [0, T \wedge \tau]$
\begin{equation*}
    \abs{z(t)} \leq \frac{\abs{g}_{*,T \wedge \tau}}{m} + 2\abs{M}_{*, T \wedge \tau } + e^{-mt} \brR*{\abs{z(0)} + \abs{M(0)}}.
\label{eq:znconv0}
 \end{equation*}
\end{lemma}

\begin{lemma}
	\label{lem:drift-exceeds}
	Fix $T \in (0,\infty)$.
	For each $n$,  let $V_n$ be a martingale with respect to some filtration $\{\clg^n_t\}$ such that $V_n(0) = 0$.
	Let $(r_n)_{n=1}^\infty$ be a positive sequence so that $\lim_{n\to \infty} r_n = +\infty$. Suppose  that 
	there is a $C \in (0,\infty)$ such that for all $n\in \NN$ and $t\in [0,T]$, $\QV{V_n}_t \leq C t$. Then for any $\epsilon > 0$
	\begin{equation*}
	\prob \brR*{\sup_{t \leq T} \left(V_n(t)  - r_n t\right) > \epsilon} \to 0
	\end{equation*}
	as $n \to \infty$.
\end{lemma}
\begin{proof}
	Let $\delta_n \doteq \frac{1}{\sqrt{r_n}}$. Then
	\begin{align*}
	\prob \brR*{\sup_{0 \leq t \leq T} \brS*{V_n(t)  - r_n t} > \epsilon} \leq& \prob \brR*{\sup_{0 \leq t \leq \delta_n} \abs{V_n(t)} > \epsilon} + \prob \brR*{\sup_{\delta_n < t \leq T} \abs{V_n(t)} > r_n \delta_n } \nonumber \\
	 \quad \leq&  \frac{4\E V_n(\delta_n)^2}{\epsilon^2} + \frac{4\E V_n(T)^2}{(r_n\delta_n)^2} \nonumber \\
	\quad =&  \frac{4\E \QV{V_n}_{\delta_n}}{\epsilon^2} + \frac{4\E \QV{V_n}_T}{(r_n\delta_n)^2} 
	\leq  \frac{4C\delta_n }{\epsilon^2} + \frac{4CT}{(r_n\delta_n)^2} \to 0
	\end{align*} 
	where the inequality on the second line is from Doob's maximal inequality.
\end{proof}

\section{Proof of Theorem \ref{thm:diffusion-limit}}
\label{sec:proof-diffusion}
Now we start with some preliminary lemmas. Recall from Remark \ref{rem:onthm1}(\ref{item:mun-conv-fk}) that under the hypothesis of Theorem \ref{thm:diffusion-limit} we have $\Mun \to \fk \in \ld$ as $n \to \infty$. Along with the tightness of $\brC*{\norm{\Zn(0)}_1}_{n \in \nat}$ this also shows that $\Gn(0) \to \fk \in \ld$ as $n \to \infty$.

\begin{lemma}
	\label{lem:tapprox}
	Let $d_n \to \infty$, $\frac{d_n}{\sqrt{n}} \to 0$, and $\lambda_n \nearrow 1$. Assume that for some $k \in \nat$, $\Mun \to \fk$  in $\R^{\infty}$. Then for any $M > 0$ and $1 \leq i \leq k$, as $n \to \infty$
	\begin{equation}
		\label{eq:tbyalpha}
		\sup_{0 < \abs{z} \leq M} \abs{(\beta_n'(\mu_{n,i}) z)^{-1}\lam_n\int_0^{z} \beta_n'\brR*{\mu_{n,i} + y/\sqrt{n}}dy  - 1}  \to 0
	\end{equation}
\end{lemma}
\begin{proof}

	To prove \eqref{eq:tbyalpha}, we will approximate $\beta_n'(x)$ by $\gamma_n'(x)$. Using Lemma \ref{lem:betan-approx}
	 \begin{equation*}
	 	\epsilon_n \doteq \sup_{x \in [1/2, 1]}\abs{\frac{\beta_n'(x)}{\gamma_n'(x)} - 1} \to 0.
	 \end{equation*} 
	Since $\mu_{n,k} \to 1$, there is an $N_0$ so that for $n \geq N_0$, $\mu_{n,i} + \frac{y}{\sqrt{n}} \geq \frac{1}{2},$ for any $i \leq k$ and $y \in \R$ with $\abs{y} \leq L$. Hence uniformly in  $0 < \abs{z} \leq L$ and $i \leq k$:
	\begin{align*}
	\frac{\lam_n}{z} \int_0^{z} \frac{\beta'_n(\mu_{n,i} + \frac{y}{\sqrt{n}})}{\beta_n'(\mu_{n,i})} dy &=   \frac{1 + o(1)}{z}\int_0^z \frac{\gamma_n'(\mu_{n,i} + \frac{y}{\sqrt{n}})}{\gamma_n'(\mu_{n,i})} dy  \nonumber\\
	&= \frac{1 + o(1)}{z}\int_0^z \brR*{1 + \frac{y}{\sqrt{n}\mu_{n,i}}}^{d_n-1} dy \nonumber\\
	&= \frac{1 + o(1)}{z}\int_0^z \exp \brC*{ (d_n - 1)\log \brR*{1 + \frac{y}{\sqrt{n}\mu_{n,i}}}} dy \nonumber\\
	&= \frac{1 + o(1)}{z}\int_0^z \exp \brC*{ O\brR*{\frac{d_nL}{\sqrt{n}\mu_{n,k}}}} dy  \to 1
	\end{align*}
	This shows \eqref{eq:tbyalpha}. 
\end{proof}
\begin{remark}
	\label{rem:tapprox-final}
	Suppose that the hypothesis of Lemma \ref{lem:tapprox} hold. Recall the definition of  $\Delta_{n,i}$ for $i>k$ from the proof of Lemma \ref{lem:finitedimensional}.
	We extend this definition by setting 
	\begin{equation}
		\label{eq:Delta}
		\Delta_{n,i}(z) \doteq 
	 t_{n,i}(z)/(\beta_n'(\mu_{n,i})z)\I{z \neq 0}  - 1 \text{ if } 1 \leq i  \leq k
	\end{equation}
        where $t_{n,i}$ is defined by \eqref{eq:tniz}.  With this extension
	\begin{equation}
		\label{eq:tapprox-final}
		t_{n,i}(z) = \begin{cases}
			\beta_n'(\mu_{n,i}) (1 + \Delta_{n,i}(z)) z & \text{ if }  1 \leq i \leq k\\
			\Delta_{n,i}(z) z & \text{ if } i > k 
		\end{cases}.
	\end{equation}
        Using this notation, Lemma \ref{lem:tapprox} and Lemma \ref{prop:taileffectszero} show that, for any $L > 0$
	\begin{equation}
	\label{eq:delta-is-small}
	\gamma_{n, L} \doteq \sup_{i \in \nat} \sup_{0 < \abs{z} \leq L}\abs{\Delta_{n,i}(z)} \to 0  \mbox{ as } n \to \infty.
	\end{equation}
    \end{remark}
The following corollary is an immediate consequence of Remark \ref{rem:tapprox-final} and Lemma \ref{lem:zn-compact}.
\begin{corollary} 
	\label{cor:finalzneq}
	Under the hypothesis of Lemma \ref{lem:tapprox}, $\Zn$ satisfies the following integral equations.\\
       For $i=1$
	\begin{align*}
		Z_{n,1}(t) &= Z_{n,1} (0) - \int_0^t \beta_n'(\mu_{n,1})(1 + \Delta_{n,1}(Z_{n,1}(s))) Z_{n,1}(s) ds
		  - \int_0^t (Z_{n,1}(s) - Z_{n,2}(s)) ds + \sqrt{n}M_{n,1}(t) 
	\end{align*}
	For $i \in \brC*{2, \ldots k}$
	\begin{align*}
		Z_{n,i}(t) &=  Z_{n,i}(0) + \int_0^t \beta_n'(\mu_{n,i-1})(1 + \Delta_{n,i-1}(Z_{n,i-1}(s))) Z_{n,i-1}(s) ds\nonumber \\
		& \quad - \int_0^t \beta_n'(\mu_{n,i})(1 + \Delta_{n,i}(Z_{n,i}(s))) Z_{n,i}(s) ds 
		- \int_0^t (Z_{n,i}(s)-Z_{n,i+1}(s)) ds + \sqrt{n}M_{n,i}(t) .
	\end{align*}
	For $i=k+1$
         \begin{align*}
		Z_{n,k+1}(t) &= Z_{n,k+1}(0) + \int_0^t \beta_n'(\mu_{n,k})(1 + \Delta_{n,k}(Z_{n,k}(s))) Z_{n,k}(s) ds\nonumber\\
		& \quad -\int_0^t \Delta_{n,k+1}(Z_{n,k+1}(s)) Z_{n,k+1}(s)
		 - \int_0^t (Z_{n,k+1}(s) - Z_{n,k+2}(s)) ds + \sqrt{n}M_{n,k+1}(t),
		\end{align*}
		For $i > k+1$
		\begin{align*}
		Z_{n,i}(t) &=  Z_{n,i}(0) + \int_0^t \Delta_{n,i-1}(Z_{n,i-1}(s)) Z_{n,i-1}(s) ds - \int_0^t \Delta_{n,i}(Z_{n,i}(s)) Z_{n,i}(s) ds \nonumber\\
		&\qquad - \int_0^t (Z_{n,i}(s) - Z_{n,i+1}(s)) ds + \sqrt{n}M_{n,i}(t),
	\end{align*}
	where $\Delta_{n,i}$ is as  in Remark \ref{rem:tapprox-final}. 

        Finally, if $Y_{n,1} \doteq \sum_{i=1}^k Z_{n,i}$, then 
        \begin{equation}
            \label{eq:yn1a}
            \begin{aligned}
                Y_{n,1}(t) &= Y_{n,1}(0) - \int_0^t \beta_n'(\mu_{n,k}) \brR*{1 + \Delta_{n,k}(Z_{n,k}(s)} Z_{n,k}(s) ds\\
				 &\quad- \int_0^t (Z_{n,1}(s) - Z_{n,k + 1}(s)) ds + \sum_{i=1}^k \sqrt{n} M_{n,i}(t)
        \end{aligned}
        \end{equation}
\end{corollary}

\begin{lemma}
    \label{lem:large-neg-eigenval}
   Suppose \cg{$\lambda_n \nearrow 1$}, $d_n \to \infty$ and $d_n \ll n$. Assume that for some $k \geq 2$  $\mu_{n,k} \to 1$ and $\beta_n'(\mu_{n,k}) \to \alpha \in [0, \infty)$ as $n \to \infty$. Define the $k-1 \times k-1$ tridiagonal matrix 
   $A_n(s)$ as
   \begin{equation} \label{eq:An}
   	\begin{aligned}
   		A_n(s)[j,j] &= a_{n,j}(s) +1, \; 1 \le j \le k-1,\\
		A_n(s)[j,j+1] &= -1, \; 1 \le j \le k-2, \\
		A_n(s)[j,j-1] &= -a_{n,j-1}(s), \; 2 \le j \le k-1,
   	\end{aligned}
   \end{equation}
   and for all other $j,k$, $A_n(s)[j,k]=0$, where  $a_{n,i}(s) \doteq \beta_n'(\mu_{n,i})(1 + \Delta_{n,i}(Z_{n,i}(s)))$. 
 Then for any  $T, L \in (0,\infty)$
    \begin{equation*}
       \lim_{n\to \infty} \inf_{s \in [0, T \wedge \tau_{n,L}]} \inf_{\vec{x} \in \R^{k-1} \setminus \brC{0}} \frac{\vec{x}^t A_n(s) \vec{x}}{\norm{\vec{x}}^2} = +\infty.
    \end{equation*}
\end{lemma}
\begin{proof}
    Let $b_{n,i}(s) \doteq a_{n,i}(s) + 1$.
and $B_n(s) \doteq A_n(s)  + A_n(s)^t$. Then $B_n(s)$ is a symmetric tridiagonal matrix with entries
   \begin{equation} \label{eq:Bn}
   	\begin{aligned}
   		B_n(s)[j,j] &= 2b_{n,j}(s) , \; 1 \le j \le k-1,\\
		B_n(s)[j,j+1] &= -b_{n,j}, \; 1 \le j \le k-2, \\
		B_n(s)[j,j-1] &= -b_{n,j-1}(s) , \; 2 \le j \le k-1,
   	\end{aligned}
   \end{equation}
   Let $b_n \doteq \beta_n'(\mu_{n,1})$. By Lemma \ref{prop:ratiosto1}, $b_n \to \infty$ and by the uniform convergence in \eqref{eq:delta-is-small} and
    Lemma \ref{prop:ratiosto1} once more
    \begin{equation*}
        \max_{i\le k-1}\sup_{s \in [0, T \wedge \tau_{n,L}]} \abs{\frac{b_{n,i}(s)}{b_n} - 1} \to 0 \mbox{ as } n \to \infty.
    \end{equation*}
   This in particular  shows that
    \begin{equation}
        \label{eq:convtoH}
        \sup_{s \in [0,T \wedge \tau_{n,L}]} \norm{\frac{1}{b_n} B_n(s) - H}_F \to 0,
    \end{equation}
    where $\norm{\cdot}_F$ is the Frobenius norm and $H$ is the  $k-1 \times k-1$ tridiagonal matrix given as
    \begin{equation*} \label{eq:H}
    	\begin{aligned}
    		H[j,j] &= 2 , \; 1 \le j \le k-1,\\
 		H[j,j+1] &= -1, \; 1 \le j \le k-2, \\
 		H[j,j-1] &= -1 , \; 2 \le j \le k-1,
    	\end{aligned}
    \end{equation*}
  Note for any $\vec{x} = (x_1, x_2, \ldots, x_{k-1}) \in \R^{k-1}$ by completing squares 
    \begin{equation*}
        \vec{x}^t H \vec{x} = x_1^2 + (x_2 - x_1)^2 + (x_3 - x_2)^2 + \cdots + (x_{k-2} - x_{k-1})^2 + x_{k-1}^2,
    \end{equation*}
    which is always non-zero if $\vec{x} \neq 0$. Let $c \doteq \inf_{\norm{\vec{x}}=1} \vec{x}^t H \vec{x}$. This shows that $H$ is a positive definite matrix.  
	Since the unit sphere is compact, the minimum is attained and hence $c > 0$. 

   Finally,
    \begin{equation*}
        \begin{aligned}
            \vec{x}^t \frac{1}{b_n} B_n(s) \vec{x} &=  \vec{x}^t H \vec{x} + \vec{x}^t \brR*{\frac{1}{b_n}B_n - H}\vec{x} \\
                                                   &\geq   \vec{x}^t H \vec{x} - \norm{b_n^{-1}B_n - H}_F \norm{\vec{x}}^2 
                                                   \geq \brR*{c - \norm{b_n^{-1} B_n - H}_F} \norm{\vec{x}}^2.
        \end{aligned}
    \end{equation*}
    By \eqref{eq:convtoH}, there is an $N_0 \in \nat$ so that for each $n \geq N_0$ and each $s \in [0, T \wedge \tau_{n,L}]$, $\norm{b_n^{-1} B_n - H}_F \leq c/2$. Hence for each $\vec{x} \in \R^{k-1}$ 
    \begin{equation*}
        2\vec{x}^t A_n(s) \vec{x} = \vec{x}^2 B_n(s) \vec{x} \geq \brR*{c/2}b_n \norm{\vec{x}}^2.
    \end{equation*}
    Since $b_n \to \infty$, this completes the proof.
\end{proof}
    
\begin{lemma} 
    \label{lem:headconv0}
   Suppose that the hypothesis of Theorem \ref{thm:diffusion-limit} holds with $k \geq 2$ and let $\vec{X}_n \doteq (Z_{n,1}, Z_{n,2}, \ldots, Z_{n,k-1})$. Then for  $L, T, \epsilon \in (0,\infty)$  
    \begin{equation}
        \label{eq:xnbounded}
         \prob \brR*{ \sup_{s \in [0,T \wedge \tau_{n,L}]} \norm{\vec{X}_n(s)} > \norm{\vec{X}_n(0)} + \epsilon} \to 0,
    \end{equation}
    and 
    \begin{equation}
        \sup_{s \in [\epsilon, T \wedge \tau_{n,L}]} \norm{\vec{X}_n(s)} \pconv 0, \label{eq:xnconv0}
    \end{equation}
    as $n \to \infty$.
\end{lemma}
\begin{proof}
    Let $\vec{W}_n= (\sqrt{n}M_{n,1}, \ldots, \sqrt{n}M_{n, k-1})$. Then by Corollary \ref{cor:finalzneq}
    \begin{equation}
        \vec{X}_n(t) = \vec{X}_n(0) - \int_0^t A_n(s) \vec{X}_n(s) ds + \vec{e}_{k-1} \int_0^t Z_{n,k}(s) ds + \vec{W_n}(t). \label{eq:Xn} 
    \end{equation}
    where $\vec{e}_{k-1}$ is the vector $(0,0, \ldots, 0,1)' \in \R^{k-1}$ and $A_n(s)$ is $k-1 \times k-1$ matrix defined in \eqref{eq:An}.
    Using Ito's formula to the function  $f(\vec{x}) = \norm{\vec{x}}^2$ along with the semimartingale representation from \eqref{eq:Xn} 
    \begin{align}
        \norm{\vec{X}_n(t)}^2 &= \norm{\vec{X}_n(0)}^2 + 2\int_{0^+}^t \inp{\vec{X}_n(s-)}{d\vec{X}_{n}(s)} + [\vec{W}_n]_t \nonumber\\
    &= \norm{\vec{X}_n(0)}^2 - 2\int_{0}^t \inp{\vec{X}_n(s)}{A_n(s) \vec{X}_n(s)} ds + 2\int_{0}^t Z_{n,k}(s) \inp{\vec{X}_n(s)}{\vec{e}_{k-1}} ds \nonumber\\
        & \quad + 2\int_{0}^t \inp{\vec{X}_n(s-)}{d\vec{W_n}(s)} + [\vec{W_n}]_t,  \label{eq:normxn}
    \end{align}
	 where $[\vec{W_n}]_t \doteq \sum_{i=1}^{k-1} [\sqrt{n}M_{n,i}]_t$.
   Define 
    \begin{align*}
        f_n(s) &\doteq \frac{\vec{X}_n(s)^t A_n(s) \vec{X}_n(s)}{\norm{\vec{X}_n(s)}^2} \I{ \vec{X}_n(s) \neq 0} + n \I{\vec{X}_n(s) = 0} \\
        g_n(s) &\doteq 2Z_{n,k}(s) Z_{n,k-1}(s)\\
        B_n(s) &\doteq 2\int_{0}^t \inp{\vec{X}_n(s-)}{d\vec{W_n}(s)} + [\vec{W_n}]_t \label{eq:Bn}
    \end{align*}
    then \eqref{eq:normxn} becomes
    \begin{equation}
            \label{eq:normnegdrift}
            \norm{\vec{X}_n(t)}^2  = \norm{\vec{X}_n(0)}^2 - 2\int_{0}^t f_n(s) \norm{\vec{X}_n(s)}^2 ds + \int_0^t g_n(s) ds  + B_n(t).
    \end{equation}
    By Lemma \ref{lem:large-neg-eigenval} 
    \begin{equation}
        \label{eq:mninf}
        m_n \doteq \inf_{s \in [0, T \wedge \tau_{n,L}]} f_n(s) \to + \infty \mbox{ as } n \to \infty.
    \end{equation}
 By Ito's isometry, for $i\le k-1$
    \begin{equation*}
    \begin{aligned}
        \E \sup_{s \in [0, T \wedge \tau_{n,L}]} \abs{\int_{0}^t Z_{n,i}(s-) d\brR*{\sqrt{n} M_{n,i}}(s)}^2 &\leq 4 \E \int_0^{T \wedge \tau_{n,L}} Z_{n,i}^2(s-) d[\sqrt{n} M_{n,i}]_s \\
                                                                                                            &\leq 4 L^2 \E [\sqrt{n} M_{n,i}]_T = 4 L^2 \E \QV{\sqrt{n} M_{n,i}}_T.
 \end{aligned}
    \end{equation*}
   where the second to last inequality is obtained by using $\norm{\Zn}_{2, T \wedge \tau_{n,L}} \leq L$.
    From the proof of Lemma \ref{lem:martconv} we see that for any $i \le k-1$, 
    $\E \QV{\sqrt{n} M_{n,i}}_T  \to 0$
    as $n \to \infty$. Hence from the definition of $B_n$ and $\vec{W_n}$
    \begin{equation}
        \label{eq:bnzero}
        \abs{B_n}_{*,T \wedge \tau_{n,L}} \pconv 0  \mbox{ as } n \to \infty.
    \end{equation}
   Applying Lemma \ref{lem:negativedrift} to \eqref{eq:normnegdrift} with $z(t) = \norm{X_n(t)}^2$, $f=2f_n$, $g = g_n$, $M = B_n$, and $\tau = \tau_{n,L}$ 
  shows for any $t \in [0, T \wedge \tau_{n,L}]$ 
    \begin{equation*}
       \norm{\vec{X}_n(t)}^2 \leq \frac{\abs{g_n}_{*,T \wedge \tau_{n,L}}}{2m_n} + 2 \abs{B_n}_{*, T \wedge \tau_{n,L}} + e^{-2 m_n t} \brR*{\norm{\vec{X}_n(0)}^2 + \abs{B_n(0)}}.
    \end{equation*}
    Taking $t = \epsilon_n \doteq 1/\sqrt{m_n}$ and using \eqref{eq:mninf}, \eqref{eq:bnzero}, $\abs{g_n}_{*, T\wedge \tau_{n,L}} \leq 2L^2$ and $\vec{X}_n(0) \pconv (z_1, \ldots z_{k-1})^t$, we see that
    \begin{equation*}
        \label{eq:nonzerotime}
        \sup_{t \in [\epsilon_n, T \wedge \tau_{n,L}]} \norm{\vec{X}_n(t)} \pconv 0.
    \end{equation*}
  Since $\epsilon_n \to 0$, this shows  \eqref{eq:xnconv0}
    for any fixed $\epsilon > 0$. Finally, from \eqref{eq:normnegdrift}, we see that
    \begin{equation*}
        \label{eq:zerotime}
        \sup_{t \in [0, \epsilon_n \wedge \tau_{n,L} \wedge T]} \norm{\vec{X}_n(t)}^2 \leq \norm{\vec{X}_n(0)}^2 + \abs{g_n}_{*,T\wedge \tau_{n,L}} \epsilon_n + \abs{B_n}_{*,T\wedge \tau_{n,L}}.
    \end{equation*}
    The convergence in \eqref{eq:xnbounded} is now immediate on using that
	 $\epsilon_n \to 0$, $\abs{g_n}_{*, T \wedge \tau_{n,L}} \leq 2L^2$ and \eqref{eq:bnzero} holds.
\end{proof}
\begin{corollary}
    \label{cor:znintconv0}
    Under the assumptions of Lemma \ref{lem:headconv0}, for each $i < k$,
$\int_0^{T \wedge \tau_{n,L}} \abs{Z_{n,i}(s)} ds \pconv 0$,  as  $n \to \infty$.
 \end{corollary}
\begin{proof}
    For any $\epsilon > 0$
    \begin{align*}
        \int_0^{T \wedge \tau_{n,L}} \abs{Z_{n,i}(s)} ds &\leq \int_{[0, \epsilon \wedge \tau_{n,L}]} \abs{Z_{n,i}(s)} ds + \int_{[\epsilon, T \wedge \tau_{n,L}]} \abs{Z_{n,i}(s)} ds \\
                                                         &\leq L \epsilon +  \sup_{s \in [\epsilon, T \wedge \tau_{n,L}]} \abs{Z_{n,i}(s)}T.
    \end{align*}
    Now fix  $\delta > 0$ and let $\epsilon=\frac{\delta}{2L}$. Then for any $i<k$
    \begin{align}
        \prob\brR*{\int_0^{T \wedge \tau_{n,L}} \abs{Z_{n,i}(s)} ds > \delta} \leq \prob\brR*{\sup_{s \in [\epsilon, T \wedge \tau_{n,L}]} \abs{Z_{n,i}(s)}  > \frac{\delta}{2\cg{T}}}, 
    \end{align}
    which from \eqref{eq:xnconv0} converges to 0  as $n \to \infty$. Since $\delta > 0$ was arbitrary, this completes the proof.
\end{proof}

\begin{proof}[Proof of Theorem \ref{thm:diffusion-limit}] 
	Recall the conditions in the theorem.
        By Remark \ref{rem:onthm1}(\ref{item:mun-conv-fk}) and the tightness of $\brC*{\norm{\Zn(0)}_1}_{n \in \nat}$, the hypothesis of Lemma \ref{lem:martconv} holds. Hence by Skorokhod's embedding theorem, we can assume that $\{(\Zn(0), \Mn)\}_{n \in \nat}$ and a standard Brownian motion $B$ are defined on a common probability space such that for any $T > 0$ 
 \begin{equation} 
                \label{eq:martconv-as} \sup_{t \leq T} \norm{\sqrt{n}\Mn(t) - \sqrt{2} B(t) \ek}_2 \to 0 
            \end{equation} and 
            \begin{equation}
\label{eq:initialconv-as}
	\norm{\Zn(0) - \bvec{z}}_{2}  \to 0
\end{equation}
almost surely, as $n \to \infty$. Let $\Y$ and $\Y_n$ be as in the statement of the theorem.
 Taking $m \doteq r - k + 1$, $\vec{Y}_n \doteq (\sum_{i=1}^k Z_{n,i}, Z_{n,k+1}, \ldots, Z_{n, r})$
 be the stochastic process with sample paths in $\DSpace{T}[\R^m]$ corresponding to the first $m$ coordinates of $\Y_n$.
  Note $\bvec{Y_{n, m+}} = \bvec{Z_{n,r+}}$, $Z_{n,k} = Y_{n,1} - \sum_{i=1}^{k-1} Z_i$, and for $k=1$, $Y_{n,1} = Z_{n,1}$. Hence by Corollary \ref{cor:finalzneq}, $\vec{Y}_n$ satisfy 
\begin{align}
    Y_{n,1}(t) &= Y_{n,1}(0) -\int_0^t  a_{n,k}(s) Y_{n,1}(s) ds - \I{k=1} \int_0^t Y_{n,1}(s) ds + \int_0^t Y_{n,2}(s) ds + \sqrt{n} M_{n,k}(t) \nonumber\\
               &\quad +\sum_{i=1}^{k-1} \int_0^t a_{n,k}(s) Z_{n,i}(s) ds - \I{k>1}\int_0^t Z_{n,1}(s) ds + \sum_{i=1}^{k-1} \sqrt{n} M_{n,i}(t), \label{eq:yn1}\\
    Y_{n,2}(t) &= Y_{n,2}(0) +\int_0^t  a_{n,k}(s) Y_{n,1}(s) ds -\int_0^t Y_{n,2}(s) ds +\int_0^t Y_{n,3}(s) ds  \nonumber\\
               &\quad - \sum_{i=1}^{k-1} \int_0^t a_{n,k}(s) Z_{n,i}(s)ds - \int_0^t \delta_{n,k+1}(s) Y_{n,2}(s) ds + \sqrt{n} M_{n,k+1}(t), \label{eq:yn2}\\
                                 \shortintertext{and for $i \in \brC*{3,4 \ldots m}$}
    Y_{n,i}(t) &= Y_{n,i}(0) - \int_0^t Y_{n,i}(s) ds + \int_0^t Y_{n,i+1}(s) ds \nonumber\\
               &\quad + \int_0^t \delta_{n,k+i-2}(s) Y_{n, i-1}(s) ds - \int_0^t \delta_{n,k+i-1}(s)Y_{n,i}(s) ds + \sqrt{n} M_{n,k+i-1}(t). \label{eq:yni}
\end{align}
where $a_{n,k}(s)$ is as in Lemma \ref{lem:large-neg-eigenval}  and $\delta_{n,i}(s) \doteq \Delta_{n,i}(Z_{n,i}(s))$ for $i \in \nat$. 

Since
    $\norm{\Zn}_{2,T \wedge \tau_{n,L}} \leq L$,
we have by \eqref{eq:delta-is-small} that, for any $i \in \nat$, 
\begin{equation}
    \label{eq:deltaniconv}
\abs{\delta_{n,i}}_{*, T \wedge \tau_{n,L}} \leq  \gamma_{n,L} \to 0 \mbox{ as } n \to \infty.
\end{equation}
 Moreover since $\beta_n'(\mu_{n,k}) \to \alpha \in [0, \infty)$, this also shows that the term
\begin{equation}
    \label{eq:ankconv}
    \sup_{s \in [0,T \wedge \tau_{n,L}]} \abs{a_{n,k}(s) - \alpha} \to 0 \mbox{ as } n \to \infty.
\end{equation}
We now show that
    \begin{equation}
        \label{eq:ynconv}
        \norm{\Yn - \Y}_{2,T \wedge \tau_{n,L}} \pconv 0 \mbox{ as } n \to \infty.
    \end{equation}
To see this, note that, by Remark \ref{rem:onthm1}(\ref{item:mun-conv-fk}) , the hypothesis of  Lemma \ref{lem:finitedimensional} is satisfied, and hence $\norm{\bvec{Z_{n,r+}}}_{2,T\wedge \tau_{n,L}} \pconv 0$. Since $\bvec{Y_{n,m+}} = \bvec{Z_{n,r+}}$ and $\bvec{Y_{m+}}=0$, this shows that 
    \begin{equation}
        \label{eq:yntailconv}
        \norm{\bvec{Y_{n,m+}} - \bvec{Y_{m+}}}_{2,T \wedge \tau_{n,L}} \pconv 0.
    \end{equation}
   Thus in order to prove \eqref{eq:ynconv} it suffices to show that $\sum_{i=1}^m |Y_{n,i} - Y_i|_{*, T \wedge \tau_{n,L}} \pconv 0$ as $n \to \infty$. To show this we consider
 $U_{n,i} \doteq  Y_{n,i} - Y_i$. 
 Subtracting \eqref{eq:diffusion-limit1} from \eqref{eq:yn1}, \eqref{eq:yn2} and \eqref{eq:yni}, we see
     \begin{equation}
         \label{eq:un}
         \begin{aligned}
         U_{n,1}(t) &= U_{n,1}(0) -(\alpha + \I{k=1}) \int_0^t  U_{n,1}(s) ds + \int_0^t U_{n,2}(s) ds + \sqrt{n} M_{n,k}(t) - B(t) + W_{n,1}(t)\\
          U_{n,2}(t) &= U_{n,2}(0) +\alpha \int_0^t  U_{n,1}(s) -\int_0^t U_{n,2}(s) ds +\int_0^t U_{n,3}(s) ds  + W_{n,2}(t)\\
          U_{n,i}(t) &= U_{n,i}(0) - \int_0^t U_{n,i}(s) ds + \int_0^t U_{n,i+1}(s) ds  + W_{n,i}(t) \quad \text{for $i \in \brC*{3,4 \ldots m}$}
      \end{aligned}
     \end{equation}
    where
    \begin{align*}
        W_{n,1}(t) \doteq&\int_0^t (\alpha- a_{n,k}(s)) Y_{n,1}(s)ds + \sum_{i=1}^{k-1} \int_0^t a_{n,k}(s) Z_{n,i}(s) ds - \I{k>1}\int_0^t Z_{n,1}(s) ds + \sum_{i=1}^{k-1} \sqrt{n} M_{n,i}(t) \\
        W_{n,2}(t) \doteq& \int_0^t  (a_{n,k}(s)-\alpha) Y_{n,1}(s)ds - \sum_{i=1}^{k-1} \int_0^t a_{n,k}(s) Z_{n,i}(s)ds - \int_0^t \delta_{n,k+1}(s) Y_{n,2}(s) ds + \sqrt{n} M_{n,k+1}(t),\\
        W_{n,i}(t) \doteq& \int_0^t \delta_{n,k+i-2}(s) Y_{n, i-1}(s) ds - \int_0^t \delta_{n,k+i-1}(s)Y_{n,i}(s) ds + \sqrt{n} M_{n,k+i-1}(t) \quad \text{for $i \in \brC*{3, \ldots m}$}.  
    \end{align*} 
    Note that, for each $n$,  $\norm{\Yn}_{2, T \wedge \tau_{n,L}} \leq k \norm{\Zn}_{2,T \wedge \tau_{n,L}}$, which by \eqref{eq:znboundeduptostopping} is bounded above by $kL$. Hence by \eqref{eq:ankconv}, \eqref{eq:deltaniconv}, \eqref{eq:martconv-as} and Corollary \ref{cor:znintconv0},
    \begin{equation}
        \label{eq:wnconv0}
        \abs{W_{n,i}}_{*, T \wedge \tau_{n,L}} \pconv 0 \text{ as } n \to \infty
    \end{equation}
    for each $i \in [m]$. Let $\norm{U_n}_{1,t} \doteq \sup_{s \in [0,t]} \sum_{i=1}^m \abs{U_{n,i}(t)}$. 
	Then, from \eqref{eq:un},  for any $t \in [0,T \wedge \tau_{n,L}]$ 
    \begin{equation*}
        \norm{U_n}_{1,t} \leq \sum_{i=1}^m \brR*{\abs{U_{n,i}(0)} + \abs{W_{n,i}}_{*,T \wedge \tau_{n,L}}} + \abs{\sqrt{n}M_{n,k}-B}_{*,T} + R \int_0^t \norm{U_n}_{1,s} ds
    \end{equation*}
with $R \doteq \max(2 \alpha + \I{k=1}, 2)$. Hence by Gronwall's inequality
    \begin{equation*}
        \norm{U_n}_{1,T\wedge \tau_{n,L}} \leq \brR*{\abs{\sqrt{n}M_{n,k}-B}_{*,T} + \sum_{i=1}^m \brR*{\abs{U_{n,i}(0)} + \abs{W_{n,i}}_{*,T \wedge \tau_{n,L}}}} e^{RT}.
    \end{equation*}
    By our hypothesis, as $n \to \infty$, $\abs{U_{n,i}(0)} =\abs{Z_{n,k+i-1}(0) - z_{n,k+i-1}} \pconv 0$ for each $i \in [m]$. Hence by
    \eqref{eq:wnconv0} and \eqref{eq:martconv-as}, $\norm{U_n}_{1,T\wedge \tau_{n,L}} = \sum_{i=1}^m |Y_{n,i} - Y_i|_{*, T \wedge \tau_{n,L}} \pconv 0$ as $n \to \infty$. Combined with \eqref{eq:yntailconv}, this completes the proof of \eqref{eq:ynconv}.

Next we prove \eqref{eq:tightness}. Fix  $\delta > 0$.
Since $\Y$ has sample paths in $\CC([0,T]:\lSpace{2})$, we can find
 $L_1 \in (0,\infty)$ so that 
\begin{equation}
    \label{eq:boundy}
    \prob\brR*{\norm{\Y}_{2, T} > L_1} \leq \frac{\delta}{2}.
\end{equation}
Also, since $\Zn(0) \pconv \bvec{z}$, we can find a  $L_2 \in (0,\infty)$ so that
\begin{equation}
    \label{eq:boundznzero}
    \sup_n \prob\brR*{ \norm{\Zn(0)}_2 > L_2} \leq \frac{\delta}{2}.
\end{equation}
Let $L \doteq (L_1+1) + k (L_2 + 1) + 1$. Also, let $\vec{X}_n$ be as in Lemma \ref{lem:headconv0} when $k>1$. For $k=1$, we set
$\vec{X}_n \doteq 0$. Then,
\begin{equation*}
    \begin{aligned}
        \norm{\Zn}_{2,T\wedge \tau_{n,L}} &\leq  \norm{\vec{X}_n}_{2,T \wedge \tau_{n,L}} + \norm{\Yn - \bvec{e_1}\sum_{i=1}^{k-1} Z_{n,i}}_{2,T\wedge \tau_{n,L}}  \\
                                          &\leq \norm{\Yn}_{2,T\wedge \tau_{n,L}}  + k\I{k > 1} \norm{\vec{X}_n}_{2,T \wedge \tau_{n,L}}. 
    \end{aligned}
\end{equation*}
Hence for each $n \in \nat$
\begin{align*}
    \prob\brR*{\tau_{n,L} \leq T} &\leq \prob\brR*{\norm{\Zn}_{2,T\wedge \tau_{n,L}} > L - 1}\nonumber\\
                                  &\leq \prob\brR*{\norm{\Yn}_{2,T\wedge \tau_{n,L}} > L_1 + 1} + \prob\brR*{\norm{\vec{X}_n}_{2,T\wedge \tau_{n,L}} > L_2 + 1}, \nonumber\\
       &\leq \delta + \prob\brR*{\norm{\Yn-\Y}_{2,T\wedge \tau_{n,L}} > 1} + \prob\brR*{\norm{\vec{X}_n}_{2,T\wedge \tau_{n,L}} > \norm{\vec{X}_n(0)} + 1},
\end{align*}
where the last inequality uses   \eqref{eq:boundy} and \eqref{eq:boundznzero}.
From Lemma \ref{lem:headconv0} and \eqref{eq:ynconv} we see
\begin{equation*}
    \limsup_{n\to \infty} \prob \brR*{\norm{\Zn}_{2,T} \geq L} \leq \limsup_{n\to \infty} \prob\brR*{\tau_{n,L} \leq T} \leq \delta.
\end{equation*}
Since $\delta > 0$ is arbitrary, the convergence in  \eqref{eq:tightness} is now immediate.

This convergence in particular says that
$
    \lim_{L \to \infty} \sup_n \prob \brR*{\tau_{n,L} \leq T} = 0.
$
Using Lemma \ref{lem:stopping-time-trick} with $F_n(t) = \norm{\Yn-\Y}_{2,t}$ we now see from \eqref{eq:ynconv} that
$
    \norm{\Yn - \Y}_{2,T} \pconv 0$
as $n \to \infty$.  Similarly, if $k > 1$, then taking $F_n(t) =  \sup_{s \in [\epsilon, t]} \abs{Z_{n,i}(s)}$ in Lemma \ref{lem:stopping-time-trick} we conclude from Lemma \ref{lem:headconv0} that for each  $i \in [k-1]$ and $\epsilon > 0$
$\sup_{s \in [\epsilon, T]} \abs{Z_{n,i}(s)} \pconv 0$
as $n \to \infty$ . This completes the proof of Theorem \ref{thm:diffusion-limit}.
\end{proof}

\section{Proof of Theorem \ref{thm:exponential}}
\label{sec:proof-thm-exponential}
In this section we give the proof of Theorem \ref{thm:exponential}. We begin by giving a convenient representation for $\Z_n$ under the assumptions of Theorem \ref{thm:exponential} and establishing some apriori convergence properties.
\begin{lemma}
	\label{lem:exp-setup}
        Suppose $c_n = \frac{d_n}{\sqrt{n}} \to c \in (0, \infty)$ and $\lambda_n = 1 - \brR*{\frac{\log d_n}{d_n} + \frac{\alpha_n}{\sqrt{n}}}$ where $\alpha_n \in \RR$, $\liminf_{n\to \infty}\alpha_n > -\infty$ and 
        $\frac{\alpha_n}{n^{1/4}} \to 0$. Suppose also that $\brC*{\norm{\Zn(0)}_1}_{n \in \nat}$ is a tight sequence of random variables and $\bvec{Z_{n,r+}}(0) \pconv \bvec{0}$ in $\lSpace{2}$ for some $r\chmd{\geq}2$. Then there are real stochastic processes $\delta_n, W_{n,i}$ with sample paths in $\DSpace*{\infty}[\R]$ such that for any $t \ge 0$
        \begin{equation}
	\begin{aligned}
            Z_{n,1}(t) = Z_{n,1}(0)- &\int_0^{t} Z_{n,1}(s) ds + \int_0^{t}  Z_{n,2}(s) ds + \sqrt{n} M_{n,1}(t) \\
                                     &- (c_ne^{c_n \alpha_n})^{-1} \int_0^{t} (1+\delta_n(s)) \brR*{e^{c_nZ_{n,1}(s)} - 1}  ds\\
            Z_{n,2}(t) = Z_{n,2}(0)   &- \int_0^{t} Z_{n,2}(s) ds + \int_0^t Z_{n,3}(s) ds + W_{n,2}(t)\\
	&+ (c_ne^{c_n \alpha_n})^{-1} \int_0^{t} (1+\delta_n(s)) \brR*{e^{c_n Z_{n,1}(s)} - 1}   ds   \\
        Z_{n,i}(t) = Z_{n,i}(0)  - &\int_0^{t} Z_{n,i}(s) ds + \int_0^t Z_{n,i+1}(s) ds + W_{n,i}(t) \quad \text{for } i \in \brC*{3, \ldots, r} 
	\end{aligned}
        \label{eq:znexpeq}
        \end{equation}
        and for any fixed  $L, T \in (0,\infty)$,  
        \begin{enumerate}
            \item $\sqrt{n} M_{n,1} \dconv \sqrt{2} B$ in $\DSpace*{\infty}[\R]$ where $B$ is a standard Brownian motion,
            \item $\supabs{\delta_n}{T_{n}} \to 0$ a.s. 
            \item $\supabs{W_{n,i}}{T_n} \pconv 0$ for $i \in \brC*{2, \ldots, r}$,
            \item $\norm{\bvec{Z_{n,r+}}}_{2, T_{n}} \pconv 0$,
        \end{enumerate}
        where $T_{n} \doteq T \wedge \tau_{n, L}$ and $\tau_{n, L}$ is defined as in \eqref{eq:defntaunm}.
\end{lemma}
\begin{proof}
	Recall the definition of $t_{n,i}$ from Lemma \ref{lem:zn-compact}. Define
	\begin{equation*}
		\delta_n(s) \doteq t_{n,1}(Z_{n,1}(s))c_n \brR*{ e^{c_n\brS*{Z_{n,1}(s) - \alpha_n}} - e^{-c_n \alpha_n}}^{-1} - 1
	\end{equation*}
        so that 
        \begin{equation*}
            \label{eq:t1exp}
            t_{n,1}(Z_{n,1}(s)) = (1 + \delta_n(s))c_n^{-1} \brR*{ e^{c_n\brS*{Z_{n,1}(s) - \alpha_n}} - e^{-c_n \alpha_n}}.
        \end{equation*}
	Since $\sup_{s \leq T\wedge \tau_{n, L}} \abs{Z_{n,1}(s)} \leq L$, Lemma \ref{lem:t-as-exponential} shows that $\supabs{\delta_n}{T_n} \to 0$ a.s. 
        Define
 \begin{equation*}
    \begin{aligned}
        W_{n,2}(t) &\doteq  - \int_0^t t_{n,2}(Z_{n,2}(s)) ds + \sqrt{n} M_{n,2}(t) \\
        W_{n,i}(t) &\doteq   \int_0^t t_{n,i-1}(Z_{n,i-1}(s)) ds - \int_0^t t_{n,i}(Z_{n,i}(s)) ds + \sqrt{n} M_{n,i}(t) \quad \text{ for } i \in \{3, \ldots,r \}.
    \end{aligned} 
 \end{equation*}
 From Lemma \ref{lem:zn-compact} it follows that \eqref{eq:znexpeq} is satisfied.  
Lemma \ref{prop:f1} shows that $\Mun \to \bvec{f_1} \in \ld$. Along with the assumed tightness of $\brC*{\norm{\Zn(0)}_1}_{n \in \nat}$, this shows $\Gn(0) = \Mun + \frac{\Zn(0)}{\sqrt{n}}\pconv \bvec{f_1}$ in $\ld$. Hence by Lemma \ref{lem:martconv} and Lemma \ref{lem:finitedimensional}, 
\begin{equation}\label{eq:cgcenmn}
\sqrt{n}\Mn \dconv \sqrt{2}B \bvec{e_1} \mbox{ in } \DSpace*{\infty}[\lSpace{2}]
\end{equation} and  $\norm{\bvec{Z_{n,r+}}}_{2, T \wedge \tau_{n, L}} \pconv 0$ as $n \to \infty$. 
Since $\supabs{Z_{n,i}}{T \wedge \tau_{n,L}} \leq L$ and $\mu_{n,2}\to 0$, Lemma \ref{prop:taileffectszero}, together with
\eqref{eq:cgcenmn}, shows that $\supabs{W_{n,i}}{T_n} \pconv 0 $ for each $i \in \brC*{2, \ldots, r}$, as $n \to \infty$.
\end{proof}
The next lemma gives pathwise existence and uniqueness of solutions to a system of stochastic differential equations
in which the drift fails to satisfy a linear growth condition.
\begin{lemma}
	\label{lem:exp-uniqueness}
        Suppose $c \in (0, \infty)$, $\alpha \in (0,\infty]$ and $B$ is a standard Brownian motion. Then \chmd{for any $r \geq 2$}  the system of equations 
        \begin{equation}
	\begin{aligned}
            Z_{1}(t) &= z_1- \int_0^{t} Z_{1}(s) ds + \int_0^{t}  Z_{2}(s) ds + \sqrt{2} B(t) - \brR*{ce^{c\alpha}}^{-1} \int_0^t \brR*{e^{cZ_1(s)} - 1} ds  	 \\
            Z_{2}(t) &= z_{2}   - \int_0^{t} Z_{2}(s) ds + \int_0^t Z_{3}(s) ds +  \brR*{ce^{c\alpha}}^{-1} \int_0^t \brR*{e^{cZ_1(s)} - 1} ds \\
            Z_{i}(t) &= z_{i}  - \int_0^{t} Z_{i}(s) ds + \int_0^t Z_{i+1}(s) ds \quad \text{for } i \in \brC*{3, \ldots, r} \\
            Z_{i}(t) &= 0 \qquad \text{for } i > r 
        \end{aligned} \label{eq:zexp}
        \end{equation}
        has a unique pathwise solution $\Z$ with sample paths  in $\CSpace*{\infty}[\lSpace{2}]$ for any $(z_1, \ldots, z_r) \in \R^r$.
\end{lemma}
\begin{proof}
    The case when $\alpha = \infty$ is standard and is thus omitted. Consider now the case $\alpha < \infty$. It is straightforward to see that there is a unique   $\bvec{Z_{2+}} \doteq (Z_3, Z_4, \ldots)$ in $\CSpace*{\infty}[\lSpace{2}]$ that solves the last two equations in \eqref{eq:zexp}. Hence it suffices to show that, the system of equations
\begin{equation}\label{eq:z12}
	\begin{aligned}
            Z_1(t) &= z_1 - \brR*{ce^{c\alpha}}^{-1} \int_0^t \brR*{e^{cZ_1(s)} - 1} ds + \int_0^t (Z_2(s)-Z_1(s)) ds + \sqrt{2}B(t)\\
            Z_2(t) &= z_2 + \brR*{ce^{c\alpha}}^{-1} \int_0^t \brR*{e^{cZ_1(s)} - 1}ds - \int_0^t Z_2(s) ds + \int_0^t f(s) ds
	\end{aligned}
\end{equation}
has a unique pathwise solution $(Z_1, Z_2)$ with sample paths  in $\CSpace*{\infty}[\R^2]$ where  $f \doteq Z_3 \in \CSpace*{\infty}[\R]$ is a given (non-random) continuous trajectory and $(z_1, z_2)\in \RR^2$.

Define 
$y_1=z_1+z_2$, $y_2 = z_2$ and 
consider the equation:
\begin{equation}\label{eq:y12}
	\begin{aligned}
            Y_1(t) &= y_1 - \brR*{ce^{c\alpha}}^{-1} \int_0^t \brR*{e^{cY_1(s)} - 1} ds + \int_0^t (Y_2(s)-2Y_1(s)) ds + \sqrt{2}B(t)\\
            Y_2(t) &= y_2  - \int_0^t Y_1(s) ds + \int_0^t f(s) ds + \sqrt{2}B(t).
	\end{aligned}
\end{equation}
Note that $(Z_1, Z_2)$ solve \eqref{eq:z12} if and only if $(Y_1, Y_2)$, with $Y_1=Z_1$ and $Y_2=Z_1+Z_2$ solve \eqref{eq:y12}.
Thus it suffices to prove
existence and uniqueness of solutions for
\eqref{eq:y12}.

For $L\in (0,\infty)$, let $\eta_L: \RR \to [0,1]$ be such that $\eta_L$ is smooth, $\eta_L(x)=1$ for $|x|\le L$ and $\eta_L(x)=0$
for $|x|\ge L+1$. Consider the equation
\begin{equation}\label{eq:y12m}
	\begin{aligned}
            Y_1^L(t) &= y_1 - \brR*{ce^{c\alpha}}^{-1} \int_0^t e^{cY_1^L(s)}\eta_L(Y_1^L(s)) ds  + (ce^{c\alpha})^{-1} t\\ 
                     &\quad  + \int_0^t (Y_2^L(s)-2Y_1^L(s)) ds + \sqrt{2}B(t)\\
            Y_2^L(t) &= y_2  - \int_0^t Y_1^L(s) ds + \int_0^t f(s) ds + \sqrt{2}B(t).
	\end{aligned}
\end{equation}
Since for each $L$ \eqref{eq:y12m} is an equation with (globally) Lipschitz coefficients, by standard results,  it has a unique pathwise \chmd{continuous} solution.

Fix  $T \in (0,\infty)$ and let $\tau_L = \inf\{t\ge 0: |Y_1^L(t)| \ge L\} \wedge T$ for any $L > 0$. Then by pathwise uniqueness of \eqref{eq:y12m}, for
$0 \le t \le \tau_L \wedge \tau_{L+1}$, 

\begin{equation*}
Y^L(t) = Y^{L+1}(t).
\end{equation*}
This in particular shows that, $\tau_L \le \tau_{L+1}$ a.s.

We now estimate the second moment of $|Y_1^L(t)|$.
By It\^{o}'s formula
\begin{equation*}
	\begin{aligned}
            (Y_1^L(t))^2 &= (y_1)^2 - 2(ce^{c\alpha})^{-1}\int_0^t Y_1^L(s)e^{cY_1^L(s)}\eta_L(Y_1^L(s)) ds + 2(ce^{c\alpha})^{-1}\int_0^t Y_1^L(s) ds \\
                         &\quad + 2\int_0^t Y_1^L(s)(Y_2^L(s)-2Y_1^L(s)) ds + 2\sqrt{2}\int_0^t Y_1^L(s) dB(s)+ 2t\\
            (Y_2^L(t))^2 &= (y_2)^2  - 2\int_0^t Y_1^L(s) Y_2^L(s)ds + 2\int_0^t Y_2^L(s) f(s) ds + 2\sqrt{2} \int_0^t Y_2^L(s) dB(s)+  2t.
	\end{aligned}
\end{equation*}
Thus 
\begin{equation*}
	\begin{aligned}
	(Y_1^L(t))^2 + (Y_2^L(t))^2 &= (y_1)^2 + (y_2)^2 - 2(ce^{c\alpha})^{-1}\int_0^t Y_1^L(s)e^{cY_1^L(s)}\eta_L(Y_1^L(s)) ds \\
                                    &\quad + 2(ce^{c\alpha})^{-1}\int_0^t Y^L_1(s) ds  + 2 \int_0^t Y_2^L(s) f(s) ds\\
                                    &\quad - 4\int_0^t (Y_1^L(s))^2 ds + 2\sqrt{2} \int_0^t (Y_1^L(s)+Y_2^L(s))  dB(s)+ 4t.\\
	\end{aligned}
\end{equation*}
Since $c>0$, we have on using the inequality  $\abs{x} \leq 1 + \abs{x}^2$ that $-xe^{cx}\eta_L(x) \le (1+|x|^2)$ for all $x \in \RR$. Thus \chmd{with $\| Y^L \|_{*,t} \doteq \sup_{s \in [0,t]} \norm{Y^L(s)}$:}
\begin{equation*}
	\begin{aligned}
            \|Y^L\|_{*,t}^2 &\le \|y\|^2 + 4(ce^{c\alpha})^{-1}\int_0^t (1+\|Y^L\|_{*,s}^2) ds  \\
                            &\quad+ 2 \int_0^t (1 + \|Y^L\|_{*,s}^2) \abs{f(s)} ds \\
                            &\quad+  2 \sqrt{2} \brR*{1 + \sup_{0\le s \le t}\left|\int_0^s (Y_1^L(u)+Y_2^L(u))  dB(u)\right|^2}+ 4t.
	\end{aligned}
\end{equation*}
Taking expectations, for any $t \in [0,T]$
\begin{equation*}
	\begin{aligned}
            \E\|Y^L\|_{*,t}^2 &\le \|y\|^2 + (4(ce^{c\alpha})^{-1}+2\supabs{f}{T})\int_0^t (1+\E\|Y^L\|_{*,s}^2) ds\\
			   &\quad + 2\sqrt{2} \brR*{1 + 4\E\int_0^t |Y_1^L(u)+Y_2^L(u)|^2  du} +4t\\
	&\le (\|y\|^2+K(T+1)) + K\int_0^t \E\|Y^L\|_{*,s}^2 ds.
	\end{aligned}
\end{equation*}
with $K \doteq 4(ce^{c\alpha})^{-1} + 2\supabs{f}{T} + 16\sqrt{2}$. By Gronwall lemma, for every $L \in \NN$
\begin{equation*}
    E\|Y^L\|_{*,T}^2 \le (\|y\|^2 + K(T+1)) e^{KT} \doteq c_1.
\end{equation*}
Thus, as $L\to \infty$
$$P(\tau_L < T) \le P(\|Y^L\|_{*,T}\ge L) \le c_1/L^2 \to 0$$
and consequently $\tau_L \uparrow T$ as $L\to \infty$.
Now define 
$Y(t) \doteq Y^L(t)$  for  $0 \le t \le \tau_L.$
Then $Y$ is a solution of \eqref{eq:y12} on $[0, T)$.
The same argument as before shows that this is the unique pathwise solution on $[0,T)$.
Since $T$ is arbitrary we get a unique pathwise solution of \eqref{eq:y12} on $[0,\infty)$.
This completes the proof of the lemma.
\end{proof}

\begin{lemma}
	\label{lem:exp-conv}
        Suppose the assumptions of Theorem \ref{thm:exponential} hold. 
		Suppose further that $\Zn(0), \Mn$ and a standard Brownian motion $B$ are given on a common probability space such that
		$\Zn(0) \to \bvec{z}$ in $\ld$ and $\Mn \to \sqrt{2} B \bvec{e_1}$ in $\DSpace*{\infty}[\lSpace{2}]$ almost surely. Let $\Z$ be as defined in Lemma \ref{lem:exp-uniqueness}. Then for any  $T, L \in (0,\infty)$
        \begin{equation}\label{eq:zntozcg}
            \norm{\Zn - \Z}_{2, T \wedge \tau_{n,L} \wedge \tau_L} \pconv 0 \quad \text{  as } n \to \infty.
        \end{equation}
        where $\tau_L \doteq \inf\brC*{t \mid \norm{\Zn(t)}_{2,t} > L}$.
\end{lemma}
\begin{proof}
    Fix $L, T \in (0,\infty)$ and let $T_n \doteq T \wedge \tau_{n,L} \wedge \tau_{L}$ and $U_{n,i} \doteq Z_{n,i} - Z_{i}$ for $i \in \nat$. Using the estimate
        $\abs{e^{ax} - e^{ay}} \leq ae^{a(x \vee y)} \abs{x-y}$ for $x, y \in \R$, $a \ge 0$  and since $\abs{Z_{n,1}(s)}, \abs{Z_{1}(s)} \leq L$ for any $s \in [0,T_n]$,
    \begin{equation*}
        \begin{aligned}
            &\abs{a_n(s) e^{ c_n Z_{n,1}(s)} - a e^{cZ_{1}(s)}} \\
            &\leq \abs{a_n(s) e^{c_n Z_{n,1}(s)} - a_n(s) e^{c_nZ_{1}(s)}} + \abs{a_n(s) e^{c_n Z_{1}(s)} - a_n(s) e^{cZ_{1}(s)}} + \abs{e^{cZ_1(s)}}\abs{a_n(s) - a}\\
            &\leq \abs{a_n(s)} c_ne^{c_n L} \abs{U_{n,1}(s)} + \abs{a_n(s)}Le^{L(c_n \vee c)}\abs{c_n - c} + e^{cL}\abs{a_n(s) - a}
        \end{aligned}
    \end{equation*}
    where $a_n(s) \doteq (c_n e^{c_n \alpha_n})^{-1}(1 + \delta_n(s))$, $c_n = d_n/\sqrt{n}$, $\delta_n$ is as in Lemma \ref{lem:exp-setup}, and $a \doteq (ce^{c\alpha})^{-1}$. Since $c_n \to c$ and $\supabs{\delta_n}{T_n} \to 0$ by Lemma \ref{lem:exp-setup}, $\supabs{a_n - a}{T_n} \to 0$. Hence  for any $s \in [0,T_n]$
    \begin{equation}
        \supabs{a_ne^{c_nZ_{n,1}} - ae^{cZ_1}}{s} \leq K \supabs{U_{n,1}}{s} + r_n  \label{eq:exp-term-bound}
    \end{equation}
    where $K \doteq \sup_n  \brR*{c_ne^{c_nL}\supabs{a_n}{T_n}} < \infty$ and 
    \begin{equation*}\label{eq:eqrn}
        r_n \doteq \supabs{a_n}{T_n}Le^{L (c_n \vee c)}\abs{c_n-c} + e^{cL} \supabs{a_n-a}{T_n} \to 0
    \end{equation*}
    almost surely.
    Subtracting \eqref{eq:zexp} from \eqref{eq:znexpeq},  for any $t > 0$,
    \begin{equation}
        \begin{aligned}
            U_{n,1}(t) &= U_{n,1}(0) - \int_0^t \brR*{U_{n,1}(s) - U_{n,2}(s)} ds + \sqrt{n}M_{n,1}(t) - \sqrt{2}B(t)  \\
                       &\quad  - \int_0^t \brR*{a_{n,1}(s) e^{c_n Z_{n,1}(s)} - a e^{cZ_1(s)}} ds + \int_0^t \brR*{a_n(s) - a} ds \\
            U_{n,2}(t) &= U_{n,2}(0) - \int_0^t \brR*{U_{n,2}(s) - U_{n,3}(s)} ds + W_{n,2}(t)  \\
                       &\quad  + \int_0^t \brR*{a_{n,1}(s) e^{c_n Z_{n,1}(s)} - a e^{cZ_1(s)}} ds - \int_0^t \brR*{a_n(s) - a} ds \\
            U_{n,i}(t) &= U_{n,i}(0) - \int_0^t \brR*{U_{n,i}(s) - U_{n,i+1}(s)} ds + W_{n,i}(t)  \quad \text{for } i \in \{ 3, \ldots, r\}.
        \end{aligned}
        \label{eq:unexp}
    \end{equation}
    Let $H_t \doteq \sup_{s \in [0,t]} \sum_{i=1}^r \abs{U_{n,i}(s)}$. Then from \eqref{eq:exp-term-bound} and \eqref{eq:unexp},
 for any $t \in [0,T_n]$,
    \begin{equation*}
        H_t \leq H_0 + \supabs{\sqrt{n}M_{n,1}-\sqrt{2}B}{T} + 2T(\supabs{a_n - a}{T_n} + r_n) + \sum_{i=2}^r \supabs{W_{n,i}}{T_n} + \supabs{U_{n,r+1}}{T_n} + 2(1+K)\int_0^t H_s ds.
    \end{equation*}
Hence by Gronwall's lemma
    \begin{equation*}
        H_{T_n} \leq \brR*{H_0 + \supabs{\sqrt{n}M_{n,1}-\sqrt{2}B}{T} + 2T(\supabs{a_n - a}{T_n} + r_n) + \sum_{i=2}^r \supabs{W_{n,i}}{T_n} + \supabs{U_{n,r+1}}{T_n}} e^{2(1+K)T}.
    \end{equation*}
    Note $\bvec{U_{n,r+}} = \bvec{Z_{n,r+}}$ and $U_{n,i}(0) = Z_{n,i}(0) - z_i$. Then by Lemma \ref{lem:exp-setup}, $H_{T_n} \pconv 0$ and $\norm{\bvec{U_{n,r+}}}_{2,T_n} \pconv  0$. Together these show $\norm{\bvec{U}}_{2,T_n}=\norm{\Zn-\Z}_{2,T_n} \pconv 0$ as $n \to \infty$.
\end{proof}

\begin{corollary}
	\label{cor:stoptime}
    Under assumptions of Theorem \ref{thm:exponential}, $\brC*{\norm{\Zn}_{2,T}}_{n \in \nat}$ is a tight sequence of random variables and 
    \begin{equation}
        \label{eq:strong-stop-div}
        \lim_{L \to \infty} \sup_n \prob \brR*{\tau_{n,L} \leq T} = 0.
    \end{equation}
\end{corollary}
\begin{proof}
Fix  $\delta > 0$. Since $\Z$ has sample paths in $\CC([0,T]:\lSpace{2})$, we can find
 $L \in (0,\infty)$ so that 
$
    \prob\brR*{\norm{\Z}_{2, T} > L} \leq \delta.
$
Note that the events $\brC*{\norm{\Z}_{2,T} \leq L} \subseteq \brC*{\tau_{L+2} > T}$, where $\tau_{L+2} \doteq \inf\brC*{t \mid \norm{\Z(t)}_{2} > L+2}$.
Hence for each $n \in \nat$, by right continuity of $\Zn$
\begin{align*}
    \prob\brR*{\tau_{n,L+2} \leq T} &\leq \prob\brR*{\norm{\Zn}_{2,T\wedge \tau_{n,L+2}} > L + 1} \leq \prob\brR*{\norm{\Zn - \Z}_{2,T\wedge \tau_{n,L+2}} > 1 \text{ or } \norm{\Z}_{2,T} > L}\nonumber\\
                                    &\leq \prob\brR*{\norm{\Zn - \Z}_{2,T\wedge \tau_{n,L+2} \wedge \tau_{L+2}} > 1 } + \prob\brR*{ \norm{\Z}_{2,T} > L }\\
                                    &\leq \delta + \prob\brR*{\norm{\Zn-\Z}_{2,T\wedge \tau_{n,L+2} \wedge \tau_{L+2}} > 1}. 
\end{align*}
Sending $n\to \infty$ and using  Lemma \ref{lem:exp-conv} we see that $\limsup_n \prob\brR*{\tau_{n,L+2} \leq T} \leq \delta$.
Finally,
\begin{equation*}
    \limsup_n \prob\brR*{\norm{\Zn}_{2,T} > L+2} \leq \limsup_n \prob\brR*{\tau_{n,L+2} \leq T} \leq  \delta.
\end{equation*}
Since $\delta > 0$ is arbitrary, this shows that $\brC*{\norm{\Zn}_{2,T}}_{n \in \nat}$ is tight.
The convergence in \eqref{eq:strong-stop-div} is now immediate.
\end{proof} 

\begin{proof}[Proof of Theorem \ref{thm:exponential}] 
Using Lemma \ref{lem:exp-setup}	and Skorohod embedding theorem we can assume without los of generality that
$\Zn(0), \Mn$ and  $B$ are given on a common probability space,
		$\Zn(0) \to \bvec{z}$ in $\ld$, and $\Mn \to \sqrt{2} B \bvec{e_1}$ in $\DSpace*{\infty}[\lSpace{2}]$ almost surely.
		From Lemma \ref{lem:exp-conv} we now have that for every $T,L \in (0,\infty)$ \eqref{eq:zntozcg} holds.
		Finally using from Corollary \ref{cor:stoptime} the fact that
		$\sup_n \prob \brR*{\tau_{n,L} \leq T} + \prob \brR*{\tau_{L} \leq T} \to 0$
		as $L\to \infty$, we have that $\norm{\Zn - \Z}_{2,T} \pconv 0$ as $n \to \infty$.
		The result follows.
\end{proof}

\section{Proof of Theorem \ref{thm:reflection}}

\label{sec:proof-reflection}
In this section we give the proof of Theorem \ref{thm:reflection}. As for the proof of Theorem \ref{thm:exponential} we begin with a convenient representation for $\Zn$ and by establishing some useful convergence properties.
\begin{lemma}
	\label{lem:reflection-setup}
        Let $\lam_n, \alpha_n, d_n$ be as in the statement of Theorem \ref{thm:reflection}. Suppose that $\brC*{\norm{\Zn(0)}_1}_{n \in \nat}$ is a tight sequence of random variables and $\bvec{Z_{n,r+}}(0) \pconv \bvec{0}$ in $\lSpace{2}$ for some $r\chmd{\geq}2$. Then there are real stochastic processes $U_n, V_n, W_{n,i}, \eta_n$ with sample paths in $\DSpace*{\infty}[\R]$ so that, $U_n,\eta_n$ have absolutely continuous paths a.s., $U_n(0)=\eta_n(0)=0$, and for any $t \ge 0$
	\begin{align}
	Z_{n,1}(t) = Z_{n,1}(0)- &\int_0^{t} Z_{n,1}(s) ds + \int_0^{t}  Z_{n,2}(s) ds + \sqrt{n} M_{n,1}(t) + U_n(t) - \eta_n(t) 	 \label{eq:zreflection1} \\
        Z_{n,2}(t) = Z_{n,2}(0)  - &\int_0^{t} Z_{n,2}(s) ds + \int_0^t Z_{n,3}(s) ds + V_n(t)  + \eta_n(t) \label{eq:zreflection2} \\
        Z_{n,i}(t) = Z_{n,i}(0)  - &\int_0^{t} Z_{n,i}(s) ds + \int_0^t Z_{n,i+1}(s) ds + W_{n,i}(t) \quad \text{for } i \in \brC*{3, \ldots, r} \label{eq:zreflectioni}.
	\end{align}
        Furthermore, $\eta_n$ is non-decreasing process with $\eta_n(0) = 0$ that satisfies 
		\begin{equation}\label{eq:etan}
			\eta_n(t) = \int_0^t \I{Z_{n,1}(s) \geq \theta_n} d\eta_n(s) \mbox{ a.s. }
			\end{equation}
			 for some constants $\theta_n = \alpha_n + O(\sqrt{n}/d_n)\cgg{\ge 0}$ as $n\to \infty$. Also for any $L, T \in (0, \infty)$, as $n \to \infty$
	\begin{enumerate}
            \item $\sqrt{n} M_{n,1} \dconv \sqrt{2} B$ in $\DSpace{T}[\R]$
		\item $\TVar{U_n}{[0,T_{n}]} \doteq \int_0^{T_n}|\dot{U}_n(s)| ds \pconv 0$ 
                \item $\supabs{V_n}{T_{n}} \pconv 0$ and $\supabs{W_{n,i}}{T_n} \pconv 0$ for $i \in \brC*{3, \ldots, r}$
		\item $\norm{\bvec{Z_{n,r+}}}_{2, T_{n}} \pconv 0$.
	\end{enumerate}
        Here $B$ is a standard Brownian motion and $T_{n} \doteq T \wedge \tau_{n, L}$. 
\end{lemma}
\begin{proof}
    By our assumptions on $\alpha_n$, we can find $\kappa \in (0,\infty)$ be such that $\theta_n \doteq \alpha_n + \frac{\kappa \sqrt{n}}{d_n} \geq 0$ for every $n$. Note $\theta_n \to \alpha$ as $n \to \infty$. Define
    \begin{equation*}
        \begin{aligned}
            U_n(t) &\doteq -\int_0^t t_{n,1}(Z_{n,1}(s))  \I{ Z_{n,1}(s) < \theta_n}ds \\
            \eta_n(t) &\doteq \int_0^t t_{n,1}(Z_{n,1}(s))  \I{ Z_{n,1}(s) \geq \theta_n}ds
        \end{aligned}
    \end{equation*}
    so that $\eta_n(t) = \int_0^t \I{Z_{n,1}(s) \geq \theta_n} d\eta_n(s)$, and 
    \begin{equation}
     \label{eq:tn1asreflection}
     \int_0^t t_{n,1}(Z_{n,1}(s)) ds = \eta_n(t) - U_n(t).
    \end{equation}
    From Lemma \ref{lem:zn-compact} it then follows  that \eqref{eq:zreflection1} is satisfied.  Recall the expression for $t_{n,1}(z)$
	from 
	\eqref{eq:tn2}.
    Then, by monotonicity of $\beta_n$, $t_{n,1}(z) \geq 0$ whenever $z \geq 0$. Since $\theta_n \geq 0$, $\eta_n$ is non-decreasing and
    \begin{equation*}
        \begin{aligned}
            \sup_{z \leq \theta_n}\abs{t_{n,1}(z)} &\leq 2\sqrt{n} \beta_n\brR*{\lambda_n + \theta_n/\sqrt{n}} \leq 2\sqrt{n} (\lam_n + \theta_n/\sqrt{n})^{d_n} \\
                                                   &= 2 \sqrt{n} \brR*{1 - \brR*{ (\log d_n)/d_n +  (\alpha_n - \theta_n)/\sqrt{n}}}^{d_n}  = 2 \sqrt{n} \brR*{1 - (\log d_n - \kappa)/d_n}^{d_n}\\
                                                   &\leq 2 \exp \brR*{- \log \frac{d_n}{\sqrt{n}} + \kappa} \to 0 \text{ as }  n \to \infty.
        \end{aligned}
    \end{equation*}
    This shows that $\TVar{U_n}{[0,T]} \to 0$ almost surely. 

    Next, since $d_n(1-\lam_n) \to \infty$, Lemma \ref{prop:f1} shows that
 \begin{equation}
     \label{eq:munconvto1}
     \Mun \to \bvec{f_1} \in \ld \text{ as } n \to \infty.
 \end{equation}
 Therefore $\Gn(0) = \Mun + \frac{\Zn(0)}{\sqrt{n}} \to \bvec{f_1}$ in $\ld$.  Then by Lemma \ref{lem:martconv}, 
 \begin{equation}
     \label{eq:martconv1}
     \sqrt{n} \Mn \dconv \sqrt{2} B \bvec{e_1} \mbox{ in } \DSpace*{\infty}[\lSpace{2}],
 \end{equation}
  and by Lemma \ref{lem:finitedimensional},  $\norm{\bvec{Z_{n,r+}}}_{2, T \wedge \tau_{n, L}} \pconv 0$ as $n \to \infty$. Define
 \begin{equation*}
 	V_n(t) \doteq  - \int_0^t t_{n,2}(Z_{n,2}(s)) ds + \sqrt{n} M_{n,2}(t) - U_n(t).
 \end{equation*}
 Using \eqref{eq:tn1asreflection} and Lemma \ref{lem:zn-compact} once more, we see that \eqref{eq:zreflection2} is satisfied. Finally, for $i \in \brC*{3, \ldots, r}$, define
 \begin{equation*} 
     W_{n,i}(t) \doteq \int_0^t t_{n,i-1}(Z_{n,i-1}(s)) ds - \int_0^t t_{n,i}(Z_{n,i}(s)) + \sqrt{n} M_{n,i}(t).
 \end{equation*}
 Then, from Lemma \ref{lem:zn-compact} again, it follows that  \eqref{eq:zreflectioni} is satisfied with the above choice of $W_{n,i}$.  Lemma \ref{prop:taileffectszero} along with \eqref{eq:munconvto1}, \eqref{eq:martconv1}, and $\supabs{Z_{n,i}}{T \wedge \tau_{n,L}} \leq L$ show that, as $n\to \infty$, $\supabs{V_n}{T_n}\pconv 0$ and $\supabs{W_{n,i}}{T_n} \pconv 0 $ for each $i \in \brC*{3, \ldots, r}$.
\end{proof}
\begin{cor}
    \label{cor:explicit-reflection}
	Suppose that the assumptions in Lemma \ref{lem:reflection-setup} are satisfied. Assume further that $d_n \ll n^{2/3}$. Then, the conclusions of Lemma
	\ref{lem:reflection-setup} hold with $\theta_n = \alpha_n$ and
    \begin{equation}
        \label{eq:eta-exponential-form}
        \eta_n(t) \doteq \int_0^{t} \gamma_n^{-1} (1+\delta_n(s))^+ e^{\gamma_n(Z_{n,1}(s) - \alpha_n)} \I{Z_{n,1}(s) \geq \alpha_n} ds,
    \end{equation}
  where $\gamma_n \doteq \frac{d_n}{\sqrt{n}}$ and $\delta_n$ is a process with sample paths in $\DD([0,\infty), \R)$ such that $\supabs{\delta_n}{T \wedge \tau_{n,L}} \to 0$ a.s. for each $L > 0$.
\end{cor}
\begin{proof}
        Since $d_n \ll n^{2/3}$ and $\alpha_n = O(n^{1/6})$, the hypothesis of Lemma \ref{lem:t-as-exponential} is satisfied.  Define
	\begin{equation*}
		\delta_n(s) \doteq t_{n,1}(Z_{n,1}(s))\gamma_n \brR*{ e^{\gamma_n\brS*{Z_{n,1}(s) - \alpha_n}} - e^{-\gamma_n \alpha_n}}^{-1} - 1.
        \end{equation*}
            Since $\sup_{s \leq T\wedge \tau_{n, L}} \abs{Z_{n,1}(s)} \leq L$, Lemma \ref{lem:t-as-exponential} shows that $\supabs{\delta_n}{T_n} \to 0$ a.s. as $n \to \infty$. Next define
	\begin{equation*}
            \begin{aligned}
                U_n(s) &\doteq \gamma_n^{- 1}  \int_0^t (1+\delta_n(s)) \brR*{e^{-\gamma_n \alpha_n} -  e^{\gamma_n(Z_{n,1}(s) - \alpha_n )}\I{Z_{n,1}(s) < \alpha_n}} ds \\
                       &\quad +\int_0^t \gamma_n^{-1} (1+\delta_n(s))^- e^{\gamma_n (Z_{n,1}(s) - \alpha_n)}\I{Z_{n,1}(s) \geq \alpha_n} ds
            \end{aligned}
	\end{equation*}
        Then $U_{n}(0) = 0$, $U_n$ is absolutely continuous and, with $\kappa = \sup_n \frac{d_n}{\sqrt{n}}\alpha_n^- < \infty$ ,
	\begin{align*}
            \TVar{U_n}{[0,T_n]}\I{\supabs{\delta_n}{T_n} < 1} &= \gamma_n^{-1} \int_0^{T_n} \abs{1+\delta_n(s)}\abs{e^{-\gamma_n \alpha_n} -  e^{\gamma_n(Z_{n,1}(s) - \alpha_n )}\I{Z_{n,1}(s) < \alpha_n}} ds \nonumber\\
		&\leq \frac{2(1+e^\kappa)T}{\gamma_n} \to 0 \quad \text { as } n \to \infty.
	\end{align*}
        Hence, \cg{since $\supabs{\delta_n}{T_n} \to 0$, we have that} $\TVar{U_n}{[0,T_n]} \pconv 0$ as $n \to \infty$. \cg{By rearranging terms we see that, with the above definitions of $U_n$ and $\eta_n$, \eqref{eq:tn1asreflection} is satisfied. The result follows.} 
\end{proof}

Since $\gamma_n \to \infty$ and $\theta_n \to \alpha$ as $n \to \infty$, the previous lemma suggests a connection to the Skorokhod map $\Gamma_{\alpha}$ defined in \eqref{eq:SMdef}. In order to make this   connection precise, we begin with the following lemma.

\begin{lemma}
	\label{lem:never-exceeds}
        Under the assumptions of Theorem \ref{thm:reflection}, for any $L \in (0,\infty)$
		\begin{equation}
		\label{eq:never-exceeds}
                \sup_{t \in [0, T \wedge \tau_{n,L}] } (Z_{n,1}(t) - \alpha_n)^+ \pconv 0 \text{ as  } n \to \infty.
		\end{equation}
\end{lemma}
\begin{proof}
    \cg{Consider first the case $d_n \gg \sqrt{n} \log n$. For this case $\epsilon_n \doteq \frac{\sqrt{n} \log d_n}{d_n} \to 0$, and since   
        \begin{equation*}
            Z_{n,1}(t) = \sqrt{n} \brR*{G_{n,1}(t) - \lam_n} \leq \sqrt{n} (1 - \lam_n) = \frac{\sqrt{n} \log d_n}{d_n} + \alpha_n,
        \end{equation*}
		we have that \eqref{eq:never-exceeds} holds. Now consider the complementary case, namely
        $d_n = O(\sqrt{n} \log n)$. We will use
		Corollary \ref{cor:explicit-reflection}.} Since $Z_{n,1}(0) \pconv z_1 \in \R$ with $z_1 \leq \alpha$, we  have  $\brR*{Z_{n,1}(0) - \alpha_n}^+ \pconv 0$
                as $n \to \infty$.
	It now suffices to show that for any $\epsilon > 0$ 
	\begin{equation*}
            \prob\brR*{\sup_{t \in [0, T \wedge \tau_{n,L}] }Z_{n,1}(t) > \alpha_n + 6\epsilon} \to 0 \text{ as } n \to \infty.
	\end{equation*}
    Let $\tau_n \doteq \inf\set{ t\ge 0 \mid Z_{n,1}(t) > \alpha_n + 6\epsilon}$ and, as before, $T_n \doteq T \wedge \tau_{n,L}$. It is  then enough to show that $\prob\brR*{\tau_n \leq T_n} \to 0$ as $n \to \infty$. For this inductively define stopping times,
	 $\sigma_{n,0}  = 0$,
	 \begin{equation*}
	\begin{aligned}
		\sigma_{n, 2k-1} &= \inf\set{t  > \sigma_{n, 2k-2} \mid Z_{n,1}(t) > \alpha_n + 3\epsilon} \\
		\sigma_{n, 2k} &= \inf\set{t  > \sigma_{n, 2k-1} \mid Z_{n,1}(t) < \alpha_n + 2\epsilon} 
	\end{aligned},\;\;\; k \in \NN. 
	\end{equation*}
        Note that for each $n\in \NN$,  $\sigma_{n, r} \to \infty$ as $r \to \infty$, amost surely. 
		Also, henceforth, without loss of generality we consider only $n$ that are large enough so that $1/\sqrt{n} < \ep$.
		Hence on the ser $\tau_n < \infty$,  $\tau_n \in [\sigma_{n,2k-1}, \sigma_{n,2k})$ for some $k \in \nat$. Then for every $K \in \nat$
	\begin{equation*}
		\prob \brR{\tau_n \leq T_n} \leq \sum_{k=1}^{K} \prob \brR{\tau_n \in [\sigma_{n,2k-1}, \sigma_{n,2k}\wedge T_n]} + \prob\brR{\sigma_{n,2K+1} \leq T_n}.
	\end{equation*}
	Hence to complete the proof it is enough to show that
	\begin{enumerate}
		\item For each $k \in \nat$, $\lim_{n\to \infty} \prob\brR*{\tau_n \in [\sigma_{n,2k-1}, \sigma_{n, 2k} \wedge T_n]} = 0 $, 
		\item $\lim_{K \to \infty} \limsup_{n\to \infty} \prob(\sigma_{n,2K+1} \leq T_n) = 0$.
	\end{enumerate}
	Consider (1) first. Note that on the set
    $C_{n,1} \doteq \set{Z_{n,1}(0) \leq \alpha_n + 3 \epsilon}$, for any $k \in \NN$,
 	\begin{align}
 		\alpha_n + 3\epsilon \leq Z_{n,1}(\sigma_{n, 2k-1}) &= Z_{n,1}(\sigma_{n, 2k-1}-) + Z_{n,1}(\sigma_{n, 2k-1}) - Z_{n,1}(\sigma_{n, 2k-1}-) \nonumber\\
 												&\leq Z_{n,1}(\sigma_{n, 2k-1}-) + \epsilon \leq \alpha_n + 4\epsilon. \label{eq:start-at-stopped}
 	\end{align}
	Similarly, 
	\begin{equation}
		Z_{n,1}(t) \geq \alpha_n + \epsilon \quad \text{for each } t \in [\sigma_{n,2k-1}, \sigma_{n, 2k}]. \label{eq:lowerbound-stops}
	\end{equation}
	Let  $M'_n(t) \doteq \sqrt{n} M_{n,1}(t + \sigma_{n,2k-1}) - \sqrt{n} M_{n,1}(\sigma_{n,2k-1})$
	for  $t \geq 0$
	and consider  the sets
	$$C_{n,2} \doteq \set{\tau_n, \sigma_{n, 2k-1} \leq T_n}, \;\; C_{n,3} \doteq \brC*{\supabs{U_n}{T_n} \leq \epsilon/2, \abs{\delta_n}_{*,T_n} \leq \frac{1}{2}}.$$
	Then on the set $C_n = \cap_{i=1}^3 C_{n,i}$, 
        using 
		Corollary \ref{cor:explicit-reflection}, for any $t \in [0, (T_n \wedge \sigma_{n,2k}) - \sigma_{n,2k-1}]$, 
	\begin{align*}
		&Z_{n,1}(t + \sigma_{n, 2k-1}) - Z	_{n,1}(\sigma_{n, 2k-1}) \nonumber\\
		&= - \int_{\sigma_{n, 2k-1}}^{\sigma_{n, 2k-1}+t} (Z_{n,1}(s) - Z_{n,2}(s)) ds + M_n'(t) + U_n(t + \sigma_{n, 2k-1}) - U_n(\sigma_{n, 2k-1})\nonumber\\
		 &\qquad - \int_{\sigma_{n, 2k-1}}^{\sigma_{n, 2k-1}+t} \gamma_n^{-1}(1+\delta_n(s))^+ e^{\gamma_n (Z_{n,1}(s) - \alpha_n)} \I{Z_{n,1}(s) \geq \alpha_n} ds 
	\end{align*}
		Since for $t$ in the above interval $\sigma_{n,2k-1} + t \leq T_n \leq \tau_{n ,L}$, $\abs{Z_{n,1}(s)} + \abs{Z_{n,2}(s)} \leq 2L$ for any $s \leq \sigma_{n,2k-1} + t$. Also, since $\sigma_{n,2k-1} + t \leq \sigma_{n, 2k}$, by \eqref{eq:lowerbound-stops}, $Z_{n,1}(s) - \alpha_n \geq \epsilon$ for any $s \in [\sigma_{n, 2k-1}, \sigma_{n, 2k-1}+t]$. Thus on $C_n$ we have
		\begin{equation}\label{eq:znincrem}
			Z_{n,1}(t + \sigma_{n, 2k-1}) - Z_{n,1}(\sigma_{n, 2k-1}) \leq 2Lt  + M'_n(t) + \epsilon - \frac{t}{2\gamma_n} \exp\brR{\gamma_n \epsilon} \doteq Y_n(t).
\end{equation}
	Using \eqref{eq:start-at-stopped}, on $C_n$,  $Z_n(\tau_n) - Z_n(\sigma_{n, 2k-1}) \geq \alpha_n + 6\epsilon - \alpha_n - 4\epsilon = 2\epsilon$. Hence 
	\begin{align}
		\prob \brR*{\tau_n \in [\sigma_{n, 2k-1}, \sigma_{n, 2k} \wedge T_n)}
		&\leq \prob\brR*{\tau_n \in [\sigma_{n, 2k-1}, \sigma_{n, 2k} \wedge T_n), C_n }
		 + \prob\brR{C_{n,1}^c} +\prob\brR{C_{n,3}^c} \nonumber\\
		&\leq \prob\brR*{\sup_{t \in [0, T]} Y_n(t) \geq 2\epsilon } +\prob\brR{C_{n,1}^c} + \prob\brR{ C_{n,3}^c}, \label{eq:ub-tau-middle}
	\end{align}
	where the second inequality is on observing that on the set $\{\tau_n \in [\sigma_{n, 2k-1}, \sigma_{n, 2k} \wedge T_n)\}$,
	\eqref{eq:znincrem} holds with $t$ replaced by $\tau_n$.
	Next note that $M_{n}'$ is a $\{\clg^n_t\}$ martingale, where $\clg^n_t = \clf^n_{t+\sigma_{n,2k-1}}$ and
	\begin{align*}
		\QV{M_{n}'}_t &=  \QV{\sqrt{n} M_{n,1}}_{t+\sigma_{n,2k-1}} -  \QV{\sqrt{n} M_{n,1}}_{\sigma_{n,2k-1}} \\
		&= \int_{\sigma_{n,2k-1}}^{\sigma_{n,2k-1} + t} \left[G_{n,1}(s) - G_{n,2}(s) ds + \lambda_n (1 - \beta_n(G_{n,1}(s)))\right] ds \\
		&\leq 2t,
	\end{align*}
	where the second equality is from \eqref{eq:martQV}.
	
	Since $\gamma_n \to \infty$ we can apply Lemma \ref{lem:drift-exceeds} to conclude
	\begin{equation*}
		\prob \brR*{\sup_{t \in [0,T]} Y_n(t) \geq 2\epsilon} = \prob \brR*{\sup_{t \in [0,T]} \brS*{M'_n(t) - \brR*{\frac{\exp(\gamma_n \epsilon)}{2\gamma_n} - 2L}t} \geq \epsilon} \to 0
	\end{equation*}
        as $n \to \infty$. We also have $\lim_n \prob \brR*{C^c_{n,i}} = 0$ for $i = 1,3$ since, as noted earlier $(Z_{n,1}(0) - \alpha_n)^+ \pconv 0$, and by \cg{Corollary \ref{cor:explicit-reflection}}, respectively. From these observations it follows that the right side of \eqref{eq:ub-tau-middle} converges to 0 as $n \to \infty$, which completes the proof of (1)

Now we prove (2). Let $\rho_i \doteq \sigma_{n,i} \wedge \tau_{n,L}$
	and define
	\begin{align*}
		Y_{n, K} (t)  \doteq \sum_{i=0}^{K} \left(Z_{n,1}(t \wedge \rho_{n, 2i+1} ) - Z_{n,1}(t \wedge \rho_{n,2i}) \right).
	\end{align*}
	Note that
		$\brC{\sigma_{n, 2K+1} \leq T_n} \subseteq \brC{Y_{n,K}(T) \geq K\epsilon}$ and hence to prove (2) 
	it is sufficient to show that 
	\begin{equation}\label{eq:ynkttoz}
		\limsup_{n\to \infty} \prob\brR{Y_{n,K}(T) \geq K\epsilon} \to 0
		 \mbox{ as }K \to \infty.
		\end{equation} 
		From \cg{Corollary \ref{cor:explicit-reflection}}, we have that on the set $C_{n,4}\doteq \{ \TVar{U_n}{[0,T_n]} \le 1\}$,
	\begin{align}
			&\Scale[0.95]{Y_{n, K}(T)}\nonumber\\
			 &  \Scale[0.95]{=\sum_{i=0}^{K} \int_{T \wedge \rho_{n,2i} }^{T \wedge \rho_{n, 2i+1} } (Z_{n,2}(s) - Z_{n,1}(s)) ds   +\sum_{i=0}^{K} \sqrt{n}M_{n,1}(T \wedge \rho_{n,2i+1} ) - \sqrt{n}M_{n,1}(T \wedge \rho_{n,2i} )} \nonumber\\
			&\Scale[0.95]{ + \sum_{i=0}^{K}  U_{n}(T \wedge \rho_{n,2i+1} ) - U_n(T \wedge \rho_{n,2i} ) - \sum_{i=0}^K \int_{T \wedge \rho_{n,2i} }^{T \wedge \rho_{n, 2i+1} } \gamma_n^{-1}(1+\delta_n(s))^+ e^{\gamma_n (Z_{n,1}(s) - \alpha_n)} \I{Z_{n,1}(s) > \alpha_n} ds} \nonumber\\
                         &\Scale[0.95]{\leq 2LT  + \sum_{i=0}^{K} \left(\sqrt{n}M_{n,1}(T \wedge \rho_{n,2i+1} ) - \sqrt{n}M_{n,1}(T \wedge \rho_{n,2i} )\right) + \TVar{U_n}{[0,T]}}\nonumber\\
			 &\Scale[0.95]{\leq 2LT +1 + M_{n,K}'(T)}\nonumber
	\end{align}
	where we have used the facts that $\sup_{s \leq \tau_{n,L}} \abs{Z_{n,1}(s)} \leq L$, and that the rightmost term in the third line is non-positive. Also, here
	\begin{equation*}
		M'_{n, K}(t) \doteq \sum_{i=0}^{K} \left(\sqrt{n}M_{n,1}(t \wedge \rho_{n,2i+1} ) - \sqrt{n}M_{n,1}(t \wedge \rho_{n,2i} )\right).
	\end{equation*}
        Using \eqref{eq:martQV}, we see that $M'_{n,K}$ is a $\clf^n_t$-martingale with quadratic variation given by
	\begin{equation*}
	\begin{aligned}
		\QV{M_{n,K}'}_t  &= \sum_{i=0}^{K} \left(\QV{\sqrt{n}M_{n,1}}_{t \wedge \rho_{n,2i+1}} - \QV{\sqrt{n}M_{n,1}}_{t \wedge \rho_{n,2i} }\right)  \\
		&= \sum_{i=0}^K  \int_{t \wedge \rho_{n,2i} }^{t \wedge \rho_{n, 2i+1} } \left(G_{n,1}(s) - G_{n,2}(s) + \lambda_n  - \lambda_n \beta_n(G_{n,1}(s))\right) ds \\
		&\leq 2t.
	\end{aligned} 
	\end{equation*}
	Hence 
	\begin{align*}
	\prob \brR*{Y_{n,K}(T) \geq K\epsilon} 
	&\le \prob \brR*{Y_{n,K}(T) \geq K\epsilon,  C_{n,4}} +  \prob \brR*{C_{n,4}^c}\\
        &\leq \prob \brR*{M_{n,K}'(T) > K\epsilon - (2LT+1)} +  \prob \brR*{C_{n,4}^c} \nonumber\\
	&\le  \frac{\E M_{n,K}^{\prime2}(T)}{(K\epsilon - (2LT+1))^2} +  \prob \brR*{C_{n,4}^c} \nonumber\\
	&\leq  \frac{2T}{(K\epsilon - (2LT+1))^2} +  \prob \brR*{C_{n,4}^c}.
	\end{align*}
	From \cg{Corollary \ref{cor:explicit-reflection}}, $\prob \brR*{C_{n,4}^c} \to 0$ as $n\to \infty$.
	This together with the above display shows $\lim_{K \to \infty} \limsup_{n\to \infty} 	\prob \brR*{Y_{n,K}(T) \geq K\epsilon}  = 0$. Thus we have shown \eqref{eq:ynkttoz} and the proof of (2) is complete. The result follows.
\end{proof}

\begin{lemma}
	\label{lem:as-reflection}
        Suppose the hypothesis of Theorem \ref{thm:reflection} holds, then for each $n \in \nat$, there is a real constant $\theta_n = \alpha_n + O(\sqrt{n}/d_n) \geq 0$ and processes $U'_n, V_n$ with sample paths in $\DD([0, \infty):\R)$ such that with $\tilde{Z}_{n,1} \doteq Z_{n,1} \wedge \theta_n$ 
    \begin{align}
        \tilde{Z}_{n,1}(t) &= \Gamma_{\theta_n}\brR*{\tilde{Z}_{n,1}(0) - \int_0^{\cdot} \brR*{\tilde{Z}_{n,1}(s) - Z_{n,2}(s)} + \sqrt{n}M_{n,1}(\cdot) + U_n'(\cdot)}(t), \text{ and } \label{eq:zn1tilde}\\
    Z_{n,2}(t) &= Z_{n,2}(0) - \int_0^t \brR*{Z_{n,2}(s) - Z_{n,3}(s)} ds + V_n(t) + \eta_n(t) \qquad \mbox{ for all } t > 0\nonumber,
    \end{align}
    where
	\begin{equation}
\label{eq:etadef}
		\eta_n = \hat \Gamma_{\theta_n}\brR*{\tilde{Z}_{n,1}(0) - \int_0^{\cdot} \brR*{\tilde{Z}_{n,1}(s) - Z_{n,2}(s)} + \sqrt{n}M_{n,1}(\cdot) + U_n'(\cdot)}.
		\end{equation}
               Furthermore, for any $L, T \in (0,\infty)$, $\supabs{U_n'}{T \wedge \tau_{n,L}}$, $\supabs{V_n}{T \wedge \tau_{n,L}}$ and $\supabs{(Z_{n,1}-\theta_n)^+}{T \wedge \tau_{n,L}} \pconv 0$ as $n \to \infty$.
\end{lemma}
\begin{proof}
    Let $\theta_n$ be as in Lemma \ref{lem:reflection-setup}.
    Since $d_n \gg \sqrt{n}$, \chmd{$\theta_n = \alpha_n + o(1)$ and}  Lemma \ref{lem:never-exceeds} shows 
\begin{equation}\supabs{(Z_{n,1} - \theta_n)^+}{T \wedge \tau_{n,L}} \to 0.\label{eq:eqznith}
	\end{equation}  Note that $\tilde{Z}_{n,1} = Z_{n,1} - (Z_{n,1} - \theta_n)^+$. Hence we can rewrite \eqref{eq:zreflection1} and \eqref{eq:zreflection2}  as
	        \begin{align}
	            \tilde{Z}_{n,1}(t) &= \tilde{Z}_{n,1}(0) - \int_0^t \tilde{Z}_{n,1}(s) ds + \int_0^t Z_{n,2}(s) ds + \sqrt{n}M_{n,1}(t) + U'_n(t) - \eta_n(t) \label{eq:zntildeform}\\
	            Z_{n,2}(t) &= Z_{n,2}(0) - \int_0^t (Z_{n,2}(s) - Z_{n,3}(s)) ds + V_n(t) + \eta_n(t) \nonumber,
	\end{align}
	       where
	        \begin{equation}
	            \begin{aligned}
	            U_n'(t) &\doteq U_n(t) - \int_0^t (Z_{n,1}(s) - \theta_n)^+ ds - (Z_{n,1}(t)-\theta_n)^+ + (Z_{n,1}(0) - \theta_n)^+\nonumber\\
	        \end{aligned} \label{eq:auxfunctions}.
	        \end{equation}
                The properties of $\eta_n$ from Lemma \ref{lem:reflection-setup} (and \cg{Corollary \ref{cor:explicit-reflection}}) say that  $\eta_n$ is a non-decreasing process, with $\eta_n(0) = 0$ and $\eta_n(t) = \int_0^t \I{\tilde{Z}_{n,1}(s) = \theta_n} d\eta_n(s)$. Since $\tilde{Z}_{n,1} \leq \theta_n$, \eqref{eq:zntildeform} and the characterizing properties of the Skorokhod map  show \eqref{eq:zn1tilde} and \eqref{eq:etadef}.
                Finally, by Lemma \ref{lem:reflection-setup}, Corollary \ref{cor:explicit-reflection},  and Lemma \ref{lem:never-exceeds}
                \begin{equation*}
        \supabs{U_n}{T \wedge \tau_{n,L}}  \pconv 0, \mbox{ and } \supabs{V_n}{T \wedge \tau_{n,L}}  \pconv 0
                \end{equation*}
                as $n \to \infty$.  Hence, \cg{using \eqref{eq:eqznith}}, $\supabs{U'_n}{T \wedge \tau_{n,L}} \pconv 0$ as $n \to \infty$, and the result follows.
\end{proof}
\cg{The following lemma will be needed in order to prove the tightness of $\Zn$.}
\begin{lemma}
	Under the hypothesis of Theorem \ref{thm:reflection}, the collection of random variables $\set{\norm{\Zn}_{2,T}}_{n \in \nat}$ is tight for any $T \in (0,\infty)$.
	\label{lem:process-bounded}
\end{lemma}
\begin{proof}
	Fix $T\in (0,\infty)$. 
    \chmd{In Lemma \ref{lem:as-reflection} using the definition of the Skorokhod map $\Gamma_{\theta_n}$ for $\theta_n \geq 0$}  (see \eqref{eq:SMdef}), we have,  for any $t > 0$ that
    \begin{equation*}
        \eta_n(t) \leq  \abs{\tilde{Z}_{n,1}(0)} + \int_0^t \abs{\tilde{Z}_{n,1}(s)} ds + \int_0^t \abs{Z_{n,2}(s)} ds + \supabs{\sqrt{n}M_{n,1}}{t} + \supabs{U'_n}{t}.
        \label{eq:etabound}
    \end{equation*}
    This shows that for any $t \geq 0$ 
    \begin{align*}
        \supabs{\tilde{Z}_{n,1}}{t} &\leq  2\brR*{\abs{\tilde{Z}_{n,1}(0)} + \int_0^t \supabs{\tilde{Z}_{n,1}}{s} ds + \int_0^t \supabs{Z_{n,2}}{s} ds + \supabs{\sqrt{n}M_{n,1}}{t} + \supabs{U'_n}{t}}\\
        \supabs{Z_{n,2}}{t} &\leq  \abs{\tilde{Z}_{n,1}(0)} + \abs{Z_{n,2}(0)}+ \int_0^t \supabs{\tilde{Z}_{n,1}}{s} ds + 
		\int_0^t  (2\supabs{Z_{n,2}}{s} + \supabs{Z_{n,3}}{s}) ds\nonumber\\ 
                            &\quad + \supabs{\sqrt{n}M_{n,1}}{t}  + \supabs{U'_n}{t} + \supabs{V_n}{t}, 
                            \shortintertext{and}
        \supabs{Z_{n,i}}{t} &\leq \abs{Z_{n,i}(0)} + \int_0^t \supabs{Z_{n,i}}{s} ds + \int_0^t \supabs{Z_{n,i+1}}{s} ds + \supabs{W_{n,i}}{t} \quad \text{for }  i \in \{3, \ldots, r\} 
    \end{align*}
	\cg{where the last line is from Lemma \ref{lem:reflection-setup}.}
    Let $H_t \doteq \supabs{\tilde{Z}_{n,1}}{t} + \supabs{Z_{n,2}}{t} + \ldots + \supabs{Z_{n,r}}{t}$. By adding over equations in the above display, we have for  $t \in [0, \tau]$ and $\tau \in [0,T]$ that
    \begin{equation*}
        0 \leq H_t \leq 4\brR*{H_0 +  \supabs{\sqrt{n} M_{n,1}}{\tau} + \supabs{U_n'}{\tau} + \supabs{V_n}{\tau} + \sum_{i=3}^r  \supabs{W_{n,i}}{\tau} + \int_0^t H_s ds}.
    \end{equation*}
   By Gronwall's inequality, for all $\tau \in [0,T]$,
    \begin{equation}
        \label{eq:gronwallbound}
        H_\tau \leq 4\brR*{H_0 +  \supabs{\sqrt{n} M_{n,1}}{\tau} + \supabs{U_n'}{\tau} + \supabs{V_n}{\tau} + \sum_{i=3}^r  \supabs{W_{n,i}}{\tau}}e^{4 \tau}.
    \end{equation}
    Let $\vec{Z}_n \doteq (\tilde{Z}_{n,1}, Z_{n,2}, \ldots, Z_{n,r})$. Since $\vec{Z}_n(0) \pconv (z_1, \ldots, z_r)$, and $\sqrt{n}\Mn \dconv B\bvec{e_1}$, for every $\ep>0$ there is a $L_1 \in (0,\infty)$ such that for 
	every $n\in \NN$
    \begin{equation*}
        \prob \brR*{C_{n,1}} \leq \frac{\epsilon}{2}, \text{ where } C_{n,1} \doteq \brC*{H_0 +  \supabs{\sqrt{n}M_{n,1}}{T} \geq L_1}.
    \end{equation*}
    Applying Lemmas \ref{lem:reflection-setup} and Lemma \ref{lem:as-reflection} with  $L = 4(L_1+1)e^{4T} + 2$ we can find an $n_0 \in \nat$ so that $\prob \brR*{C_{n,2}} \leq \frac{\epsilon}{2}$ for  $n \geq n_0$, where
    \begin{equation*}
        C_{n,2} \doteq \brC*{\supabs{U_n'}{T_n} + \supabs{V_n}{T_n} + \sum_{i=3}^r  \supabs{W_{n,i}}{T_n} + \supabs{(Z_{n,1}-\alpha_n)^+}{T_n} + \norm{\bvec{Z_{n,r+}}}_{2, T_n}\geq 1}
    \end{equation*}
    and $T_n \doteq T \wedge \tau_{n,L}$.
On the event $\brR*{C_{n,1} \cup C_{n,2}}^c$
\begin{equation*}
    \norm{\vec{Z}_n}_{1,T_n} = H_{T_n} < 4(L_1 + 1)e^{4T}
\end{equation*}
by \eqref{eq:gronwallbound}, and hence by triangle inequality 
\begin{equation}
    \begin{aligned}
         \norm{\Zn}_{2,T_n} &\leq \norm{\vec{Z}_n}_{1,T_n} + \supabs{\brR*{Z_{n,1}-\alpha}^+}{T_n}  + \norm{\bvec{Z_{n,r+}}}_{2,T_n} \\
                            &< 4(L_1+1)e^{4T} + 1 = L - 1.
    \end{aligned}
    \label{eq:recursivebound}
\end{equation}
Also, by the definition of $\tau_{n,L}$, $\norm{\Zn(\tau_{n,L})}_2 \geq L - \frac{1}{\sqrt{n}}$ on the set $\tau_{n,L} < T$. Hence we must have that $\tau_{n,L} > T$ whenever \eqref{eq:recursivebound} holds, and hence
\begin{equation*}
    \norm{\Zn}_{2,T} < L - 1
 \mbox{ on the event } \brR*{C_{n,1} \cup C_{n,2}}^c.
\end{equation*}
This shows that
\begin{equation*}
    \prob\brR*{\norm{\Zn}_{2,T} \geq L} \leq \prob\brR*{C_{n,1} \cup C_{n,2}} \leq \epsilon \quad \forall n \geq n_0
\end{equation*}
Since $\epsilon > 0$ is arbitrary, the result follows.
\end{proof}
The following result is immediate from \cg{Lemmas \ref{lem:stopping-time-trick}, \ref{lem:reflection-setup}, \ref{lem:as-reflection}, and \ref{lem:process-bounded}}.
\begin{corollary}
    \label{cor:uniform-conv}
    Under the hypothesis of Theorem \ref{thm:reflection}, for any $T > 0$,
 $\lim_{L \to \infty} \sup_n \prob\brR*{\tau_{n,L} \leq T} = 0$.
    In particular the processes $W_{n,i}, U'_n, V_n, \norm{\bvec{Z_{n,r+}}}_2, (Z_{n,1}-\alpha_n)^+, (Z_{n,1}-\theta_n)^+$  converge in probability to  zero  in $\DSpace*{\infty}[\R]$.
\end{corollary}

\begin{corollary}
    Under the hypothesis of Theorem \ref{thm:reflection}, the sequence of processes $\brC*{\Zn}_{n \in \nat}$ is tight in $\DSpace*{\infty}[\lSpace{2}]$.
	\label{cor:tightness}
\end{corollary}
\begin{proof}
    Let $\theta_n$ be as in Lemma \ref{lem:as-reflection}. Then by Corollary \ref{cor:uniform-conv}, \cg{for each fixed $T<\infty$,} $\norm{\bvec{Z_{n,r+}}}_{2,T}  \pconv 0$ and $\supabs{\brR*{Z_{n,1} - \theta_n}^+}{T} \pconv 0$. Hence it is sufficient to show that the sequence $\brC*{\vec{Z}_n}_{n \in \nat}$ introduced in the proof of
	Lemma \ref{lem:process-bounded} is tight in $\DSpace{T}[\R^r]$.
	From Lemma \ref{lem:process-bounded}, the convergence of $W_{n,i}$ in Corollary \ref{cor:uniform-conv}, and equations for
	$Z_{n,j}$, $j= 3, \ldots r$ in Lemma \ref{lem:reflection-setup}, it is immediate that $(Z_{n,3}, \ldots Z_{n,r})$ is tight in $\DSpace*{\infty}[\R^{r-2}]$.  Finally consider the pair $(\tilde Z_{n,1}, Z_{n,2})$. Once again using Lemma  \ref{lem:process-bounded}, the convergence of 
	$\sqrt{n}M_{n,1}$ in Lemma \ref{lem:reflection-setup}, and the convergence of
	$U'_n$ in Corollary \ref{cor:uniform-conv}, it follows that
	\begin{equation}\label{eq:wn}
		R_n \doteq \tilde{Z}_{n,1}(0) - \int_0^{\cdot} \brR*{\tilde{Z}_{n,1}(s) - Z_{n,2}(s)} + \sqrt{n}M_{n,1}(\cdot) + U_n'(\cdot)
		\end{equation}
	is tight in $\DSpace*{\infty}[\R]$.
	Using the identity
	$$\Gamma_{\theta_n}(R_n)(t) = \Gamma_{\theta_n}\left(\Gamma_{\theta_n}(R_n)(s) + R_n(\cdot+s) - R_n(s)\right)(t-s)$$
	for $0\le s \le t \le T$, we see from the definition of the Skorohod map that
	$$|\Gamma_{\theta_n}(R_n)(t)-\Gamma_{\theta_n}(R_n)(s)| \le 2 \sup_{s\le u \le t} |R_n(u)-R_n(s)|.$$
	Together with the tightness of
	$R_n$ this immediately implies the tightness of $\tilde Z_{n,1} = \Gamma_{\theta_n}(R_n)$ and of $\hat \Gamma_{\theta_n}(R_n)$. Finally the tightness of $Z_{n,2}$ is now immediate from Lemma \ref{lem:process-bounded}, the convergence of $V_n$ in Corollary \ref{cor:uniform-conv} and the  tightness of $\hat \Gamma_{\theta_n}(R_n)$ noted above.
	The result follows.
\end{proof}

\begin{proof}[Proof of Theorem \ref{thm:reflection}]
	From Lemma \ref{prop:f1} and from the tightness of $\brC*{\norm{\Zn(0)}_1}_{n \in \nat}$, it follows under the conditions of the theorem that $\Mu_n \pconv \bvec{f_1}$
	and $\Gn(0) \pconv \bvec{f_1}$ in $\ld$. This proves the first statement in the theorem. Now consider the second statement. 
	Fix $T<\infty$.
	From Corollary \ref{cor:tightness}, $\brC*{\Zn}_{n \in \nat}$ is tight in $\DSpace*{\infty}[\lSpace{2}]$. Also from Lemma \ref{lem:reflection-setup}, $\sqrt{n}M_{n,1}$ converges in distribution to $\sqrt{2}B$ where $B$ is a standard Brownian motion
	and from Corollary \ref{cor:uniform-conv}
	$$(\{W_{n,i}\}_{i=3}^r, U'_n, V_n,  (Z_{n,1}-\theta_n)^+) \pconv \mathbf{0} \mbox{ in } \DSpace{T}[\RR^r]$$
	Suppose that along a subsequence
	$$(\Zn , \sqrt{n}M_{n,1}, \{W_{n,i}\}_{i=3}^r, U'_n, V_n,  (Z_{n,1}-\theta_n)^+) \dconv (\Z , \sqrt{2}B, \mathbf{0})$$
	 in $\DSpace*{\infty}[\lSpace{2}\times \RR^{r+2}]$ and for notational simplicity label the subsequence once more as $\{n\}$. Also by appealing to Skorohod embedding theorem we assume that all the  processes in the above display 
	are given on a common probability space and the above convergence holds a.s.
Since $J(\Zn) \leq \frac{1}{\sqrt{n}}$ and $\Zn(0) \pconv \bvec{z}$, we have $J(\Z) = 0$ and  $\Z(0) = \bvec{z}$ almost surely. In particular $\Z$ has sample paths in  $\CSpace*{\infty}[\lSpace{2}]$ and $(\Zn , \sqrt{n}M_{n,1}) \to (\Z , \sqrt{2}B)$ uniformly over compact time intervals
in $\lSpace{2}\times \RR$.
Since by  Corollary \ref{cor:uniform-conv}, for every $T<\infty$, $\norm{\bvec{Z_{n,r+}}}_{2,T}  \pconv 0$, it suffices to show that
$(Z_1, \ldots , Z_r)$ along with $B$ satisfy \eqref{eq:z1zi}.

From the  equations of $(Z_{n,3}, \ldots Z_{n,r})$ in Lemma \ref{lem:reflection-setup}, uniform convergence of $\Zn$ to $\Z$, and the uniform convergence of $W_{n,i}$ to $0$, it is immediate that 
$(Z_3, \ldots , Z_r)$ satisfy \eqref{eq:z1zi}. Finally consider the equations for $(Z_1, Z_2)$.
From \eqref{eq:wn} and uniform convergence properties observed above it is immediate that $R_n$ converges uniformly, a.s., to
$R$ given as
$$
R(\cdot) =  Z_{1}(0) - \int_0^{\cdot} \brR*{Z_{1}(s) - Z_{2}(s)} + \sqrt{2}B(\cdot). $$
Since $\theta_n = \alpha_n + O(\sqrt{n}/d_n) \to \alpha$, this shows that, for every $T<\infty$,
\begin{align*}
	\Gamma_{\theta_n}(R_n)(t) &=  R_n(t) - \sup_{s \in [0,t]} (R_n(t)-\theta_n)^+\\
	& \to R(t) - \sup_{s \in [0,t]} (R(t)-\alpha)^+ = \Gamma_{\alpha}(R)(t)
\end{align*}
uniformly for $t \in [0,T]$, a.s., where $(R(t)-\alpha)^+$ is taken to be $0$ when $\alpha = \infty$.
Similarly,
$$\hat \Gamma_{\theta_n}(R_n)(t) \to \hat \Gamma_{\alpha}(R)(t)$$
uniformly for $t \in [0,T]$, a.s. Here, when $\alpha = \infty$, $\Gamma_{\alpha}$ and $\hat \Gamma_{\alpha}$ are as introduced in \eqref{eq:skorinf}.
The fact that $(Z_1, Z_2)$ solve the first two equations in \eqref{eq:z1zi} is now immediate from Lemma \ref{lem:as-reflection}, \cg{the convergence
$\tilde Z_{n,1}- Z_{n,1}\pconv 0$,}
and the uniform convergence of $V_n$ to $0$ noted previously.
The result follows.
\end{proof}

\section*{Acknowledgements}
Research of SB is supported in part by NSF grants DMS-1613072, DMS-1606839 and ARO grant W911NF-17-1-0010. 
Research  of AB is supported in part by the National Science Foundation (DMS-1814894 and DMS-1853968).
Research of MD is supported by the NSF grant DMS-1613072 and NIH R01 grant HG009125-01.

\bibliographystyle{plain}
\bibliography{refs}

 \appendix

\section{Proofs of results in Section \ref{sec:technical-estimates}}
\label{app:proof-tech-est}
\subsection{Proof of Lemma \ref{lem:betan-approx}:}

\begin{proof}
	Fix $\ep \in (0,1)$. 
	First suppose $\frac{d_n}{n} \to 0$. Consider $x \in (\eps, 1]$.
	Let $\Delta_n(x) \doteq \log \beta_n(x) - \log \gamma_n(x)$.
	Let $n_0 \in \NN$ be such that for all $n \ge n_0$, $d_n/n < \ep/2$. Then, for $n\ge n_0$,
	\begin{align}
	\Delta_n(x) &\doteq \sum_{i=0}^{d_n - 1} \log\brR*{\frac{x - i/n}{1 - i/n}} - \log x^{d_n} = \sum_{i=0}^{d_n - 1} \brC*{\log\brR*{\frac{x - i/n}{1 - i/n}} -  \log x}, \nonumber\\
	&= \sum_{i=0}^{d_n - 1} \log \brR*{\frac{1 - i/(nx)}{1 - i/n}} = \sum_{i=0}^{d_n - 1} \log \brR*{ 1 - (i/n)\frac{1/x - 1}{1-i/n}}. \label{eq:diff}
	\end{align}
	Differentiating $\Delta_n$ gives, 
	\begin{align*}
	\Delta_n'(x) &= \sum_{i=0}^{d_n - 1}\brR*{\frac{1}{x - i/n} - \frac{1}{x}} =  \sum_{i=0}^{d_n-1} \frac{i/n}{x(x - i/n)}.
	\end{align*}
	Since $n\ge n_0$  and $x \in [\epsilon, 1]$ we have $x(x-\frac{i}{n}) \geq \epsilon^2/2$ for $i \le d_n-1$. Hence,
	$$\abs{\Delta'_n(x)} \leq \frac{2}{\epsilon^2} \sum_{i=0}^{d_n - 1} (i/n) \leq \frac{1}{\epsilon^2} \frac{d_n^2}{n}.$$ 
	From the definition of $\Delta_n$, we also have,
	\begin{equation}
	\Delta'_n(x) = \frac{\beta_n'(x)}{\beta_n(x)} - \frac{\gamma_n'(x)}{\gamma_n(x)} = \frac{\gamma'_n(x)}{\gamma_n(x)} \brR*{\frac{\beta'_n(x)}{\gamma'_n(x)} \frac{\gamma_n(x)}{\beta_n(x)} - 1}. \label{eq:delta-prime-1}
	\end{equation}
	Since $\frac{\gamma_n'(x)}{\gamma_n(x)} = \frac{d_n}{x} \geq d_n$ for $x \in [\epsilon,1]$, from \eqref{eq:delta-prime-1} we have,
	\begin{equation*}
	\sup_{x \in [\epsilon, 1]} \abs{\frac{\beta'_n(x)}{\gamma'_n(x)} \frac{\gamma_n(x)}{\beta_n(x)} - 1} \leq \frac{1}{d_n} \sup_{x \in [\epsilon, 1]}\abs{\Delta_n'(x)} \leq \frac{1}{\epsilon} \frac{d_n}{n} \to 0.
	\end{equation*}
	This proves \eqref{eq:betan-ratio-approx}.
	
	 Now assume $\frac{d_n}{\sqrt{n}} \to 0$. Once more consider $x \in (\eps, 1]$ and $n\ge n_0$.
	 Let $C \doteq \sup_{n\ge n_0} \frac{1/\epsilon - 1}{1 - d_n/n} < \infty$ and let $n_1>n_0$ be such that
	 $d_nC/n < 1/2$ for all $n\ge n_1$.
	  Then for $n \ge n_1$ and $x \in [\epsilon, 1]$:
	\begin{align}
	\abs{\Delta_n(x)} &\leq  \sum_{i=0}^{d_n - 1} 2\abs{(i/n)\frac{1/x - 1}{1 - i/n}} \leq 2C  \sum_{i=0}^{d_n - 1} i/n \leq C \frac{d_n^2}{n},
        \label{eq:deltabound}
	\end{align} 
	where the first inequality is from \eqref{eq:diff} and the inequality $\abs{\log (1+h)} \leq 2\abs{h}$ for $\abs{h} \leq 1/2$.
	This shows $\sup_{x \in [\epsilon,1]} \abs{\Delta_n(x)} \to 0$, hence showing the first convergence in  \eqref{eq:betan-prime-approx}. Finally the second convergence \eqref{eq:betan-prime-approx} is immediate on combining the first convergence with \eqref{eq:betan-ratio-approx}.
\end{proof}

\subsection{Proof of Corollary \ref{cor:betanbounds}:}

This is an immediate consequence of the estimate in \eqref{eq:deltabound}.

\subsection{Proof of Corollary \ref{cor:betaprimemun}:}

\begin{proof}
    Let $\ep>0$ and $n_0 \in \NN$ be such that
	 $\mu_{n,i} > \epsilon$ for all $n \ge n_0$. By Lemma \ref{lem:betan-approx}, as $n \to \infty$
    \begin{align}
        \frac{\beta'_n(\mu_{n,i})}{\beta_n(\mu_{n,i})}  = (1 + o(1)) \frac{\gamma'_n(\mu_{n,i})}{\gamma_n(\mu_{n,i})}  \label{eq:ratiobeta}.
    \end{align}
    Recall that $\mu_{n,i+1} \doteq \lambda_n \beta_n(\mu_{n,i})$ and $\gamma_n(x) \doteq x^{d_n}$. Hence \eqref{eq:ratiobeta} gives 
    \begin{equation}
        \frac{\beta'_n(\mu_{n,i})}{\mu_{n,i+1}/\lambda_n} = (1 + o(1)) \frac{d_n}{\mu_{n,i}}
    \end{equation}
    completing the proof.
\end{proof}

\subsection{Proof of Lemma \ref{lem:fixedpointapprox}:} 

\begin{proof}
From  Corollary \ref{cor:betanbounds}, there is a $n_0 \in \NN$ and $C \in (0,\infty)$ such that for all $n\ge n_0$
$$\sup_{x\in [\epsilon, 1]} \abs{\log \beta_n(x) - \log \gamma_n(x)} \le \frac{Cd_n^2}{ n}.$$
Thus, if for $n\ge n_0$ and $i \in \NN$,  $\mu_{n,i} \geq \epsilon$, then 
    \begin{align}
        \log \mu_{n,i+1} &= \log \lambda_n + \log \beta_n(\mu_{n,i}) = \log \lambda_n + \log \gamma_n(\mu_{n,i}) + \gamma_{n,i}\nonumber\\
                         &= \log \lambda_n + d_n \log \mu_{n,i} +\gamma_{n,i}, \label{eq:mulogrec}
    \end{align}
	where $|\gamma_{n,i}| \le \frac{Cd_n^2}{n}$.
	Now let $k \in \NN$ and $n_1 \in \NN$ be such that for all $n \ge n_1$, $\mu_{n,k} \ge \ep$.
    We will show that for $n \ge n_0\vee n_1$ and $j \in \set{1, \ldots, k}$ that
    \begin{equation}
        \log \mu_{n,j+1} = \brR*{\log \lambda_n } \brR*{\sum_{i=0}^j d_n^i} + \beta_{n,j},
        \label{eq:mu-ind}
    \end{equation}
where $|\beta_{n,j}|\le \frac{C}{n}\sum_{i=1}^{j} d_n^{i+1}$.
Note the the lemma is immediate from \eqref{eq:mu-ind} on taking $j=k$. To prove \eqref{eq:mu-ind} we argue inductively.
    First note that since $\Mun \in \ld$, $\mu_{n,i} \geq \mu_{n,k} \geq \epsilon$ for each $i \leq k$ and $n \ge n_1$. Hence \eqref{eq:mulogrec} holds for each $i \leq k$ and $n \ge n_0\vee n_1$.
    Taking $i=1$ in \eqref{eq:mulogrec} and noting that $\mu_{n,1} = \lambda_n$ proves \eqref{eq:mu-ind} for the  case  $j=1$. 

    Suppose now \eqref{eq:mu-ind} holds for some $j \leq k-1$. Then, using $i = j+1$, in \eqref{eq:mulogrec}
    $$
        \log \mu_{n,j+2} = \log \lambda_n + d_n \log \mu_{n,j+1} + \gamma_{n,j+1},$$
		where $|\gamma_{n,j+1}| \le \frac{Cd_n^2}{ n}$.
       By the induction hypothesis, \eqref{eq:mu-ind} holds for $j$. Hence
        \begin{align*}
		\log \mu_{n,j+2} &= \log \lambda_n + d_n \brC*{(\log \lambda_n) \brR*{\sum_{i=0}^{j} d_n^i} + \beta_{n,j}} = (\log \lambda_n) \brR*{\sum_{i=0}^{j+1} d_n^i} + d_n \beta_{n,j} + \gamma_{n,j} \nonumber
    \end{align*}
	and hence $\beta_{n,j+1} = d_n \beta_{n,j} + \gamma_{n,j}$. This shows
	$$|\beta_{n,j+1}| = |d_n \beta_{n,j} + \gamma_{n,j}| \le d_n \frac{C}{n}\sum_{i=1}^{j} d_n^{i+1} + \frac{Cd_n^2}{n}
	= \frac{C}{n}\sum_{i=1}^{j+1} d_n^{i+1}$$
    which shows that \eqref{eq:mu-ind} holds for $j+1$. This completes the proof.
\end{proof}

\subsection{Proof of Corollary \ref{cor:diffusionconditions}:}

\begin{proof}
    Since $d_n \to \infty$, the assumption $\frac{\xi_n^2}{d_n} \to 0$ shows that $\frac{\abs{\xi_n}}{d_n} \leq \frac{1 + \xi_n^2}{d_n} \to 0$. This shows that $\epsilon_n \doteq 1 -  \lambda_n = \frac{\xi_n + \log d_n}{d_n^k}$ also converges to $0$. 

    We first show that  $\mu_{n,i} \to 1$ for each $i \in \set{1, \ldots, k}$. We will argue inductively.
	Since $\mu_{n,1} \doteq \lambda_n = 1 - \epsilon_n$, we have  $\mu_{n,1} \to 1$. Suppose now that $\mu_{n,i} \to 1$ for some $i \leq k-1$. Hence eventually $\mu_{n,i} \geq \frac{1}{2}$.  Applying Lemma \ref{lem:fixedpointapprox} with $k = i$ and $\epsilon=\frac{1}{2}$
	and simplifying the resulting expression,  we get
    \begin{align}
        \log \mu_{n, i+1} &= (\log \lambda_n) \frac{d_n^{i+1} - 1}{d_n - 1} + O\brR*{\frac{d^2_n(d_n^i - 1)}{n(d_n - 1)}}\label{eq:muasymptoticspr}\\        
          &= O(\epsilon_n) \frac{d_n^{i+1} - 1}{d_n - 1} + O\brR*{\frac{d^2_n(d_n^i - 1)}{n(d_n - 1)}} = O\brR*{\frac{\xi_n + \log d_n}{d_n^{k-i}}} + O\brR*{\frac{d_n^{i+1}}{n}}, \label{eq:muasymptotics}
    \end{align}
	where the second equality uses
	 $\log \lambda_n = \log (1 - \epsilon_n) = O(\epsilon_n)$ and the third follows on recalling that  $d_n \to \infty$.
    Since $i \leq k-1$,  $\frac{\abs{\xi_n}}{d_n^{k-i}} \leq \frac{1+\xi_n^2}{d_n} \to 0$. Using this along with   $d_n^{k+1} \ll n$ in \eqref{eq:muasymptotics} shows that $\mu_{n,i+1} \to 1$. Hence, by induction,  $\mu_{n,i} \to 1$ for  $i \leq k$.

    Next we argue that $\beta_n'(\mu_{n,k}) \to \alpha$. Since $\lam_n\to 1$ and $\mu_{n,k}\to 1$,  from Corollary \ref{cor:betaprimemun} we have that
    \begin{align*}
        \lim_{n\to \infty} \frac{\beta'_n(\mu_{n,k})}{d_n \mu_{n,k+1}} = 1
    \end{align*}
      Hence it suffices to show that $d_n \mu_{n,k+1} \to \alpha$. For this note that
   \begin{align*}
       \log (d_n \mu_{n,k+1}) &= \log \mu_{n,k+1} + \log d_n \\
        &= \log (1-\epsilon_n) \brR*{\frac{d_n^{k+1} - 1}{d_n - 1}} + O\brR*{\frac{d_n^2}{n} \frac{d_n^k - 1}{d_n - 1}} + \log d_n \\
                    &= (-\epsilon_n + O(\epsilon_n^2)) d_n^k (1 + O(1/d_n)) + \log d_n + O\brR*{\frac{d_n^{k+1}}{n}}, 
	\end{align*}
							 where the second equality is from \eqref{eq:muasymptoticspr}
							 and last equality is by using Taylor's expansion for $\log(1-\ep_n)$.
Using  $d_n^{k+1} \ll n$ and $\abs{\epsilon_n^2d_n^k}\leq \frac{2(\xi_n^2 + (\log d_n)^2)}{d_n^k} \to 0$, we now have
\begin{align*}
  \log (d_n \mu_{n,k+1}) &= (-\epsilon_n d_n^k + o(1)) (1 + O(1/d_n)) + \log d_n + o\brR*{1} \nonumber\\
                           &= (-\xi_n - \log d_n) (1 + O(1/d_n)) + \log d_n + o(1) \nonumber\\
                           &= -\xi_n - \log d_n + \log d_n +  O\brR*{\frac{\xi_n + \log d_n}{d_n}} + o(1)  = -\xi_n + o(1) \to \log(\alpha) \nonumber
\end{align*}					 
	where the last equality once more uses the observation that						 
   $\frac{\abs{\xi_n}}{d_n}  \to 0$.
   Thus we have $d_n \mu_{n,k+1} \to \alpha$ as $n\to \infty$ which completes the proof.
\end{proof}

\subsection{Proof of Lemma \ref{prop:ratiosto1}:}

\begin{proof}
    Since $\mu_{n,k} \to 1$ and $j\mapsto \mu_{n,j}$ is nonincreasing,
	we have $\mu_{n,i} \to 1$ for each $i \leq k$. Additionally, since $\lambda_n \to 1$, Corollary \ref{cor:betaprimemun} shows that
	for any $i \in [k]$ 
$
        \lim_{n\to \infty} \frac{\beta_n'(\mu_{n,i})}{d_n \mu_{n,i+1}} = 1.
 $
    As a consequence,  $\beta_n'(\mu_{n,k-1}) \to \infty$ as $n\to \infty$, and for any $j \in [k-2]$
    \begin{equation*}
        \lim_{n\to \infty} \frac{\beta_n'(\mu_{n,j})}{\beta_n'(\mu_{n,j+1})} = \lim_{n\to \infty}\frac{d_n \mu_{n,j+1}}{d_n \mu_{n,j+2}} = \lim_{n\to \infty}\frac{ \mu_{n,j+1}}{ \mu_{n,j+2}} = 1.
    \end{equation*}
    This completes the proof of the lemma.
\end{proof}

\subsection{Proof of Lemma \ref{lem:betanapprox-1}:}

\begin{proof}
	By the first part of Lemma \ref{lem:betan-approx}, \eqref{eq:betan-prime-approx1} is immediate from \eqref{eq:betan-approx1}.
	Now consider \eqref{eq:betan-approx1}.
	 Taking logarithms in \eqref{eq:betan}, for $x> d_n/n$,
	\begin{align*}
		\log \beta_n(x) = \sum_{i=0}^{d_n-1} \left(\log\brR*{x - \frac{i}{n}} - \log\brR*{1 - \frac{i}{n}}\right) = \sum_{i=0}^{d_n-1} \left(\log\brR*{1 - \frac{i}{n} - (1-x)} - \log\brR*{1 - \frac{i}{n}}\right). \nonumber
	\end{align*}
	 Let $\delta_n = \epsilon_n + \frac{d_n}{n}$. For large $n$, $\delta_n \leq \frac{1}{2}$, and hence, 
	 using the expansion $\log(1-h) = -h + O(h^2)$ for $\abs{h} \leq \frac{1}{2}$, for any $x \in [1- \epsilon_n, 1]$:
							\begin{align*}
						\log \beta_n(x)	
							 &= \sum_{i=0}^{d_n-1} 	\brC*{- \frac{i}{n} - (1-x) + \frac{i}{n} + O(\delta_n^2)}
							= -d_n(1-x) + O( d_n \delta_n^2)\\
							& = d_n \log(1 - (1-x)) + O( d_n \delta_n^2)
							 = \log \gamma_n(x) + O(d_n \delta_n^2).
	\end{align*}
	Note that $\delta_n^2  = \brR*{\epsilon_n + d_n/n}^2 \leq 2\brR*{\epsilon_n^2 + \frac{d_n^2}{n^2}}$. Hence by our assumptions  $d_n \delta_n^2 \to 0$. This proves \eqref{eq:betan-approx1} and completes the proof of the lemma.
\end{proof}



\subsection{Proof of Lemma \ref{prop:betan-conv-zero}:}

\begin{proof}
	By \eqref{eq:gammageqbeta}
	\begin{equation*}
		\sup_{x \in [0, 1-\epsilon_n]} \abs{\beta_n(x)} \leq (1 - \epsilon_n)^{d_n} = e^{-d_n \epsilon_n + o(1)} \to 0.
	\end{equation*}
	Similarly, by \eqref{eq:betanprimeub}, under the assumption $\limsup_n \frac{d_n}{n} < 1$, for large $n$,
	\begin{equation*}
		\sup_{x \in [0, 1-\epsilon_n]} \abs{\beta'_n(x)} \leq (1 - d_n/n)^{-1} d_n (1 - \epsilon_n)^{d_n-1} = e^{-d_n \epsilon_n + \log d_n +  O(1)} \to 0.
	\end{equation*}
\end{proof}

\section{Proof of Lemma \ref{lem:negativedrift}}
\label{appendsec:proof-lemma-neg}
	 For a right continuous bounded variation function $F: [0,T] \to \R$, let $dF$ denote the signed measure on $(0,T]$ given by $dF(a,b] = F(b) - F(a)$
	 for $0\le a<b\le T$, and $d\lambda$ denote the Lebesgue measure on $(0,T]$. 
	 Bounded measurable functions $h : [0,T] \to \R$ act on signed measure $d\mu$ on $(0,T]$ on the left as follows: $h d\mu$ denotes the signed measure $A \mapsto \int_A h(x) d\mu(x)$, $A \in \clb(0,T]$. 

         Let $F(t) \doteq \int_0^t f(s) d\lambda(s)$ for $t \in [0,T]$.  Note that $z$ defined in \eqref{eq:ode-blowingup}  is a right continuous function with bounded variations. The corresponding measure $dz$ on $(0,T]$ satisfies the identity
	\begin{align*}
	dz = -f z d\lambda &+ g d\lambda + dM, \nonumber\\
	\shortintertext{namely}
	dz + f z d\lambda &= g d\lambda + dM.
	\end{align*}
	Acting on the left in the above identity by the bounded continuous function $e^{F}(x) \doteq e^{F(x)}$ we get
	$$e^{F} dz + e^{F}f z d\lambda = e^{F} g d\lambda + e^{F} dM. $$
Since $dF = f d\lambda$, by the change of variable formula (cf. \cite[Theorem VI.8.3]{mcshane})
$de^{F} = f e^{F} d\lambda$. Hence
	$$e^{F} dz + z de^{F} = e^{F} g d\lambda + e^{F} dM .$$
       Two applications of the integration by parts formula
	    (cf. \cite{billingsley1995probability}*{Theorem 18.4}) show that
	\begin{align*}
		d\brR*{e^{F}z}  = e^{F} g d\lambda &+ d\brR*{e^{F} M} - M de^{F}. \label{eq:final-form-z1}
	\end{align*}
	Computing the total measure on $(0,t]$ for  $t \leq T$:
	\begin{equation*}
	\begin{aligned}
	e^{F(t)} z(t) - z(0)  = &\int_0^{t} e^{F(s)} g(s) d\lambda(s) 
	 +e^{F(t)}M(t) - M(0) - \int_0^t M(s) de^{F}(s)
	\end{aligned}
	\end{equation*}
        Rearranging terms and multiplying by $e^{-F(t)}$ on both sides:
        \begin{equation}
            \begin{aligned}
                z(t) &= \int_0^{t} e^{F(s) - F(t)} g(s) d\lambda(s) + M(t) 
                      - e^{-F(t)} \int_0^t M(s) de^{F}(s) + e^{-F(t)}(z(0) - M(0)).
            \end{aligned}
            \label{eq:znif}
        \end{equation}
 We  now estimate the various terms on the right hand side of  \eqref{eq:znif}. The first term on the right hand side of \eqref{eq:znif} satisfies for $t \in [0, T \wedge \tau]$
        \begin{align}
            \abs{\int_0^{t} e^{F(s) - F(t)} g(s) d\lambda(s)} &\leq \abs{g}_{*, T \wedge \tau} \int_0^{t} e^{F(s) - F(t)} d\lambda(s)  \nonumber\\
                                                                    &\leq \abs{g}_{*, T \wedge \tau} \int_0^t e^{-\int_s^t f(u) du} d\lambda(s)\nonumber\\
                                                                    &\leq \abs{g}_{*, T \wedge \tau} \int_0^t e^{- m (t-s)} d\lambda(s) =  \abs{g}_{*, T \wedge \tau}\frac{1 - e^{-t m}}{m} \leq \frac{\abs{g}_{*, T \wedge \tau}}{m} \label{eq:boundfirst}.
        \end{align}
        Next we estimate the third term in the right hand side of \eqref{eq:znif}. Since $f$ is non-negative on  $[0, T \wedge \tau]$,  $de^F$ in a positive measure on $(0, T \wedge \tau]$. Hence for $t \in [0, T \wedge \tau]$
        \begin{align}
            \abs{e^{-F(t)} \int_0^t M(s) de^{F}(s)} &\leq \abs{M}_{*, T \wedge \tau} e^{-F(t)} \int_0^{t} de^{F}(s) \leq \abs{M}_{*, T \wedge \tau}.
                                                          \label{eq:boundthird}
        \end{align}
        Finally,  the last term in the right hand side of \eqref{eq:znif} for any $t \in [0, \tau \wedge T]$ can be bounded as
        \begin{equation}
            \abs{e^{-F(t)} \brR*{z(0) - M(0)}} \leq \brR*{\abs{z(0)} + \abs{M(0)}} e^{-F(t)} \leq \brR*{\abs{z(0)} + \abs{M(0)}}e^{-m t}.
            \label{eq:boundlast}
        \end{equation}
         Using \eqref{eq:boundfirst}, \eqref{eq:boundthird} and \eqref{eq:boundlast}  in \eqref{eq:znif} completes the proof of the lemma.
    \hfill \qed
\end{document}